\documentclass[11pt]{amsart}
\usepackage{amssymb,xy,color,doi}
\xyoption{all}
\usepackage{ytableau,hyperref}
\hypersetup{
	colorlinks = true,
	linkcolor = blue,
	anchorcolor = blue,
	citecolor = blue,
	filecolor = blue,
	urlcolor = blue
}
\usepackage{hyperref}
\usepackage{amsmath}
\usepackage{MnSymbol}
\usepackage{tikz}
\usepackage{array, multirow}
\usepackage[margin= 0.8 in]{geometry}
\usepackage{diagbox}
\usetikzlibrary{matrix}

\definecolor{darkred}{rgb}{0.7,0,0}
\newcommand{\defncolor}{\color{darkred}}
\newcommand{\defn}[1]{{\defncolor\emph{#1}}} 

\newcommand{\CC}{\mathbb C}

\newcommand{\s}{<_{std}}

\newlength\scratchlength
\newcommand\filling[2]{
        \settoheight\scratchlength{\mathstrut}%
        \scratchlength=\number\numexpr\number#1-1\relax\scratchlength
        \lower.5\scratchlength\hbox{\scalebox{1}[#1]{$#2$}}%
}

\usepackage[colorinlistoftodos]{todonotes}

\newtheorem{theorem}{Theorem}[section]
\newtheorem{lemma}[theorem]{Lemma}
\newtheorem{proposition}[theorem]{Proposition}
\newtheorem{corollary}[theorem]{Corollary}
\newtheorem{conjecture}[theorem]{Conjecture}

\theoremstyle{definition}
\newtheorem{definition}[theorem]{Definition}
\newtheorem{remark}[theorem]{Remark}
\newtheorem{example}[theorem]{Example}
\newtheorem{question}[theorem]{Question}

\numberwithin{equation}{section}

\title[Immersion poset]{The immersion poset on partitions}

\author[Johnston]{Lisa Johnston}
\address[L. Johnston]{Department of Mathematics, University of California, One Shields
        Avenue, Davis, CA 95616-8633, U.S.A.}
\email{lisjohnston@ucdavis.edu}

\author[Kenepp]{David Kenepp}
\address[D. Kenepp]{Department of Mathematics, University of California, One Shields
        Avenue, Davis, CA 95616-8633, U.S.A.}
\email{dskenepp@ucdavis.edu}

\author[Nguyen]{Evuilynn Nguyen}
\address[E. Nguyen]{Department of Mathematics, University of California, One Shields
        Avenue, Davis, CA 95616-8633, U.S.A.}
\email{evtnguyen@ucdavis.edu}

\author[Paul]{Digjoy Paul}
\address[D. Paul]{Department of Mathematics, Indian Institute of Science, Bangalore 560012, India}
\email{digjoypaul@iisc.ac.in}

\author[Schilling]{Anne Schilling}
\address[A. Schilling]{Department of Mathematics, University of California, One Shields
        Avenue, Davis, CA 95616-8633, U.S.A.}
\email{anne@math.ucdavis.edu}
\urladdr{\href{http://www.math.ucdavis.edu/~anne}{http://www.math.ucdavis.edu/~anne}}

\author[Simone]{Mary Claire Simone}
\address[M. C. Simone]{Department of Mathematics, University of California, One Shields
        Avenue, Davis, CA 95616-8633, U.S.A.}
\email{mcsimone@ucdavis.edu}

\author[Zhou]{Regina Zhou}
\address[R. Zhou]{Department of Mathematics, University of California, One Shields
        Avenue, Davis, CA 95616-8633, U.S.A.}
\email{rgazhou@ucdavis.edu}

\keywords{immersion poset, monomial-positivity, Schur-positivity, characters of symmetric group}

\makeatletter
\@namedef{subjclassname@2020}{%
  \textup{2020} Mathematics Subject Classification}
\makeatother
\subjclass[2020]{06A07, 06A11, 05E05, 05A17, 20C05}

\begin{document}

\begin{abstract}
        We introduce the immersion poset $(\mathcal{P}(n), \leqslant_I)$ on partitions, defined by $\lambda \leqslant_I \mu$ if and only if 
        $s_\mu(x_1, \ldots, x_N) - s_\lambda(x_1, \ldots, x_N)$ is monomial-positive. Relations in the immersion poset determine when irreducible polynomial 
        representations of $GL_N(\mathbb{C})$ form an immersion pair, as defined by Prasad and Raghunathan~\cite{MR4493873}. 
        We develop injections $\mathsf{SSYT}(\lambda, \nu) \hookrightarrow \mathsf{SSYT}(\mu, \nu)$ on semistandard Young tableaux
        given constraints on the shape of $\lambda$, and present results on immersion relations among hook and two column partitions. 
        The standard immersion poset $(\mathcal{P}(n), \leqslant_{std})$ is a refinement of the immersion poset, defined by 
        $\lambda \leqslant_{std} \mu$ if and only if $\lambda \leqslant_D \mu$ in dominance order and 
        $f^\lambda \leqslant f^\mu$, where $f^\nu$ is the number of standard Young tableaux of shape $\nu$. We classify maximal elements of certain 
        shapes in the standard immersion poset using the hook length formula. Finally, we prove Schur-positivity of power sum symmetric functions
        $p_{A_\mu}$ on conjectured lower intervals in the immersion poset, addressing questions posed by Sundaram~\cite{MR3855421}. 
\end{abstract}

\maketitle

\tableofcontents
		
\section{Introduction}

\subsection{Immersion of representations} 
Given two finite-dimensional representations $\pi_1 \colon G \to GL(W_1)$ and $\pi_2 \colon G \to GL(W_2)$ of a group $G$, we say that 
the representation $\pi_1$ is \defn{immersed} in the representation $\pi_2$ if the eigenvalues of $\pi_1(g)$, counting multiplicities, are contained in 
the eigenvalues of $\pi_2(g)$ for all $g\in G$. In this case, we call $(W_1,W_2)$ an \defn{immersion pair} denoted by $W_1 \leqslant_I W_2$. 
Note that, if  $\pi_1$ is a subrepresentation of $\pi_2$, then $W_1 \leqslant_I W_2$, but the converse is not true.

\begin{question}[Prasad and Raghunathan~\cite{MR4493873}]
        Classify immersion of representations  $W_1 \leqslant_I W_2$ for a given group.
\end{question}

Recently, some progress was made on the above problem for symmetric groups~\cite{APV} and alternating groups~\cite{APV.2024}.
In this paper, we study immersion pairs for finite-dimensional irreducible polynomial representations of the general linear group $GL_N(\CC)$.

\subsection{Polynomial representation theory of $GL_N(\CC)$ and symmetric polynomials} 
The polynomial representation theory of $GL_N(\CC)$ was developed by Schur~\cite{Schur.1907} and later popularized by Weyl~\cite{Weyl.1939}
in his expository book on the representation theory of the classical groups. Briefly, the homogeneous irreducible polynomial representations (of degree $n$) 
of $GL_N(\CC)$, also known as \defn{Weyl modules} $W_{\lambda}(\CC^N)$, are indexed by integer partitions $\lambda$ (of size $n$) with at most 
$N$ non-zero parts. The corresponding irreducible characters, known as \defn{Schur polynomials} $s_{\lambda}(x_1,\ldots, x_N)$, are homogeneous 
symmetric polynomials (of degree $n$) in $N$ variables $x_1,\ldots, x_N$. 

\subsection{Monomial positivity} 
Given a partition $\lambda$ of $n$ with at most $N$ parts, the \defn{monomial symmetric polynomials} is
$m_{\lambda}(x_1,\ldots,x_N):=\sum_{\alpha}x_1^{\alpha_1}\cdots x_N^{\alpha_N}$, where the sum is over all distinct permutations $\alpha$
of the parts of $\lambda$. For example, $m_{(2,1)}(x_1,x_2,x_3)=x_1^2x_2+x_1^2x_3+x_2^2x_1+x_2^2x_3+x_3^2x_1+x_3^2x_2.$

The \defn{Schur polynomials} $\{s_{\lambda} \mid \lambda \vdash n\}$ as well as the monomial symmetric polynomials 
$\{m_{\lambda} \mid \lambda \vdash n\}$ form a basis for the vector
space of symmetric polynomials of degree $n$.  A symmetric polynomial $f(x_1,\ldots,x_N)$ is called \defn{monomial-positive} if
\[
	f(x_1,\ldots,x_N)=\sum_{\lambda}c_{\lambda}m_{\lambda}(x_1,\ldots,x_N),
\]
where the coefficients $c_{\lambda}$ are non-negative integers.
\subsection{Immersion of Weyl modules: the immersion poset} 
For a partition $\lambda$ of $n$ with length $\ell(\lambda)\leqslant N$, let 
\[
	\rho_{\lambda}\colon GL_N(\CC) \longrightarrow GL\left(W_{\lambda}(\CC^N)\right)
\]
be the irreducible polynomial representation of degree $n$ of highest weight $\lambda$. It is a known fact that (for example, see \cite[Chapter 7]{EC2}) if $g\in GL_N(\CC)$ has the 
eigenvalues $x_1,\ldots,x_N$, then the eigenvalues of $\rho_{\lambda}(g)$ are the monomials appearing in the Schur polynomial $s_{\lambda}(x_1,\ldots,x_N)$.

Thus, given two partitions $\lambda,\mu$ of $n$ with $\ell(\lambda),\ell(\mu)\leqslant N$, the Weyl module $W_{\lambda}(\CC^N)$ is immersed in 
$W_{\mu}(\CC^N)$  if and only if $s_{\mu}(x_1,\ldots,x_N)-s_{\lambda}(x_1,\ldots,x_N)$ is monomial-positive. 
Hence studying the immersion of Weyl modules is equivalent to studying monomial positivity of the difference of Schur polynomials.

Let $\mathcal{P}(n)$ denote the set of integer partitions of $n$.

\begin{definition}
        \label{definition.immersion}
        We define a partial order on $\mathcal{P}(n)$ as follows. For $\lambda, \mu \in \mathcal{P}(n)$, we define $\lambda \leqslant_I \mu$ if 
        $s_{\mu}(x_1,\ldots,x_N)-s_{\lambda}(x_1,\ldots,x_N)$ is monomial-positive. We call the poset $(\mathcal{P}(n), \leqslant_I)$ the 
        \defn{immersion poset}.
\end{definition}

\subsection{Representation theory of symmetric groups} 
The irreducible representations as well as the conjugacy classes of the symmetric group $S_n$ are indexed by partitions of $n$. Let $\chi^{\lambda}(\mu)$ 
denote the character value of the irreducible character $\chi^{\lambda}$ evaluated at an element of cycle type $\mu$. The character table of $S_n$ is a square matrix encoding character values, whose rows are indexed by irreducible characters $\chi^{\lambda}$ and whose columns are indexed by conjugacy classes $C_{\mu}$.
The character values of $S_n$ are all integers. Solomon~\cite{Solomon.1961} proved that all row sums of the character table of $S_n$ are non-negative 
integers. Finding a combinatorial interpretation of the row sums is still an open problem (see~\cite[Exercise 7.71]{EC2}).

\subsection{Schur-positivity}
A symmetric polynomial $f(x_1,\ldots,x_N)$ of degree $n$ is called \defn{Schur-positive} if
\[
	f(x_1,\ldots,x_N)=\sum_{\lambda \vdash n}c_{\lambda}s_{\lambda}(x_1,\ldots,x_N),
\]
where the coefficients $c_{\lambda}$ are non-negative integers. Schur-positivity is intimately tied to representation theory. Namely, the symmetric function
$f$ is Schur-positive if it is the character of a representation $W$ of $GL_N(\CC)$ which admits the decomposition into irreducibles
$W\cong \underset{\lambda \vdash n}{\bigoplus} W_{\lambda}(\CC^N)^{\oplus c_{\lambda}}.$

The \defn{Frobenius characteristic map} is a bridge between characters of the symmetric group and symmetric polynomials. The irreducible character 
$\chi^{\lambda}$ maps to $s_{\lambda}$ under the Frobenius characteristic map.
Via the Frobenius map, Schur-positivity of $f$ implies that there exists a representation $V$ of $S_n$ such that 
$V\cong \underset{\lambda \vdash n}{\bigoplus} V_{\lambda}^{\oplus {c_{\lambda}}}$, where $V_{\lambda}$ is the irreducible representation of $S_n$
indexed by $\lambda$. 

\subsection{Power sum symmetric polynomials and restricted row sums of character table}
Define the $r$-th power sum symmetric polynomial as 
\[
	p_r(x_1,\ldots,x_N):=\sum_{i=1}^{N}x_i^r.
\]
For a partition $\mu=(\mu_1,\mu_2,\ldots) \vdash n$, define the \defn{power sum symmetric polynomial} as $p_{\mu}:=p_{\mu_1}p_{\mu_2}\cdots$.

Given a subset $A_n$ of partitions of $n$, consider the sum of power sum symmetric polynomials 
\begin{equation}
	\label{char}
        p_{A_n}:=\sum_{\mu \in A_n} p_{\mu}.
\end{equation}

By the Murnaghan--Nakayama rule \cite[Corollary 7.17.4]{EC2}, $p_{\mu}$ can be expressed in the basis of Schur polynomials as 
\[
	p_{\mu}=\sum_{\lambda \vdash n}\chi^{\lambda}(\mu) s_{\lambda}.
\]
Observe that the coefficient of $s_{\lambda}$ in the expansion of $p_{A_n}$ is $\sum_{\mu \in A_n}\chi^{\lambda}(\mu)$. This is precisely the restricted 
row sum (ignoring the columns not in $A_n$) of the character table of $S_n$. 
These values need not always be non-negative integers, that is, $p_{A_n}$ need not be Schur-positive. For example,
if $A_4=\{(1^4), (2,1,1),(4)\}$, then one can deduce from the character table of $S_4$ that $p_{A_4}=3 s_{(4)} +3 s_{(3,1)} +2 s_{(2,2)}+ 3 s_{(2,1,1)}-s_{(1^4)}$ 
is not Schur-positive.

\begin{question}[Sundaram~\cite{MR3855421}]
        \label{question.Sundaram}
        For which choices of $A_n$ is the symmetric polynomial $p_{A_n}$ Schur-positive? 
        In other words, which subsets $A_n$ of columns in the character table of $S_n$ result in non-negative row sums?
\end{question}

In pursuit of Sundaram's question, we explore the immersion poset in detail, hence understanding the immersion of polynomial representations for 
$GL_N(\CC)$. Given a partition $\mu$ of $n$, consider the interval in the immersion poset 
$[(1^n),\mu]:=\{\lambda \mid (1^n) \leqslant_I \lambda \leqslant_I \mu\}$. 
One may ask for what choices of $\mu$, the symmetric polynomial $p_{[(1^n),\mu]}$ defined in Equation~\eqref{char} is Schur-positive.
Assuming Conjectures~\ref{conjecture.interval1} and~\ref{conjecture.interval2}, we prove that:
\begin{enumerate}
\item $p_{[(1^n), (n-2,1,1)]}$ is Schur-positive;
\item $p_{[(1^n), (n-2,2)]}$ is Schur-positive for $n\neq 7$.
\end{enumerate}        

One natural question which arises from the Schur-positivity of the above symmetric functions is to explore the representation theory behind it. 
It would be interesting to construct a natural representation $V$ of the symmetric group such that its character maps to the symmetric polynomial 
$p_{[(1^n),\mu]}$ under the Frobenius map, when $\mu=(n-2,1,1)$ or $(n-2,2)$.

\subsection{Results}

In this paper we analyze various properties of the immersion poset. We begin in Section~\ref{section.standard} by defining the 
standard immersion poset. The relation  $\lambda \leqslant_I \mu$ in the immersion poset for $\lambda,\mu \in \mathcal{P}(n)$ holds if the Kostka numbers
$K_{\lambda,\alpha}\leqslant K_{\mu,\alpha}$ for all $\alpha \in \mathcal{P}(n)$. In the standard immersion poset, one only compares
the number of standard tableaux of shape $\lambda$ and $\mu$ (instead of semistandard tableaux of all content). Relations in the immersion
poset imply relations for the standard immersion poset, but not vice versa. We study properties and maximal elements of the standard immersion.
In particular, maximal elements in the standard immersion poset are also maximal elements in the immersion poset.
In Section~\ref{section.immersion}, we study properties of the immersion poset. In particular, in Section~\ref{section.injections}
we study relations and covers in the immersion poset using explicit injections between sets of semistandard tableaux.
In Section~\ref{section.hooks}, we analyze the immersion poset restricted to partitions of hook shape. In Section~\ref{section.two column},
we analyze the immersion relations on partitions with at most two columns. Finally, in Section~\ref{section.lowerintervals} we conjecture
the structure of certain lower intervals in the immersion poset and prove that $p_{[(1^n), (n-2,1,1)]}$ and $p_{[(1^n), (n-2,2)]}$ ($n\neq 7$)
are Schur-positive. We conclude in Section~\ref{section.discussion} with a discussion of open problems.

\section*{Acknowledgements}
We thank Amritanshu Prasad and Dipendra Prasad for discussions. In particular, we thank Amritanshu Prasad for his probabilistic
questions posed in Section~\ref{section.discussion}.

AS was partially supported by NSF grant DMS--2053350.
AS thanks IPAM for hospitality at the program ``Geometry, Statistical Mechanics, and Integrability'' in Spring 2024, where this work was completed.
		
\section{Standard immersion poset}
\label{section.standard}

In this section, we introduce the standard immersion poset, which is a refinement of the immersion poset. The definition is given in
Section~\ref{section.standard immersion definition}. Basic properties of the standard immersion poset are proved in 
Section~\ref{section.properties standard}. In Section~\ref{section.maximal}, the maximal elements of the standard immersion poset are 
studied. We follow the notational conventions in~\cite[Chapter 6,7]{EC2}.

\subsection{Definition of the standard immersion poset}
\label{section.standard immersion definition}

The \defn{Schur polynomial} $s_\lambda$ for $\lambda \vdash n$ is defined as
\begin{equation}
        \label{equation.Schur}
        s_\lambda(x_1,\ldots,x_N) = \sum_{\mu \vdash n} K_{\lambda,\mu} m_\mu(x_1,\ldots,x_N),
\end{equation}
where $K_{\lambda,\mu}$ are the \defn{Kostka numbers} which count the number of semistandard Young tableaux of shape $\lambda$ and content $\mu$.
Note that with this definition the Schur polynomials are zero unless $N \geqslant \ell(\lambda)$, that is, the number of variables needs to be at least as
large as the number of parts in $\lambda$.

\begin{lemma}
        \label{lemma.immersion K}
        For $\lambda,\mu \in \mathcal{P}(n)$, $\lambda \leqslant_I \mu$ if $K_{\lambda,\alpha}\leqslant K_{\mu,\alpha}$ for all $\alpha \in \mathcal{P}(n)$. 
\end{lemma}

\begin{proof}
        By Definition~\ref{definition.immersion}, two partitions $\lambda,\mu \in \mathcal{P}(n)$ are comparable in the immersion poset $\lambda \leqslant_I \mu$ if
        \[
        s_{\mu}(x_1,\ldots,x_N)-s_{\lambda}(x_1,\ldots,x_N)
        \]
        is monomial-positive. Using~\eqref{equation.Schur}, this can be restated as saying 
        $\lambda \leqslant_I \mu$ if $K_{\lambda,\alpha}\leqslant K_{\mu,\alpha}$ for all $\alpha\in \mathcal{P}(n)$. 
\end{proof}

In particular, Lemma~\ref{lemma.immersion K} implies that a necessary condition for $\lambda \leqslant_I \mu$ is that
$K_{\lambda,(1^n)} \leqslant K_{\mu,(1^n)}$, which count the standard Young tableaux of shape $\lambda$ and $\mu$, respectively.
Note that $f^\lambda:=K_{\lambda,(1^n)}$ is also the dimension of the Specht module $V_{\lambda}$ (the irreducible representation of $S_n$) indexed by $\lambda$. 

Let $\lambda, \mu \in \mathcal{P}(n)$. Define $\lambda \leqslant_D \mu$ in \defn{dominance order} on partitions by requiring that
\[
\sum_{i=1}^k \lambda_i \leqslant \sum_{i=1}^k \mu_i \qquad \text{for all $k\geqslant 1$.}
\]
The Kostka matrix $(K_{\lambda,\alpha})_{\lambda,\alpha\in \mathcal{P}(n)}$ is unit upper-triangular with respect to dominance order, that is, 
$K_{\lambda,\lambda}=1$ and $K_{\lambda,\alpha}=0$ unless $\alpha \leqslant_D \lambda$. This implies another necessary condition for 
$\lambda \leqslant_I \mu$, namely $\lambda \leqslant_D \mu$.
This motivates the definition of the standard immersion poset.

\begin{definition}
        \label{definition.standard immersion}
        On $\mathcal{P}(n)$, define $\lambda \leqslant_{std} \mu$ if $\lambda \leqslant_D \mu$ in dominance order 
        and $f^\lambda \leqslant f^\mu$. We call this poset the \defn{standard immersion poset}.
\end{definition}

As argued above, the standard immersion poset is a refinement of the immersion poset, that is, $\lambda \leqslant_I \mu$ implies that 
$\lambda \leqslant_{std} \mu$. The converse is not always true. For $n \geqslant 12$, there are examples of $\lambda \leqslant_{std} \mu$, which
do not satisfy $\lambda \leqslant_I \mu$. For example $(5,3,1,1,1,1)$ covers $(4, 2, 2, 2, 1, 1)$ in the standard immersion poset for $n=12$, but
not in the immersion poset.

\begin{example}
        The immersion poset for $n=8$ is given in Figure~\ref{figure.immersion8}. It is equal to the standard immersion poset.
\end{example}

\begin{figure}[t]
        \scalebox{0.65}{
                \begin{tikzpicture}[>=latex,line join=bevel,]
                        \node (node_0) at (379.0bp,51.0bp) [draw,draw=none] {${\def\lr#1{\multicolumn{1}{|@{\hspace{.6ex}}c@{\hspace{.6ex}}|}{\raisebox{-.3ex}{$#1$}}}\raisebox{-.6ex}{$\begin{array}[b]{*{1}c}\cline{1-1}\lr{\phantom{x}}\\\cline{1-1}\lr{\phantom{x}}\\\cline{1-1}\lr{\phantom{x}}\\\cline{1-1}\lr{\phantom{x}}\\\cline{1-1}\lr{\phantom{x}}\\\cline{1-1}\lr{\phantom{x}}\\\cline{1-1}\lr{\phantom{x}}\\\cline{1-1}\lr{\phantom{x}}\\\cline{1-1}\end{array}$}}$};
                        \node (node_1) at (339.0bp,183.0bp) [draw,draw=none] {${\def\lr#1{\multicolumn{1}{|@{\hspace{.6ex}}c@{\hspace{.6ex}}|}{\raisebox{-.3ex}{$#1$}}}\raisebox{-.6ex}{$\begin{array}[b]{*{2}c}\cline{1-2}\lr{\phantom{x}}&\lr{\phantom{x}}\\\cline{1-2}\lr{\phantom{x}}\\\cline{1-1}\lr{\phantom{x}}\\\cline{1-1}\lr{\phantom{x}}\\\cline{1-1}\lr{\phantom{x}}\\\cline{1-1}\lr{\phantom{x}}\\\cline{1-1}\lr{\phantom{x}}\\\cline{1-1}\end{array}$}}$};
                        \node (node_21) at (419.0bp,183.0bp) [draw,draw=none] {${\def\lr#1{\multicolumn{1}{|@{\hspace{.6ex}}c@{\hspace{.6ex}}|}{\raisebox{-.3ex}{$#1$}}}\raisebox{-.6ex}{$\begin{array}[b]{*{8}c}\cline{1-8}\lr{\phantom{x}}&\lr{\phantom{x}}&\lr{\phantom{x}}&\lr{\phantom{x}}&\lr{\phantom{x}}&\lr{\phantom{x}}&\lr{\phantom{x}}&\lr{\phantom{x}}\\\cline{1-8}\end{array}$}}$};
                        \node (node_2) at (279.0bp,303.0bp) [draw,draw=none] {${\def\lr#1{\multicolumn{1}{|@{\hspace{.6ex}}c@{\hspace{.6ex}}|}{\raisebox{-.3ex}{$#1$}}}\raisebox{-.6ex}{$\begin{array}[b]{*{2}c}\cline{1-2}\lr{\phantom{x}}&\lr{\phantom{x}}\\\cline{1-2}\lr{\phantom{x}}&\lr{\phantom{x}}\\\cline{1-2}\lr{\phantom{x}}\\\cline{1-1}\lr{\phantom{x}}\\\cline{1-1}\lr{\phantom{x}}\\\cline{1-1}\lr{\phantom{x}}\\\cline{1-1}\end{array}$}}$};
                        \node (node_7) at (357.0bp,417.0bp) [draw,draw=none] {${\def\lr#1{\multicolumn{1}{|@{\hspace{.6ex}}c@{\hspace{.6ex}}|}{\raisebox{-.3ex}{$#1$}}}\raisebox{-.6ex}{$\begin{array}[b]{*{2}c}\cline{1-2}\lr{\phantom{x}}&\lr{\phantom{x}}\\\cline{1-2}\lr{\phantom{x}}&\lr{\phantom{x}}\\\cline{1-2}\lr{\phantom{x}}&\lr{\phantom{x}}\\\cline{1-2}\lr{\phantom{x}}&\lr{\phantom{x}}\\\cline{1-2}\end{array}$}}$};
                        \node (node_20) at (409.0bp,303.0bp) [draw,draw=none] {${\def\lr#1{\multicolumn{1}{|@{\hspace{.6ex}}c@{\hspace{.6ex}}|}{\raisebox{-.3ex}{$#1$}}}\raisebox{-.6ex}{$\begin{array}[b]{*{7}c}\cline{1-7}\lr{\phantom{x}}&\lr{\phantom{x}}&\lr{\phantom{x}}&\lr{\phantom{x}}&\lr{\phantom{x}}&\lr{\phantom{x}}&\lr{\phantom{x}}\\\cline{1-7}\lr{\phantom{x}}\\\cline{1-1}\end{array}$}}$};
                        \node (node_3) at (162.0bp,417.0bp) [draw,draw=none] {${\def\lr#1{\multicolumn{1}{|@{\hspace{.6ex}}c@{\hspace{.6ex}}|}{\raisebox{-.3ex}{$#1$}}}\raisebox{-.6ex}{$\begin{array}[b]{*{2}c}\cline{1-2}\lr{\phantom{x}}&\lr{\phantom{x}}\\\cline{1-2}\lr{\phantom{x}}&\lr{\phantom{x}}\\\cline{1-2}\lr{\phantom{x}}&\lr{\phantom{x}}\\\cline{1-2}\lr{\phantom{x}}\\\cline{1-1}\lr{\phantom{x}}\\\cline{1-1}\end{array}$}}$};
                        \node (node_4) at (244.0bp,417.0bp) [draw,draw=none] {${\def\lr#1{\multicolumn{1}{|@{\hspace{.6ex}}c@{\hspace{.6ex}}|}{\raisebox{-.3ex}{$#1$}}}\raisebox{-.6ex}{$\begin{array}[b]{*{3}c}\cline{1-3}\lr{\phantom{x}}&\lr{\phantom{x}}&\lr{\phantom{x}}\\\cline{1-3}\lr{\phantom{x}}\\\cline{1-1}\lr{\phantom{x}}\\\cline{1-1}\lr{\phantom{x}}\\\cline{1-1}\lr{\phantom{x}}\\\cline{1-1}\lr{\phantom{x}}\\\cline{1-1}\end{array}$}}$};
                        \node (node_16) at (541.0bp,621.0bp) [draw,draw=none] {${\def\lr#1{\multicolumn{1}{|@{\hspace{.6ex}}c@{\hspace{.6ex}}|}{\raisebox{-.3ex}{$#1$}}}\raisebox{-.6ex}{$\begin{array}[b]{*{6}c}\cline{1-6}\lr{\phantom{x}}&\lr{\phantom{x}}&\lr{\phantom{x}}&\lr{\phantom{x}}&\lr{\phantom{x}}&\lr{\phantom{x}}\\\cline{1-6}\lr{\phantom{x}}&\lr{\phantom{x}}\\\cline{1-2}\end{array}$}}$};
                        \node (node_5) at (20.0bp,525.0bp) [draw,draw=none] {${\def\lr#1{\multicolumn{1}{|@{\hspace{.6ex}}c@{\hspace{.6ex}}|}{\raisebox{-.3ex}{$#1$}}}\raisebox{-.6ex}{$\begin{array}[b]{*{3}c}\cline{1-3}\lr{\phantom{x}}&\lr{\phantom{x}}&\lr{\phantom{x}}\\\cline{1-3}\lr{\phantom{x}}&\lr{\phantom{x}}\\\cline{1-2}\lr{\phantom{x}}\\\cline{1-1}\lr{\phantom{x}}\\\cline{1-1}\lr{\phantom{x}}\\\cline{1-1}\end{array}$}}$};
                        \node (node_6) at (84.0bp,525.0bp) [draw,draw=none] {${\def\lr#1{\multicolumn{1}{|@{\hspace{.6ex}}c@{\hspace{.6ex}}|}{\raisebox{-.3ex}{$#1$}}}\raisebox{-.6ex}{$\begin{array}[b]{*{4}c}\cline{1-4}\lr{\phantom{x}}&\lr{\phantom{x}}&\lr{\phantom{x}}&\lr{\phantom{x}}\\\cline{1-4}\lr{\phantom{x}}\\\cline{1-1}\lr{\phantom{x}}\\\cline{1-1}\lr{\phantom{x}}\\\cline{1-1}\lr{\phantom{x}}\\\cline{1-1}\end{array}$}}$};
                        \node (node_9) at (186.0bp,525.0bp) [draw,draw=none] {${\def\lr#1{\multicolumn{1}{|@{\hspace{.6ex}}c@{\hspace{.6ex}}|}{\raisebox{-.3ex}{$#1$}}}\raisebox{-.6ex}{$\begin{array}[b]{*{3}c}\cline{1-3}\lr{\phantom{x}}&\lr{\phantom{x}}&\lr{\phantom{x}}\\\cline{1-3}\lr{\phantom{x}}&\lr{\phantom{x}}&\lr{\phantom{x}}\\\cline{1-3}\lr{\phantom{x}}\\\cline{1-1}\lr{\phantom{x}}\\\cline{1-1}\end{array}$}}$};
                        \node (node_11) at (244.0bp,525.0bp) [draw,draw=none] {${\def\lr#1{\multicolumn{1}{|@{\hspace{.6ex}}c@{\hspace{.6ex}}|}{\raisebox{-.3ex}{$#1$}}}\raisebox{-.6ex}{$\begin{array}[b]{*{3}c}\cline{1-3}\lr{\phantom{x}}&\lr{\phantom{x}}&\lr{\phantom{x}}\\\cline{1-3}\lr{\phantom{x}}&\lr{\phantom{x}}&\lr{\phantom{x}}\\\cline{1-3}\lr{\phantom{x}}&\lr{\phantom{x}}\\\cline{1-2}\end{array}$}}$};
                        \node (node_15) at (341.0bp,621.0bp) [draw,draw=none] {${\def\lr#1{\multicolumn{1}{|@{\hspace{.6ex}}c@{\hspace{.6ex}}|}{\raisebox{-.3ex}{$#1$}}}\raisebox{-.6ex}{$\begin{array}[b]{*{5}c}\cline{1-5}\lr{\phantom{x}}&\lr{\phantom{x}}&\lr{\phantom{x}}&\lr{\phantom{x}}&\lr{\phantom{x}}\\\cline{1-5}\lr{\phantom{x}}&\lr{\phantom{x}}&\lr{\phantom{x}}\\\cline{1-3}\end{array}$}}$};
                        \node (node_19) at (357.0bp,525.0bp) [draw,draw=none] {${\def\lr#1{\multicolumn{1}{|@{\hspace{.6ex}}c@{\hspace{.6ex}}|}{\raisebox{-.3ex}{$#1$}}}\raisebox{-.6ex}{$\begin{array}[b]{*{6}c}\cline{1-6}\lr{\phantom{x}}&\lr{\phantom{x}}&\lr{\phantom{x}}&\lr{\phantom{x}}&\lr{\phantom{x}}&\lr{\phantom{x}}\\\cline{1-6}\lr{\phantom{x}}\\\cline{1-1}\lr{\phantom{x}}\\\cline{1-1}\end{array}$}}$};
                        \node (node_8) at (84.0bp,621.0bp) [draw,draw=none] {${\def\lr#1{\multicolumn{1}{|@{\hspace{.6ex}}c@{\hspace{.6ex}}|}{\raisebox{-.3ex}{$#1$}}}\raisebox{-.6ex}{$\begin{array}[b]{*{3}c}\cline{1-3}\lr{\phantom{x}}&\lr{\phantom{x}}&\lr{\phantom{x}}\\\cline{1-3}\lr{\phantom{x}}&\lr{\phantom{x}}\\\cline{1-2}\lr{\phantom{x}}&\lr{\phantom{x}}\\\cline{1-2}\lr{\phantom{x}}\\\cline{1-1}\end{array}$}}$};
                        \node (node_18) at (186.0bp,711.0bp) [draw,draw=none] {${\def\lr#1{\multicolumn{1}{|@{\hspace{.6ex}}c@{\hspace{.6ex}}|}{\raisebox{-.3ex}{$#1$}}}\raisebox{-.6ex}{$\begin{array}[b]{*{5}c}\cline{1-5}\lr{\phantom{x}}&\lr{\phantom{x}}&\lr{\phantom{x}}&\lr{\phantom{x}}&\lr{\phantom{x}}\\\cline{1-5}\lr{\phantom{x}}&\lr{\phantom{x}}\\\cline{1-2}\lr{\phantom{x}}\\\cline{1-1}\end{array}$}}$};
                        \node (node_10) at (36.0bp,711.0bp) [draw,draw=none] {${\def\lr#1{\multicolumn{1}{|@{\hspace{.6ex}}c@{\hspace{.6ex}}|}{\raisebox{-.3ex}{$#1$}}}\raisebox{-.6ex}{$\begin{array}[b]{*{4}c}\cline{1-4}\lr{\phantom{x}}&\lr{\phantom{x}}&\lr{\phantom{x}}&\lr{\phantom{x}}\\\cline{1-4}\lr{\phantom{x}}&\lr{\phantom{x}}\\\cline{1-2}\lr{\phantom{x}}\\\cline{1-1}\lr{\phantom{x}}\\\cline{1-1}\end{array}$}}$};
                        \node (node_12) at (186.0bp,621.0bp) [draw,draw=none] {${\def\lr#1{\multicolumn{1}{|@{\hspace{.6ex}}c@{\hspace{.6ex}}|}{\raisebox{-.3ex}{$#1$}}}\raisebox{-.6ex}{$\begin{array}[b]{*{4}c}\cline{1-4}\lr{\phantom{x}}&\lr{\phantom{x}}&\lr{\phantom{x}}&\lr{\phantom{x}}\\\cline{1-4}\lr{\phantom{x}}&\lr{\phantom{x}}\\\cline{1-2}\lr{\phantom{x}}&\lr{\phantom{x}}\\\cline{1-2}\end{array}$}}$};
                        \node (node_17) at (261.0bp,621.0bp) [draw,draw=none] {${\def\lr#1{\multicolumn{1}{|@{\hspace{.6ex}}c@{\hspace{.6ex}}|}{\raisebox{-.3ex}{$#1$}}}\raisebox{-.6ex}{$\begin{array}[b]{*{5}c}\cline{1-5}\lr{\phantom{x}}&\lr{\phantom{x}}&\lr{\phantom{x}}&\lr{\phantom{x}}&\lr{\phantom{x}}\\\cline{1-5}\lr{\phantom{x}}\\\cline{1-1}\lr{\phantom{x}}\\\cline{1-1}\lr{\phantom{x}}\\\cline{1-1}\end{array}$}}$};
                        \node (node_14) at (514.0bp,525.0bp) [draw,draw=none] {${\def\lr#1{\multicolumn{1}{|@{\hspace{.6ex}}c@{\hspace{.6ex}}|}{\raisebox{-.3ex}{$#1$}}}\raisebox{-.6ex}{$\begin{array}[b]{*{4}c}\cline{1-4}\lr{\phantom{x}}&\lr{\phantom{x}}&\lr{\phantom{x}}&\lr{\phantom{x}}\\\cline{1-4}\lr{\phantom{x}}&\lr{\phantom{x}}&\lr{\phantom{x}}&\lr{\phantom{x}}\\\cline{1-4}\end{array}$}}$};
                        \node (node_13) at (108.0bp,711.0bp) [draw,draw=none] {${\def\lr#1{\multicolumn{1}{|@{\hspace{.6ex}}c@{\hspace{.6ex}}|}{\raisebox{-.3ex}{$#1$}}}\raisebox{-.6ex}{$\begin{array}[b]{*{4}c}\cline{1-4}\lr{\phantom{x}}&\lr{\phantom{x}}&\lr{\phantom{x}}&\lr{\phantom{x}}\\\cline{1-4}\lr{\phantom{x}}&\lr{\phantom{x}}&\lr{\phantom{x}}\\\cline{1-3}\lr{\phantom{x}}\\\cline{1-1}\end{array}$}}$};
                        \draw [black,->] (node_0) ..controls (365.69bp,95.271bp) and (360.5bp,112.11bp)  .. (node_1);
                        \draw [black,->] (node_0) ..controls (395.85bp,106.76bp) and (406.8bp,142.34bp)  .. (node_21);
                        \draw [black,->] (node_1) ..controls (316.57bp,228.1bp) and (306.7bp,247.52bp)  .. (node_2);
                        \draw [black,->] (node_1) ..controls (345.84bp,272.21bp) and (351.02bp,338.98bp)  .. (node_7);
                        \draw [black,->] (node_1) ..controls (356.25bp,217.76bp) and (359.15bp,223.07bp)  .. (362.0bp,228.0bp) .. controls (371.99bp,245.3bp) and (384.09bp,264.31bp)  .. (node_20);
                        \draw [black,->] (node_2) ..controls (243.96bp,337.54bp) and (207.57bp,372.38bp)  .. (node_3);
                        \draw [black,->] (node_2) ..controls (264.43bp,350.61bp) and (261.63bp,359.58bp)  .. (node_4);
                        \draw [black,->] (node_2) ..controls (345.46bp,339.73bp) and (520.25bp,436.94bp)  .. (549.0bp,492.0bp) .. controls (566.02bp,524.58bp) and (557.76bp,568.37bp)  .. (node_16);
                        \draw [black,->] (node_3) ..controls (124.87bp,438.57bp) and (81.445bp,464.27bp)  .. (49.0bp,492.0bp) .. controls (48.537bp,492.4bp) and (48.074bp,492.8bp)  .. (node_5);
                        \draw [black,->] (node_3) ..controls (137.95bp,450.69bp) and (125.43bp,467.69bp)  .. (node_6);
                        \draw [black,->] (node_3) ..controls (171.96bp,461.99bp) and (175.03bp,475.57bp)  .. (node_9);
                        \draw [black,->] (node_3) ..controls (189.2bp,453.17bp) and (207.45bp,476.75bp)  .. (node_11);
                        \draw [black,->] (node_3) ..controls (140.85bp,481.19bp) and (133.28bp,527.62bp)  .. (157.0bp,558.0bp) .. controls (197.6bp,610.0bp) and (240.52bp,567.66bp)  .. (301.0bp,594.0bp) .. controls (304.93bp,595.71bp) and (308.91bp,597.79bp)  .. (node_15);
                        \draw [black,->] (node_4) ..controls (212.87bp,440.94bp) and (198.71bp,449.95bp)  .. (185.0bp,456.0bp) .. controls (127.79bp,481.24bp) and (102.41bp,459.5bp)  .. (49.0bp,492.0bp) .. controls (48.808bp,492.12bp) and (48.616bp,492.24bp)  .. (node_5);
                        \draw [black,->] (node_4) ..controls (197.22bp,448.99bp) and (150.36bp,480.04bp)  .. (node_6);
                        \draw [black,->] (node_4) ..controls (217.79bp,465.89bp) and (211.13bp,478.07bp)  .. (node_9);
                        \draw [black,->] (node_4) ..controls (244.0bp,468.44bp) and (244.0bp,481.73bp)  .. (node_11);
                        \draw [black,->] (node_4) ..controls (275.61bp,484.22bp) and (294.41bp,523.52bp)  .. (311.0bp,558.0bp) .. controls (317.12bp,570.71bp) and (323.97bp,584.87bp)  .. (node_15);
                        \draw [black,->] (node_4) ..controls (281.88bp,453.53bp) and (307.77bp,477.82bp)  .. (node_19);
                        \draw [black,->] (node_5) ..controls (46.539bp,564.98bp) and (53.761bp,575.59bp)  .. (node_8);
                        \draw [black,->] (node_5) ..controls (42.947bp,552.8bp) and (45.953bp,555.57bp)  .. (49.0bp,558.0bp) .. controls (74.51bp,578.35bp) and (89.361bp,571.5bp)  .. (113.0bp,594.0bp) .. controls (127.06bp,607.38bp) and (152.84bp,650.81bp)  .. (node_18);
                        \draw [black,->] (node_6) ..controls (43.178bp,565.31bp) and (25.925bp,583.87bp)  .. (22.0bp,594.0bp) .. controls (12.105bp,619.54bp) and (16.542bp,650.41bp)  .. (node_10);
                        \draw [black,->] (node_6) ..controls (123.91bp,562.78bp) and (141.57bp,579.05bp)  .. (node_12);
                        \draw [black,->] (node_6) ..controls (114.39bp,551.99bp) and (119.21bp,555.28bp)  .. (124.0bp,558.0bp) .. controls (163.62bp,580.51bp) and (179.21bp,573.69bp)  .. (220.0bp,594.0bp) .. controls (220.2bp,594.1bp) and (220.4bp,594.2bp)  .. (node_17);
                        \draw [black,->] (node_7) ..controls (394.85bp,463.85bp) and (433.27bp,521.92bp)  .. (402.0bp,558.0bp) .. controls (365.09bp,600.58bp) and (205.51bp,579.71bp)  .. (151.0bp,594.0bp) .. controls (138.49bp,597.28bp) and (125.13bp,602.16bp)  .. (node_8);
                        \draw [black,->] (node_7) ..controls (310.56bp,439.76bp) and (236.89bp,475.1bp)  .. (215.0bp,492.0bp) .. controls (214.52bp,492.37bp) and (214.04bp,492.75bp)  .. (node_9);
                        \draw [black,->] (node_7) ..controls (324.35bp,448.62bp) and (293.89bp,477.2bp)  .. (node_11);
                        \draw [black,->] (node_7) ..controls (394.27bp,439.92bp) and (438.38bp,466.61bp)  .. (474.0bp,492.0bp) .. controls (478.71bp,495.36bp) and (483.6bp,499.07bp)  .. (node_14);
                        \draw [black,->] (node_7) ..controls (390.81bp,438.33bp) and (422.92bp,461.83bp)  .. (436.0bp,492.0bp) .. controls (447.67bp,518.91bp) and (453.87bp,534.74bp)  .. (436.0bp,558.0bp) .. controls (400.02bp,604.84bp) and (361.68bp,572.31bp)  .. (node_17);
                        \draw [black,->] (node_7) ..controls (357.0bp,459.05bp) and (357.0bp,477.7bp)  .. (node_19);
                        \draw [black,->] (node_8) ..controls (65.048bp,656.74bp) and (59.947bp,666.1bp)  .. (node_10);
                        \draw [black,->] (node_8) ..controls (93.893bp,658.28bp) and (96.972bp,669.56bp)  .. (node_13);
                        \draw [black,->] (node_9) ..controls (130.84bp,551.46bp) and (62.143bp,584.14bp)  .. (55.0bp,594.0bp) .. controls (38.338bp,616.99bp) and (34.175bp,649.25bp)  .. (node_10);
                        \draw [black,->] (node_9) ..controls (186.0bp,563.7bp) and (186.0bp,577.3bp)  .. (node_12);
                        \draw [black,->] (node_11) ..controls (223.24bp,559.64bp) and (213.03bp,576.2bp)  .. (node_12);
                        \draw [black,->] (node_12) ..controls (157.22bp,654.47bp) and (144.25bp,669.1bp)  .. (node_13);
                        \draw [black,->] (node_12) ..controls (186.0bp,653.7bp) and (186.0bp,667.18bp)  .. (node_18);
                        \draw [black,->] (node_14) ..controls (490.55bp,546.45bp) and (482.16bp,552.93bp)  .. (474.0bp,558.0bp) .. controls (444.5bp,576.31bp) and (408.73bp,592.69bp)  .. (node_15);
                        \draw [black,->] (node_14) ..controls (522.42bp,555.31bp) and (528.94bp,578.03bp)  .. (node_16);
                        \draw [black,->] (node_17) ..controls (229.56bp,658.89bp) and (219.26bp,670.98bp)  .. (node_18);
                \end{tikzpicture}
        }
        \caption{The (standard) immersion poset for $n=8$.
                \label{figure.immersion8}}
\end{figure}
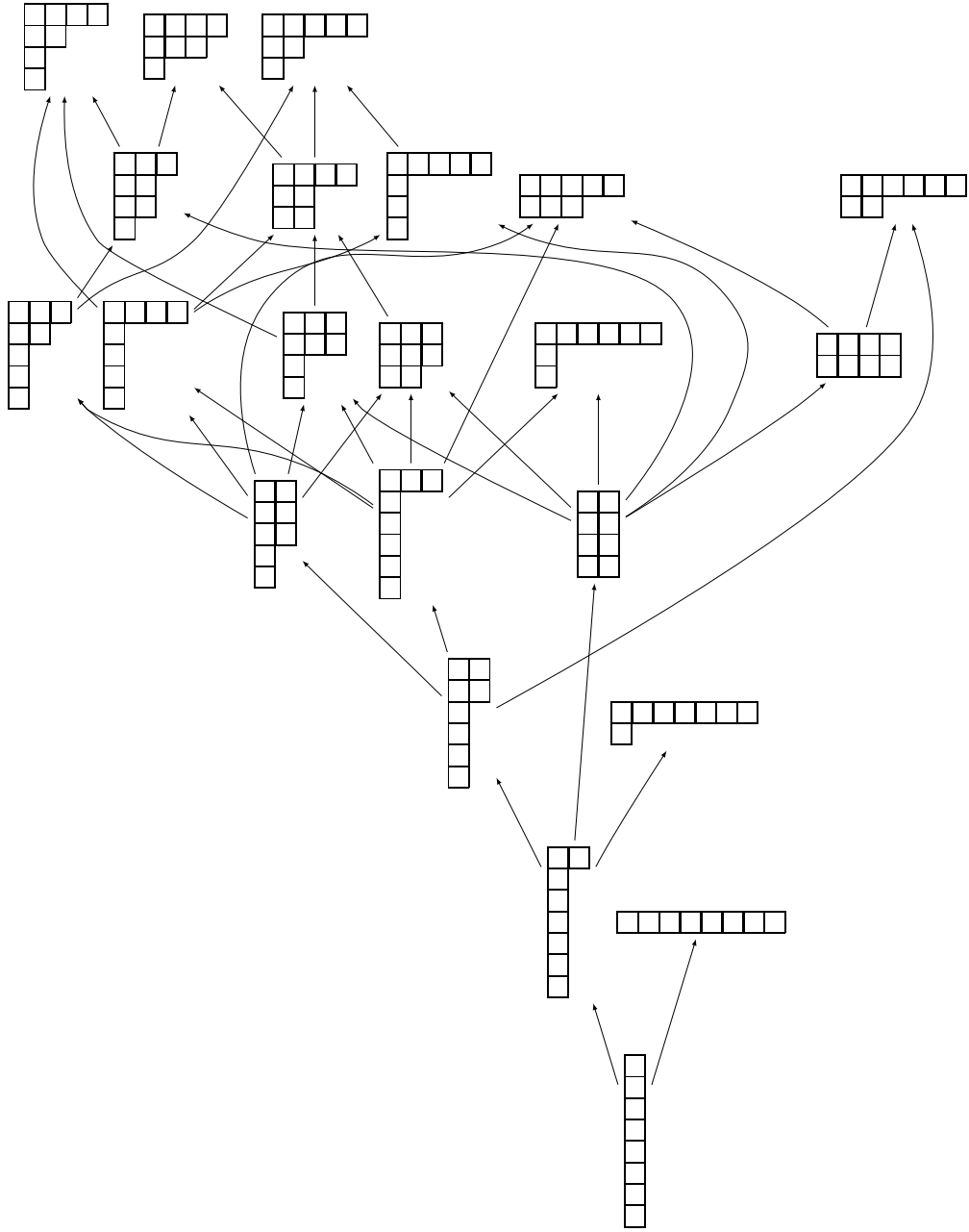

\subsection{Properties of the standard immersion poset}
\label{section.properties standard}

We now state and prove properties of the standard immersion poset. Our main tool is the \defn{hook length formula} for $\lambda \in \mathcal{P}(n)$
\begin{equation}
        \label{equation.hook length}
        f^\lambda = \frac{n!}{\prod_{u \in \lambda} h(u)},
\end{equation}
where $h(u)$ is the hook length of the cell $u$ in $\lambda$ which counts the cells weakly to the right of $u$ and strictly below $u$ (in English notation for 
partitions).

We write $\lambda \lessdot_{std} \mu$ if $\mu$ covers $\lambda$ in the standard immersion poset. More precisely, $\lambda \lessdot_{std} \mu$ if
$\lambda <_{std} \mu$ and there does not exist any $\nu$ such that $\lambda <_{std} \nu <_{std} \mu$.

\begin{lemma}
        \label{lemma.minimal}
        The partition $(1^n)$ is the unique minimal element in the standard immersion poset.
\end{lemma}

\begin{proof}
        The partition $(1^n)$ is the unique minimal element in dominance order. Furthermore, $f^{(1^n)}=1\leqslant f^\lambda$ for all $\lambda \in \mathcal{P}(n)$.
        This proves the claim.
\end{proof}

\begin{lemma}
        \label{lemma.col}
        We have
        \begin{enumerate}
                \item $(1^n) \lessdot_{std} (n)$ for all $n$ and
                \item $(2,1^{n-2}) \lessdot_{std} (n-1,1)$ for all $n \geqslant 3$.
        \end{enumerate}
\end{lemma}

\begin{proof}
        We have $(1^n) <_D (n)$ and $f^{(1^n)} = f^{(n)}=1$. There is no other partition $\lambda$ with $f^{\lambda}=1$. This implies $(1^n) \lessdot_{std} (n)$.
        Similarly, $(2,1^{n-2}) <_D (n-1,1)$ and $f^{(2,1^{n-2})} = f^{(n-1,1)}=n-1$. There is no other partition $\lambda$ with $f^{\lambda}=n-1$. 
        This implies $(2,1^{n-2}) \lessdot_{std} (n-1,1)$.
\end{proof}

\begin{remark}
        \label{remark.standard immersion}
        \mbox{}
        \begin{enumerate}
                \item Let $\lambda \s \mu$. If $\mu$ covers $\lambda$ in dominance order, then $\mu$ covers $\lambda$ with respect to $\s$. The converse is not true. 
                Take $\lambda=(1^n)$ and $\mu=(n)$. 
                \item For a given partition $\lambda$ with transpose $\lambda^t$, if $\lambda <_D \lambda^t$, then $\lambda \s \lambda^t$ as both representations have the same dimension, that
                is, $f^\lambda = f^{\lambda^t}$. In general, $\lambda^t$ does not cover $\lambda$. 
        \end{enumerate}
\end{remark}

Given a partition $\lambda$ such that $\lambda <_D \lambda^t$, it would be interesting to find all partitions $\lambda <_D \nu <_D \lambda^t$ 
satisfying $f^{\lambda}=f^{\nu}$. This would help to understand when the transpose of $\lambda$ covers $\lambda$ in the immersion poset.

\begin{lemma}
        \label{two col}
        Let $\lambda=(2^{a},1^b)$ and $\mu=(2^{a+1},1^{b -2})$. Then  $\lambda \lessdot_{std} \mu$ if and only if $\frac{b(b-1)}{2}>a$.
\end{lemma}

\begin{proof}
        We have $\lambda \lessdot_D \mu$. Hence by Remark~\ref{remark.standard immersion}(1), it suffices to show that $\lambda \s \mu$.
        By the hook length formula, this is true if $\frac{f^{\lambda}}{f^{\mu}}=\frac{(b+1)(a+1)}{(b-1)(a+b+1)} \leqslant 1$, which is equivalent to the condition
        $\frac{b(b-1)}{2}>a$.
\end{proof}

\subsection{Classifying maximal elements}
\label{section.maximal}

In this section, we study the maximal elements of the standard immersion poset. Recall that the standard immersion poset is a refinement of the immersion 
poset. This implies that if a partition is maximal in the standard immersion poset, then it is also maximal in the immersion poset. 

\begin{proposition}
        \label{prop.max_two_row}
        The partition $(a+b,a)$ is a maximal element in the standard immersion poset if and only if $\frac{b(b+3)}{2} \geqslant a$.
\end{proposition}

\begin{proof}
        Let $\lambda = (a+b,a)$. Any partition $\nu$ which dominates $\lambda$, that is, $\nu>_D\lambda$, must have the form $\nu=\nu^{(i)}
        =(a+b+i,a-i)$ for some $i \geqslant 1$. Note that $\frac{f^{\nu^{(1)}}}{f^{\lambda}}=\frac{a(b+3)}{(b+1)(a+b+2)}$. Hence $f^{\nu^{(1)}} <f^{\lambda}$ 
        if and only if $\frac{b(b+3)}{a-1}> 2$ (which is equivalent to $\frac{b(b+3)}{2} \geqslant a$). Thus, the condition is necessary.

        To prove that the condition $\frac{b(b+3)}{a-1}> 2$ is sufficient, note that
        \[
        \frac{f^{\nu^{(i+1)}}}{f^{\nu^{(i)}}}=\frac{(a-i)(b+3+2i)}{(b+1+2i)(a+b+i+2)}< \frac{a(b+3)}{(b+1)(a+b+2)}=\frac{f^{\nu^{(1)}}}{f^{\lambda}}.
        \]
        Since $f^{\nu^{(1)}}<f^{\lambda}$ when $\frac{b(b+3)}{a-1}> 2$, we must have  $\frac{f^{\nu^{(i+1)}}}{f^{\nu^{(i)}}}<1$. This is true for each $i$. 
        Hence $\lambda$ is a maximal element.
\end{proof}

\begin{proposition}
        \label{prop.max_3r1}
        Let $\lambda = (a+b, a, 1)$ where $a\geqslant 2.$ Then $\lambda$ is maximal in the standard immersion poset if and only if 
        $a \leqslant \frac{(b+1)(b+2)}{2}.$
\end{proposition}

\begin{proof}
        We first prove the reverse direction by inducting on $a$. For our base case, let $a=2$ and $2 \leqslant \frac{(b+1)(b+2)}{2}$. To prove that 
        $\lambda=(2+b,2,1)$ is maximal, we show that there exists no partition $\nu$ such that $\lambda <_D \nu $ and $f^\lambda < f^\nu$. We start by 
        classifying all partitions $\nu$ such that $\lambda <_D \nu$. It is known that $\lambda \lessdot_D \nu$ if and only if the Young diagram of $\nu$ can 
        be obtained from the Young diagram of $\lambda$ by moving a single box in row $k$ to row $k-1$ or by moving a single box in column $k$ to column $k+1$. 
        This means that the partition $(2+b,2,1)$ has exactly two covers: $(2+b,3)$ and $(3+b,1^2)$. The former is obtained by moving the box in row 3 to row 2, 
        and the latter is obtained by moving the box at the end of row 2 to row 1. Furthermore, $(2+b,3)$ and $(3+b,1^2)$ are only covered by $(3+b,2)$. 
        Below is the Hasse diagram in dominance order summarizing the specific covering relations:
        
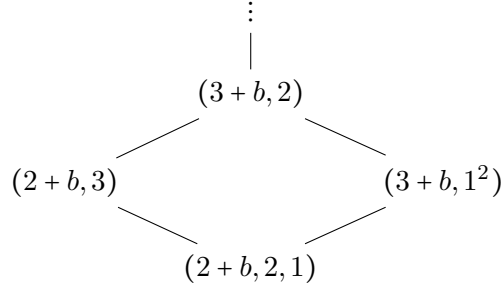
\begin{center}
                \begin{tikzpicture}
                        \matrix (a) [matrix of math nodes, column sep=0.6cm, row sep=0.5cm]{

                                & \vdots \\
                                & (3+b,2) &\\
                                (2+b,3) && (3+b,1^2) \\
                                &(2+b,2,1) &\\};

                        \foreach \i/\j in {3-1/4-2, 3-3/4-2, 2-2/3-1, 2-2/3-3, 1-2/2-2}
                        \draw (a-\i) -- (a-\j);
                \end{tikzpicture}
        \end{center}
        Let $\nu$ be any partition such that $\lambda <_D \nu$. 
        By our covering relations, we have that either $\nu=(2+b,3)$ or $\nu$ is contained in some chain $\lambda <_D (3+b,1^2) <_D \dots <_D \nu$.

        Now, we will show that $f^\lambda > f^\nu$ for all $\nu$ such that $\lambda <_D \nu$. 

        By Proposition~\ref{maxhook}, we know that $(3+b,1^2)$ is maximal in the standard immersion poset. That is, if $(3+b,1^2) <_D \nu$ then 
        $f^{(3+b,1^2)} > f^\nu$. Note that our assumption that $2 \leqslant \frac{(b+1)(b+2)}{2}$ implies $b \geqslant 1$. By this fact and the hook length formula,  
        \[
        \frac{f^{(2+b,3)}}{f^\lambda}=\frac{b(b+4)}{2(b+1)(b+3)}<1 \qquad \text{and} \qquad
        \frac{f^{(3+b,1^2)}}{f^\lambda}=\frac{3(b+4)}{2(b+1)(b+5)} <1. 
        \]
        Since $f^\lambda > f^{(2+b,3)}$ and $f^\lambda > f^{(3+b,1^2)} > f^\nu$, we have shown that $f^\lambda > f^\nu$ for all $\nu$ such that $\lambda <_D \nu$.

        Now, let $\lambda=(a+b,a,1)$ where $a\leqslant \frac{(b+1)(b+2)}{2}$ and suppose that for some $a \geqslant 2$, the partition $(c+b,c,1)$ is maximal 
        when $c<a\leqslant \frac{(b+1)(b+2)}{2}$. We follow a similar argument as the base case and show that $f^\lambda > f^\nu$ for $\lambda \lessdot_D \nu$. 
        Observe that the Hasse diagram in dominance order around $\lambda$ looks as follows:
        
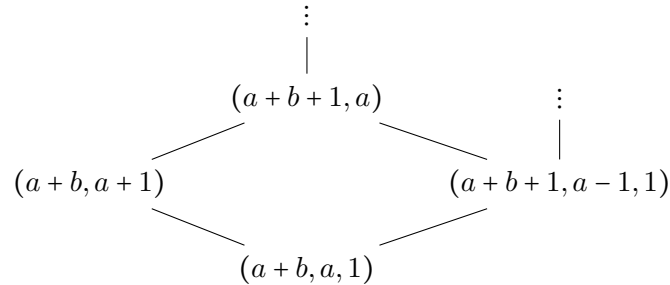
\begin{center}
                \begin{tikzpicture}
                        \matrix (a) [matrix of math nodes, column sep=0.6cm, row sep=0.5cm]{

                                & \vdots \\
                                & (a+b+1,a) & \vdots \\
                                (a+b,a+1) && (a+b+1,a-1,1) \\
                                &(a+b,a,1) &\\};

                        \foreach \i/\j in {3-1/4-2, 3-3/4-2, 2-2/3-1, 2-2/3-3, 2-3/3-3, 1-2/2-2}
                        \draw (a-\i) -- (a-\j);
                \end{tikzpicture}
        \end{center}
        We first consider the partition $(a+b, a+1).$ Then by the hook length formula,
        \[
        \frac{f^{(a+b,a+1)}}{f^\lambda} = \frac{(a+b+2)(b)}{(a+b+1)(a)(b+1)} < \frac{(a+b+2)}{(a+b+1)(a)} < 1
        \]
        where the last inequality follows since $a \geqslant 2.$

        Next, consider the partition $(a+b+1,a-1,1)$. By our inductive hypothesis, $(a+b+1,a-1,1)=((a-1)+(b+2),a-1,1)$ is maximal since 
        $a-1<a\leqslant \frac{(b+1)(b+2)}{2}<\frac{(b+3)(b+4)}{2}$. Suppose that $a$ is the upper bound of our inequality $a\leqslant \frac{(b+1)(b+2)}{2}$, 
        that is, $a=\frac{(b+1)(b+2)}{2}$. Then by the hook length formula,
        \begin{align}
                \label{equation.hook3row}
                \frac{f^{(a+b+1,a-1,1)}}{f^\lambda} &= \frac{(a+b+2)(a+1)(a-1)(b+3)}{(a+b+3)(a+b+1)(a)(b+1)} \\
                &= \frac{\left((b+1)(b+2)+2b+4\right)\left((b+1)(b+2)+2\right)\left((b+1)(b+1-2)\right)(2b+6)}{\left((b+1)(b+2)+2b+6\right)\left((b+1)(b+2)+2b+2\right)
                        \left((b+1)(b+2)\right)(2b+2)} \nonumber\\
                &= \frac{b^7 + 14 b^6 + 82 b^5 + 260 b^4 + 477 b^3 + 486 b^2 + 216 b}{b^7 + 14b^6 + 82b^5 + 260b^4 +477b^3 + 502b^2 + 280b +64} 
                < 1 \nonumber.
        \end{align}

        It follows that if $a < \frac{(b+1)(b+2)}{2}$ then $f^{(a+b+1,a-1,1)} < f^{\lambda}$ because for fixed $b,$ Equation~\eqref{equation.hook3row} decreases as 
        $a$ decreases. To see this, we examine the effect of decreasing $a$ on $\frac{a+b+2}{a+b+3}$, $\frac{a+1}{a+b+1}$, and $\frac{a-1}{a}$ individually.  
        Each of these factors is of the form $\frac{x}{x+d}$ for fixed $d > 0.$ Notice that $g(x) = \frac{x}{x+d}$ is a strictly increasing function for $x >0.$ 
        Therefore, each of the above factors decreases as $a$ decreases. Thus, we have shown that for all $\nu$ such that $\lambda <_D \nu$, 
        $f^\lambda > f^{(a+b,a+1)}$ and $f^\lambda > f^{(a+b+1,a-1,1)}>f^\nu$. 
        Hence, $(a+b,a,1)$ is maximal whenever $a \leqslant \frac{(b+1)(b+2)}{2}$.

        Now in the reverse direction, if $a > \frac{(b+1)(b+2)}{2},$ then $f^{(a+b+1, a-1, 1)} > f^\lambda .$ To see this it suffices to consider 
        $a = \frac{(b+1)(b+2)}{2}+1$ 
        since Equation~\eqref{equation.hook3row} increases as $a$ increases for the same reason as above. If $a = \frac{(b+1)(b+2)}{2}+1,$ then 
        \begin{align*}
                \frac{f^{(a+b+1, a-1, 1)}}{f^{\lambda}} = \frac{b^7 + 14 b^6 + 88 b^5 + 322 b^4 + 739 b^3 + 1056 b^2 + 852 b + 288}{b^7 + 14 b^6 + 88 b^5 + 318 b^4 
                        + 707 b^3 + 964 b^2 + 740 b + 240} > 1.
        \end{align*}

        Therefore, $\lambda$ is maximal if only if $a \leqslant \frac{(b+1)(b+2)}{2}$. 
\end{proof}

\begin{proposition}
        \label{prop.max_3r2}
        Let $\lambda= (a+b,a,2)$ where $a \geqslant 3$. Then $\lambda$ is maximal in the standard immersion poset if and only if 
        $a \leqslant \frac{(b+1)(b+2)}{2}$.
\end{proposition}

\begin{proof}
        We first prove the reverse direction by inducting on $a$. For our base case, let $a=3$ and $3 \leqslant \frac{(b+1)(b+2)}{2}$. To prove that 
        $\lambda=(3+b,3,2)$ is maximal, we follow a similar argument to Proposition ~\ref{prop.max_3r1}. We first classify all partitions $\nu$ such that 
        $\lambda <_D \nu$ and then show that $f^\lambda > f^\nu$ for all such $\nu$ by finding chains in the dominance order that contain maximal elements 
        from the standard immersion poset. Our assumption that $3\leqslant \frac{(b+1)(b+2)}{2}$ implies that $b\geqslant 1$. Hence it suffices to show that 
        $\lambda=(3+b,3,2)$ is maximal for all $b\geqslant 1$. We consider the cases $b=1,b=2$, and $b \geqslant$ 3 separately. It can be
        checked explicitly (for example using {\sc SageMath}~\cite{sagemath}) that $(4,3,2)$ and $(5,3,2)$ are maximal in the standard immersion poset.

        For $b \geqslant 3$, the Hasse diagram in dominance order around $\lambda=(3+b,3,2)$ looks as follows:
        
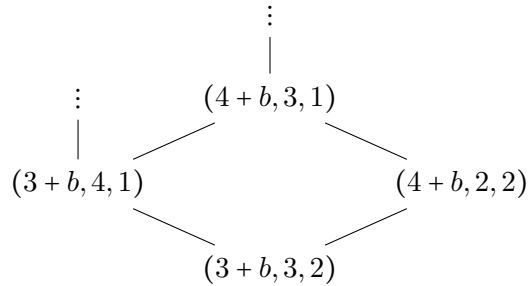
\begin{center}
                \begin{tikzpicture}
                        \matrix (a) [matrix of math nodes, column sep=0.5cm, row sep=0.5cm]{
                                & \vdots\\
                                \vdots & (4+b,3,1)&\\
                                (3+b,4,1) && (4+b,2,2)\\
                                & (3+b,3,2)&\\
                                \\};

                        \foreach \i/\j in {1-2/2-2, 2-1/3-1, 3-1/2-2, 3-3/2-2,4-2/3-1,4-2/3-3}
                        \draw (a-\i) -- (a-\j);
                \end{tikzpicture}
        \end{center}
        If $\lambda <_D \nu$ then $\nu=(3+b,4,1), (4+b,2,2), $ or $\nu$ is contained in some chain $\lambda <_D (3+b,4,1) <_D \nu$. By 
        Proposition~\ref{prop.max_3r1}, $(3+b,4,1)$ is maximal in the standard immersion poset so it suffices to show that $f^\lambda > f^\nu$ for 
        $\nu= (3+b,4,1)$ and $(4+b,2,2)$. By the hook length formula, 
        \[ 
        \frac{f^{(3+b,4,1)}}{f^\lambda} = \frac{4(b)(b+4)}{5(b+1)(b+3)}=\frac{4(b^2+4b)}{5(b^2+4b+3)}<1
        \qquad \text{and} \qquad
        \frac{f^{(4+b,2,2)}}{f^\lambda}= \frac{4(b+4)}{2(b+1)(b+6)}<1.
        \]
        Hence, for $b \geqslant 3$, $(3+b,3,1)$ is maximal in the standard immersion poset, so we have shown that $(3+b,3,1)$ is maximal for all $b\geqslant 1$. 

        Now, let $\lambda= (a+b,a,2)$ where $a\leqslant \frac{(b+1)(b+2)}{2}$ and suppose that for some $a \geqslant 3$, the partition $(c+b,c,2)$ is maximal 
        when $c<a \leqslant \frac{(b+1)(b+2)}{2}$. Again, we show that $f^\lambda > f^\nu$ for $\lambda <_D \nu$. Observe that the Hasse diagram in dominance 
        order around $\lambda$ looks as follows:
        
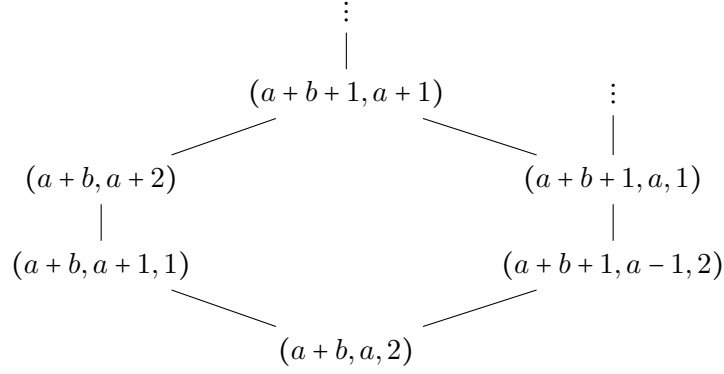
\begin{center}
                \begin{tikzpicture}
                        \matrix (a) [matrix of math nodes, column sep=0.5cm, row sep=0.5cm]{
                                & \vdots \\
                                & (a+b+1,a+1) & \vdots\\
                                (a+b,a+2) && (a+b+1,a,1)\\
                                (a+b,a+1,1) && (a+b+1,a-1,2)\\
                                & (a+b,a,2) \\};

                        \foreach \i/\j in {1-2/2-2, 2-3/3-3, 3-1/2-2, 3-1/4-1, 4-1/5-2, 5-2/4-3, 4-3/3-3, 3-3/2-2}
                        \draw (a-\i) -- (a-\j);
                \end{tikzpicture}
        \end{center}
        By the Hasse diagram, if $\nu$ is a partition such that $\lambda <_D \nu$, then $\nu= (a+b,a+1,1), (a+b,a+2), (a+b+1,a-1,2)$, or $\nu$ is contained in 
        some chain $\lambda <_D (a+b+1,a-1,2) <_D \nu$. Observe that $(a+b+1, a-1,2)=((a-1)+(b+2),a-1,2)$ is maximal by our inductive hypothesis since 
        $a-1<a\leqslant \frac{(b+1)(b+2)}{2} \leqslant \frac{(b+3)(b+4)}{2}$. Therefore, it suffices to check that $f^\lambda > f^\nu$ for 
        $\nu = (a+b,a+1,1), (a+b,a+2), \text{ and } (a+b+1,a-1,2) $.

        For $\nu = (a+b, a+1, 1),$ we have that
        \begin{equation}
                \label{equation.(a+b, a+1, 1)}
                \frac{f^{(a+b, a+1,1)}}{f^\lambda} = \frac{2b(a+b+1)(a+1)}{(a+b)(a-1)(a+2)(b+1)}.
        \end{equation}
        Since 
        \[ 
        \frac{d}{da} \frac{f^{(a+b, a+1,1)}}{f^\lambda} = 
        -\frac{(2 b (a^4 + 2 a^3 b + 4 a^3 + a^2 b^2 + 5 a^2 b + 7 a^2 + 2 a b^2 + 8 a b + 2 a + 3 b^2 + 3 b - 2))}{((a - 1)^2 (a + 2)^2 (b + 1) (a + b)^2)},
        \]
        we have that Equation~\eqref{equation.(a+b, a+1, 1)} decreases as $a$ increases. Therefore, it suffices to consider $a = 3$ which we have done in our 
        base case. Hence, $f^{(a+b, a+1,1)} < f^\lambda.$ 

        For $\nu = (a+b, a+2),$ we have that
        \[
        \frac{f^{(a+b, a+2)}}{f^\lambda} = \frac{2(b-1)(a+b+2)}{(a-1)(b+1)(a+b)(a+2)} 
        \leqslant \frac{a+b+2}{(a+b)(a+2)}
        \]
        since $a \geqslant 3.$ As $(a+b)(a+2) = a^2 + 2a + ab + 2b \geqslant a+b+2,$ we have that $f^{(a+b, a+2)} < f^\lambda.$

        Lastly, for $\nu = (a+b+1, a-1, 2),$ we first consider when $a = \frac{(b+1)(b+2)}{2}.$ By the hook length formula, we have
        \begin{align}
                \label{equation.(a+b+1,a-1,2)}
                \frac{f^{(a+b+1, a-1, 2)}}{f^\lambda} &= \frac{(b+3)(a+b+1)(a-2)(a+1)}{(b+1)(a+b+3)(a-1)(a+b)} \\
                &= \frac{(2b+6)((b+1)(b+2)+2b+2)((b+1)(b+2)-4)((b+1)(b+2)+2)}{((b+1)(b+2)+2b)(2b+2)((b+1)(b+2)+2b+6)((b+1)(b+2)-2)} \nonumber \\
                &= \frac{b^7 + 14 b^6 + 78 b^5 + 220 b^4 + 321 b^3 + 182 b^2 - 80 b - 96}{b^7 + 14 b^6 + 78 b^5 + 220 b^4 + 321 b^3 + 214 b^2 + 48 b} \nonumber \\
                &< 1. \nonumber
        \end{align}

        Following a similar argument as in Proposition~\ref{prop.max_3r1} for Equation~\eqref{equation.hook3row}, we can see that for fixed $b$, 
        Equation~\eqref{equation.(a+b+1,a-1,2)} decreases as $a$ decreases by considering $\frac{a+b+1}{a+b+3},\frac{a-2}{a-1},$ and $\frac{a+1}{a+b}.$

        We have thus shown that when $a \geqslant 3,$ $\lambda = (a+b, a, 2)$ is maximal if $a \leqslant \frac{(b+1)(b+2)}{2}.$ For the reverse direction, 
        consider Equation~\eqref{equation.(a+b+1,a-1,2)} when $a = \frac{(b+1)(b+2)}{2} + 1.$ We have that
        \[
        \frac{f^{(a+b+1, a-1,2)}}{f^\lambda} 
        = \frac{b^6 + 12 b^5 + 60 b^4 + 162 b^3 + 243 b^2 + 162 b}{b^6 + 12 b^5 + 60 b^4 + 158 b^3 + 219 b^2 + 150 b + 40} 
        > 1.
        \]
        Since Equation~\eqref{equation.(a+b+1,a-1,2)}  increases as $a$ increases, $\lambda \leqslant_{std} (a+b+1, a-1, 2)$ when $a > \frac{(b+1)(b+2)}{2}.$

        Therefore, $\lambda$ is maximal if and only if $a \leqslant \frac{(b+1)(b+2)}{2}.$
\end{proof}

\begin{table}[t]
\begin{center}
                \setlength{\arrayrulewidth}{0.5mm}
                \setlength{\tabcolsep}{12pt}
                \renewcommand{\arraystretch}{1.3}
                \begin{tabular}{ |p{2cm}|p{2cm}|p{2cm}|p{3cm}| }
                        \hline
                        Value for $\alpha$ & $\lambda=(\alpha,\beta)$ & $\lambda=(\alpha,\beta,1)$ & $\lambda=(\alpha,\beta,2)$ \\
                        \hline
                        $\alpha \geqslant 2$  & $(\alpha,1)$  &  & \\
                        $\alpha \geqslant 3$  & $(\alpha,2)$  & $(\alpha,2,1)$ & \\
                        $\alpha \geqslant 4$  &          & $(\alpha,3,1)$ & $(\alpha,3,2)$\\
                        $\alpha \geqslant 5$  & $(\alpha,3)$  &           &\\
                        $\alpha \geqslant 6$  & $(\alpha,4)$  & $(\alpha,4,1)$ & $(\alpha,4,2)$\\
                        $\alpha \geqslant 7$  & $(\alpha,5)$  & $(\alpha,5,1)$ & $(\alpha,5,2)$\\
                        $\alpha \geqslant 8$  &          & $(\alpha,6,1)$ & $(\alpha,6,2)$\\
                        $\alpha \geqslant 9$  & $(\alpha,6)$  &           & \\
                        $\alpha \geqslant 10$ & $(\alpha,7)$  & $(\alpha,7,1)$ & $(\alpha,7,2)$\\
                        $\alpha \geqslant 11$ & $(\alpha,8)$  & $(\alpha,8,1)$ & $(\alpha,8,2)$\\
                        $\alpha \geqslant 12$ & $(\alpha,9)$  & $(\alpha,9,1)$ & $(\alpha,9,2)$\\
                        $\alpha \geqslant 13$ &          & $(\alpha,10,1)$ & $(\alpha,10,2)$\\
                        $\alpha \geqslant 14$ & $(\alpha,10)$ &                &        \\
                        $\alpha \geqslant 15$ & $(\alpha,11)$ & $(\alpha,11,1)$ & $(\alpha,11,2)$\\
                        $\alpha \geqslant 16$ & $(\alpha,12)$ & $(\alpha,12,1)$ & $(\alpha,12,2)$\\
                        $\alpha \geqslant 17$ & $(\alpha,13)$ & $(\alpha,13,1)$ & $(\alpha,13,2)$\\
                        $\alpha \geqslant 18$ & $(\alpha,14)$ & $(\alpha,14,1)$ & $(\alpha,14,2)$\\
                        $\alpha \geqslant 19$ &          & $(\alpha,15,1)$ & $(\alpha,15,2)$\\
                        $\alpha \geqslant 20$ & $(\alpha,15)$ &                &       \\
                        \vdots &      \vdots    & \vdots & \vdots \\ 
                        \hline
                \end{tabular}
\end{center}
\caption{Necessary and sufficient conditions for maximality of a partition $\lambda$.
\label{table.maximal}}
 \end{table}

\begin{remark}
        We may translate the results of Propositions~\ref{prop.max_two_row}, ~\ref{prop.max_3r1}, and ~\ref{prop.max_3r2} into statements 
        about partitions of the form $(\alpha,\beta)$, $(\alpha, \beta, 1)$, and $(\alpha, \beta, 2)$. Table~\ref{table.maximal} summarizes 
        our maximality conditions.
\end{remark}

We next classify all maximal hook shape partitions. As noted in Lemma \ref{lemma.col}, $(1^n)  \lessdot_{std} (n)$ and so the single column shape is only maximal when $n=1$. By Lemma \ref{two col}, $(2, 1^b) \lessdot_{std} (2^2, 1^{b-2})$ whenever $b \geqslant 3$. Since $(2,1,1) \lessdot_{std} (3,1)$, the only maximal hook shape with arm length $2$ is $(2,1)$. In the following proposition, we investigate all hook shape partitions with arm length greater than $2$. 
\begin{proposition}
        \label{maxhook}
        Let $\lambda = (a, 1^{b})$ be a hook shape partition such that $a > 2$. Then $\lambda$ is a maximal element in the standard 
        immersion poset if and only if $b \leqslant 2$. 
\end{proposition}

\begin{proof}
        When $b = 1$, the only partition that dominates $(a, 1)$ is $(a + 1)$ and $f^{(a,1)} = (a + 1) - 1 > 1 = f^{(a + 1)}$. 
        Thus, $(a, 1)$ is maximal. When $b = 2$, the only partitions that dominate $(a, 1^2)$ are $(a + 2)$,  $(a+1, 1)$, 
        and $(a, 2)$. By the hook length formula, 
        \[
        \frac{f^{(a + 1, 1)}}{f^{(a, 1^2)}}=\frac{2}{a} < 1 \qquad \text{ and } \qquad \frac{f^{(a, 2)}}{f^{(a, 1^2)}}
        =\frac{(a+2)(a-1)}{(a+1)a} = \frac{a^2+a-2}{a^2+a} < 1.
        \]
        Therefore, no partition dominates $(a,1^2)$ and has more standard Young tableaux, so $(a, 1^2)$ is maximal.

        When $b \geqslant 3$, $(a, 1^b) \s (a, 2, 1^{b - 2})$, by the hook length formula:
        \[
        \frac{f^{(a, 1^b)}}{f^{(a, 2, 1^{b-2})}}= \frac{a(a + b - 1)}{(a + b)(a - 1)(b-1)} \leqslant 1
        \]
        since 
        \[
        (a + b)(a - 1)(b-1) \geqslant 2(a-1)(a + b) \geqslant a(a + b - 1).
        \]
        Therefore $ f^{(a, 1^b)} \leqslant f^{(a, 2, 1^{b-2})}$ and $(a, 1^b) <_D (a, 2, 1^{b-2})$, so we have 
        $(a, 1^b) \s (a, 2, 1^{b - 2})$ whenever $b \geqslant 3$. 
\end{proof}

\begin{proposition}
        \label{prop.notmaximal}
        If $\lambda$ is a maximal element in the standard immersion poset, then $\lambda_1 > \lambda_2$. 
\end{proposition}

\begin{proof}
Suppose by contradiction that $\lambda = (a^b, \lambda_{b + 1}, \ldots)$ with $a > \lambda_{b+1}$ and $b \geqslant 2$. Let 
$\mu = (a+1, a^{b - 2}, a-1, \lambda_{b + 1}, \ldots)$ and denote by $\mathsf{SYT}(\lambda)$ the set of all
standard Young tableaux of shape $\lambda$. The map
\[
	\varphi \colon \mathsf{SYT}(\lambda) \rightarrow \mathsf{SYT}(\mu),
 \]
where $\varphi(T)$ is the standard Young tableau obtained from $T$ by moving the box in position $(b, a)$ to 
position $(1, a+1)$, is an injection. Therefore, $\lambda <_D \mu$ and $f^\lambda \leqslant f^{\mu}$, which 
implies $\lambda \leqslant_I \mu$ and thus demonstrates that $\lambda$ is not a maximal element in the 
standard immersion poset. 
\end{proof}

In fact, the injection $\varphi$ used in the proof of Proposition~\ref{prop.notmaximal} remains an injection when 
the domain and codomain are extended to semistandard Young tableaux of content $\nu$, for any 
$\nu \vdash |\lambda|$. Injection arguments between sets of semistandard Young tableaux are expanded 
on in Section~\ref{section.injections}. In particular, this result is extended to the immersion poset in Corollary~\ref{cor.firstparts}. 

We conclude this section with a conjecture about more general maximal elements in the standard immersion poset.

\begin{conjecture}
        \label{conjecture.maximal}
        Suppose $\lambda=(\sum_{i=1}^\ell a_i, \sum_{i=1}^{\ell-1}a_i,\dots,a_2+a_1,a_1)$ for $\ell>2$. If
        \[
        \binom{a_j+2}{2} \geqslant \sum_{i=1}^{j-1}a_i + j -2
        \]
        is satisfied for all $2\leqslant j\leqslant \ell$, then $\lambda$ is maximal in the standard immersion poset. 
\end{conjecture}

This conjecture has been verified with {\sc SageMath}~\cite{sagemath} for $|\lambda|\leqslant 30$.

\begin{remark}
        Proposition~\ref{prop.max_two_row} addresses the case $\ell=2$ associated to Conjecture~\ref{conjecture.maximal}.
        Note that for $\ell=2$ the condition stated in Conjecture~\ref{conjecture.maximal} reads
        \[
        \binom{a_2+2}{2} \geqslant a_1, \quad \text{whereas the condition from Proposition~\ref{prop.max_two_row} is} \quad
        \binom{a_2+2}{2} > a_1.
        \]
        This discrepancy comes from the fact that for $\ell>2$, there are more factors contributing to the inequality in $\frac{f^\mu}{f^\lambda}<1$.
\end{remark}

\section{Immersion poset}
\label{section.immersion}

In this section we turn to the immersion poset. In Section~\ref{section.properties}, we study basic properties of the immersion poset.
In Section~\ref{section.injections}, we provide explicit injections between certain sets of semistandard Young tableaux, which are used to 
determine statements about maximal elements and cover relations in the immersion poset. In Sections~\ref{section.hooks} and~\ref{section.two column}, 
we study the immersion poset restricted to hook partitions and two column partitions, respectively. We conclude in Section~\ref{section.lowerintervals}
with conjectures about certain lower intervals in the immersion poset and prove that the conjectured intervals give Schur-positive sums of power sum
symmetric functions.

\subsection{Properties of the immersion poset}
\label{section.properties}

We begin by specifying the minimal element.

\begin{lemma}
        \label{lemma.minimal immersion}
        The partition $(1^n)$ is the unique minimal element in the immersion poset $(\mathcal{P}(n), \leqslant_I)$.
\end{lemma}

\begin{proof}
        We have $f^{(1^n)}=1\leqslant f^\lambda$ for all $\lambda \in \mathcal{P}(n)$. Furthermore $K_{(1^n),\alpha}=0\leqslant K_{\lambda,\alpha}$ for 
        all $\alpha\neq (1^n)$ and $\lambda \in \mathcal{P}(n)$. By Lemma~\ref{lemma.immersion K} this proves the claim.
\end{proof}

Analogously to Lemma~\ref{lemma.col}, we prove the following result.

\begin{lemma}
        \label{lemma.cover}
        We have
        \begin{enumerate}
                \item $(1^n) \lessdot_{I} (n)$ for all $n$ and
                \item $(2,1^{n-2}) \lessdot_{I} (n-1,1)$ for all $n \geqslant 3$.
        \end{enumerate}
\end{lemma}

\begin{proof}
        By Lemma~\ref{lemma.minimal immersion}, we have $(1^n) <_I (n)$. By Lemma~\ref{lemma.col}, $(1^n) \lessdot_{std} (n)$. Since in the immersion
        poset there are fewer order relations than in the standard immersion poset, the first part of the lemma follows.

        We have $(2,1^{n-2}) <_{I} (n-1,1)$ since
        \[
        s_{(2,1^{n-2})}=(n-1) m_{(1^n)} + m_{(2,1^{n-2})} \text{ and } 
        s_{(n-1,1)}=(n-1) m_{(1^n)} + (n-2) m_{(2,1^{n-2})}+\sum\limits_{\mu \neq (1^n),(2,1^{n-2})} K_{(n-1,1),\mu} m_\mu.
        \]
        Again, since by Lemma~\ref{lemma.col} we have $(2,1^{n-2}) \lessdot_{std} (n-1,1)$, the second part of the lemma follows.
\end{proof}

Unlike in the standard immersion poset, where $\lambda$ and $\lambda^t$ are always comparable as long as they are comparable in
dominance order (see Remark~\ref{remark.standard immersion}), this is not always true in the immersion poset. For example $\lambda=(4, 4, 2, 1, 1)$ 
and $\lambda^t$ are not comparable in the immersion poset since $K_{(4,4,2,1,1), (4,4,1,1,1,1)} > K_{(5,3,2,2), (4,4,1,1,1,1)}$. For hook partitions, it is however true that $\lambda<_I\lambda^t$ if $\lambda<_D \lambda^t$ (see Corollary~\ref{corollary.hook transpose}).

We prove the analog of Lemma~\ref{two col} in the next section using injections on semistandard Young tableaux. See
Corollaries~\ref{corollary.cover two col immersion0}, \ref{corollary.cover two col immersion1}, and~\ref{corollary.cover two col immersion2}.

\subsection{Explicit injections}
\label{section.injections}
		
Recall from Lemma~\ref{lemma.immersion K} that 
$\lambda \leqslant_I \mu$ if and only if $K_{\lambda,\nu}\leqslant K_{\mu,\nu}$ for all $\nu \in \mathcal{P}(n)$. 
The Kostka number $K_{\lambda,\nu}$ is the cardinality of the set of semistandard Young tableaux $\mathsf{SSYT}(\lambda,\nu)$ of
shape $\lambda$ and content $\nu$.
Hence we can analyze the order relations $\lambda \leqslant_I \mu$ by constructing explicit injections 
\begin{equation}
\label{equation.phi}
	\varphi \colon \mathsf{SSYT}(\lambda,\nu) \to \mathsf{SSYT}(\mu,\nu)
\end{equation}
for all $\nu \in \mathcal{P}(n)$.

To this end, we present one such injection, where $\mu$ differs from $\lambda$ by moving a single cell from the $c$-th column to the 
$(c + 1)$-th column, and $\lambda$ has a bound on the relative size of the two columns. Upon establishing this first injection, we
refine it to obtain more precise bounds on the relative size of the columns. We partially characterize what elements cannot be maximal in 
the immersion poset, similar to those given in Section~\ref{section.maximal} for the standard immersion poset.

Let
\begin{equation}
\label{equation.mu lambda}
\begin{split}
	\lambda &= (\lambda_1, \dots, \lambda_\alpha,  c^\beta, \lambda_{\alpha + \beta + 1}, \dots),\\
	\mu &= (\lambda_1, \dots, \lambda_\alpha, c + 1,  c^{\beta - 2}, c - 1, \lambda_{\alpha +\beta + 1}, \dots),
\end{split}
\end{equation}
such that either $\alpha > 0$ and $\lambda_{\beta + \alpha + 1} < c < \lambda_\alpha$, or $\alpha = 0$ and $\lambda_{\beta + \alpha + 1} < c$. 
In particular, $\lambda_{\beta + \alpha + 1}$ can be $0$. We define a map 
\[
	\varphi_0 \colon \mathsf{SSYT}(\lambda,\nu) \to \mathsf{YT}(\mu,\nu),
\]
where $\mathsf{YT}(\mu,\nu)$ is the set of all tableaux of shape $\mu$ and content $\nu$, not necessarily semistandard.
We will show in Proposition~\ref{two column injection 0} that when $\beta \geqslant \alpha + 2$, the image of $\varphi_0$ will be contained in $\mathsf{SSYT}(\mu,\nu)$, so $\varphi_0$ will be as in ~\eqref{equation.phi}.

For $T \in \mathsf{SSYT}(\lambda,\nu)$, we define $\varphi_0(T)$ as follows.
Suppose the entries in the $c$-th column of $T$ in increasing order are 
$x_{\beta + \alpha}, x_{\beta + \alpha - 1}, \dots, x_1$ and the entries in the $(c + 1)$-th column of $T$ in increasing order are 
$y_\alpha, y_{\alpha - 1}, \dots, y_1$. Let $i$ be the smallest index such that $x_i > y_i$.
If no such index exists, let $i = \alpha + 1$. Then $\varphi_0(T)$ is the tableau such that the entries in the $c$-th column of $\varphi_0(T)$ are 
\[
	x_{\beta + \alpha}, x_{\beta + \alpha - 1}, \dots, x_{i + 1}, y_{i - 1}, y_{i - 2}, \dots, y_1,
\]
the entries in the $(c + 1)$-th column of $\varphi_0(T)$ are 
\[
	y_\alpha, y_{\alpha - 1}, \dots, y_i, x_i, x_{i - 1}, \dots, x_1,
\]
and all other entries are the same as those in $T$. In other words, $\varphi_0$ moves the cell containing $x_1$ to the $(\alpha + 1)$-th row 
of the $(c + 1)$-th column, and swaps each $x_j$ with $y_{j - 1}$ for all $2 \leqslant j \leqslant i$.

More concretely, the $c$-th and $(c+1)$-th column in $T$
and $\varphi_0(T)$ look as follows:
\begin{equation}
\label{phi.zero tableaux}
T:
\ytableausetup{boxsize=5em}%
\scalebox{0.8}[.25]{
\begin{ytableau}
\filling{4}{x_{\beta + \alpha}} & \filling{4}{y_\alpha} \\
\filling{4}{\vdots} & \filling{4}{\vdots} \\
\filling{4}{x_{\beta + i}} & \filling{4}{y_i} \\
\filling{4}{x_{\beta + i - 1}} & *(yellow) \filling{4}{y_{i - 1}} \\
\filling{4}{\vdots} & *(yellow) \filling{4}{\vdots} \\
\filling{4}{x_{\beta + 1}} & *(yellow) \filling{4}{y_1} \\
\filling{4}{x_\beta} \\
\filling{4}{\vdots} \\
\filling{4}{x_{i + 1}} \\
*(green) \filling{4}{x_i } \\
*(green) \filling{4}{\vdots} \\
*(green) \filling{4}{x_2} \\
*(green) \filling{4}{x_1}
\end{ytableau}} \qquad
\varphi_0(T):
\ytableausetup{boxsize=5em}%
\scalebox{0.8}[.25]{
\begin{ytableau}
\filling{4}{x_{\beta + \alpha}} & \filling{4}{y_\alpha} \\
\filling{4}{\vdots} & \filling{4}{\vdots} \\
\filling{4}{x_{\beta + i}} & \filling{4}{y_i} \\
\filling{4}{x_{\beta + i - 1}} & *(green) \filling{4}{x_i} \\
\filling{4}{\vdots} & *(green) \filling{4}{\vdots} \\
\filling{4}{x_{\beta + 1}} & *(green) \filling{4}{x_2} \\
\filling{4}{x_\beta} & *(green) \filling{4}{x_1} \\
\filling{4}{\vdots} \\
\filling{4}{x_{i + 1}} \\
*(yellow) \filling{4}{y_{i - 1}} \\
*(yellow) \filling{4}{\vdots} \\
*(yellow) \filling{4}{y_1}
\end{ytableau}}
\end{equation}
The cells marked in green contain the entries that move from the $c$-th column to the $(c + 1)$-th column,
and the cells marked in yellow are the entries that move from $(c + 1)$-th column to the $c$-th column.
We continue to use this convention for all subsequent examples of $\varphi_0$.

\begin{remark}
\label{choice of i}
Observe that by our choice of $i$, both $x_{i + 1} < x_{i - 1} \leqslant y_{i - 1}$ and $y_i < x_i$, so the columns of $\varphi_0(T)$ are strictly 
increasing by construction.
\end{remark}

\begin{example}
For $\lambda = (3, 2, 1^4)$ and $\mu = (3, 2, 2, 1^2, 0)$, we have $c=1$, $\alpha=2$ and $\beta=4$. Here are some examples of the injection 
$\varphi_0$ on various tableaux of shape $\lambda$:
\begin{equation*}
\ytableausetup{boxsize=normal}
\begin{ytableau}
1 & 1 & 3 \\
2 & 2 \\
3 \\
4 \\
5 \\
*(green) 6
\end{ytableau}
\mapsto
\begin{ytableau}
1 & 1 & 3 \\
2 & 2 \\
3 & *(green) 6 \\
4 \\
5
\end{ytableau}
\hspace{50pt}
\begin{ytableau}
1 & 1 & 2 \\
2 & *(yellow) 7 \\
3 \\
4 \\
*(green) 5 \\
*(green) 6
\end{ytableau}
\mapsto
\begin{ytableau}
1 & 1 & 2 \\
2 & *(green) 5 \\
3 & *(green) 6 \\
4 \\
*(yellow) 7
\end{ytableau}
\hspace{50pt}
\begin{ytableau}
1 & *(yellow) 6 & 9 \\
2 & *(yellow) 8 \\
3 \\
*(green) 4 \\
*(green) 5 \\
*(green) 7
\end{ytableau}
\mapsto
\begin{ytableau}
1 & *(green) 4 & 9 \\
2 & *(green) 5 \\
3 & *(green) 7 \\
*(yellow) 6 \\
*(yellow) 8
\end{ytableau}
\end{equation*}
\end{example}

\begin{proposition}
\label{two column injection 0}
Let $\lambda$ and $\mu$ be as in \eqref{equation.mu lambda} with $\beta \geqslant \alpha + 2$.
Then $\varphi_0$ as defined above is an injection 
\[
\varphi_0 \colon \mathsf{SSYT}(\lambda,\nu) \to \mathsf{SSYT}(\mu,\nu).
\]
\end{proposition}

\begin{proof}
Let $T \in \mathsf{SSYT}(\lambda,\nu)$. Note that the content does not change under $\varphi_0$. We need to check that the $c$-th and 
$(c+1)$-th columns of $\varphi_0(T)$ are strictly increasing, and that the $(\alpha - i + 2)$-th through $(\alpha + \beta - 1)$-th rows of $\varphi_0(T)$ are 
weakly increasing, since all other entries are identical to those in $T$. (It may be helpful to consult~\eqref{phi.zero tableaux}.)
The columns are strictly increasing by Remark \ref{choice of i}.

For rows, we first consider the $(\alpha - i + 2)$-th through $(\alpha + 1)$-th rows.
Due to the bound $\beta \geqslant \alpha + 2$, in the $c$-th column, these rows contain $x_{\beta + i - 1}, \dots, x_\beta.$
In the $(c + 1)$-th column, irrespective of the bound on $\alpha$ and $\beta$, these rows contain $x_i, \dots, x_1$.
In particular, the bound $\beta \geqslant \alpha + 2$ makes it so that there is no $y_j$ entry in these rows, so there is no ``overlap'' of $y_j$ and $x_k$ 
for $1 \leqslant k \leqslant i$. The rows are thus strictly increasing because
$x_j > x_k$ for all $j < k$, so the $x_j$ entries in the $(c + 1)$-th column are greater than the entries to their left in the $c$-th column;
and $x_j < x_{j - 1} \leqslant y_{j - 1}$ for $2 \leqslant j \leqslant i$,
so the $x_j$ entries in the $(c + 1)$-th column are less than any entries to their right, originally from $T$.
(Such entries on the right do not necessarily exist. In particular, $x_1$ never has any cell to its right.)

Now consider the $(\alpha + 2)$-th through $(\alpha + \beta - 1)$-th rows. In the $c$-th column, these rows contain
\[
x_{\beta - 1}, \dots, x_{i + 1}, y_{i - 1}, \dots, y_1,
\]
and in the $(c + 1)$-th column, these rows contain no cells.
They are weakly increasing because \newline $y_{j - 1} \geqslant x_{j - 1} > x_j$ for $2 \leqslant j \leqslant i$,
so the $y_{j - 1}$ entries in the $c$-th column are greater than the entries to their left, originally from $T$,
and they have no cells to their right.

To show injectivity, we define an explicit inverse $\psi_0$. Let $T' \in \varphi_0(\mathsf{SSYT}(\lambda,\nu))$. Suppose the entries in the $c$-th column of $T'$ in 
increasing order are $x'_{\beta + \alpha - 1}, x'_{\beta + \alpha - 2}, \dots, x'_1$, and the entries in the $(c + 1)$-th column of $T'$ in increasing order are 
$y'_{\alpha + 1}, y'_\alpha, \dots, y'_1$.
Let $i'$ be the smallest index such that $y'_{i'} > x'_{i'}$.
This $i'$ will be equal to the $i$ from the definition of $\varphi_0$, because $x_i > x_{i + 1}$ and $x_j \leqslant y_j$ for $1 \leqslant j \leqslant i - 1$ in $T$.

Then $\psi_0(T')$ is the tableau of shape $\lambda$ such that the entries in the $c$-th column of $\psi_0(T')$ are
\[
	x'_{\beta + \alpha - 1}, x'_{\beta + \alpha - 2}, \dots, x'_{\beta - 1}, x'_{\beta - 2}, x'_{\beta - 3}, \dots, x'_{i'}, y'_{i'}, y'_{i' - 1}, \dots, y'_1,
\]
the entries in the $(c + 1)$-th column of $\psi_0(T')$ are
\[
	y'_{\alpha + 1}, y'_\alpha, \dots, y'_{i' + 1}, x'_{i' - 1}, x'_{i' - 2}, \dots, x'_1,
\]
and all other entries are the same as those in $T'$. In other words, $\psi_0$ moves the cell containing $y'_1$ to the $(\alpha + \beta)$-th position in the 
$c$-th column, and swaps each $y'_{j'}$ with $x'_{j' - 1}$ for all $2 \leqslant j' \leqslant i'$. Concretely:
\begin{equation}
\label{psi.zero tableaux}
T':
\ytableausetup{boxsize=5.5em}
\scalebox{0.8}[.25]{
\begin{ytableau}
\filling{4}{x'_{\beta + \alpha - 1}} & \filling{4}{y'_{\alpha + 1}} \\
\filling{4}{\vdots} & \filling{4}{\vdots} \\
\filling{4}{x'_{\beta + i' - 1}} & \filling{4}{y'_{i' + 1}} \\
\filling{4}{x'_{\beta + i' - 2}} & *(green) \filling{4}{y'_{i'}} \\
\filling{4}{\vdots} & *(green) \filling{4}{\vdots} \\
\filling{4}{x'_\beta} & *(green) \filling{4}{y'_2} \\
\filling{4}{x'_{\beta - 1}} & *(green) \filling{4}{y'_1} \\
\filling{4}{x'_{\beta - 2}} \\
\filling{4}{\vdots} \\
\filling{4}{x'_{i'}} \\
*(yellow) \filling{4}{x'_{i' - 1}} \\
*(yellow) \filling{4}{\vdots} \\
*(yellow) \filling{4}{x'_1}
\end{ytableau}} \qquad
\psi_0(T'): 
\scalebox{0.8}[.25]{
\begin{ytableau}
\filling{4}{x'_{\beta + \alpha - 1}} & \filling{4}{y'_{\alpha + 1}} \\
\filling{4}{\vdots} & \filling{4}{\vdots} \\
\filling{4}{x'_{\beta + i' - 1}} & \filling{4}{y_{i' + 1}} \\
\filling{4}{x'_{\beta + i' - 2}} & *(yellow) \filling{4}{x'_{i' - 1}} \\
\filling{4}{\vdots} & *(yellow) \filling{4}{\vdots} \\
\filling{4}{x'_\beta} & *(yellow) \filling{4}{x'_1} \\
\filling{4}{x'_{\beta - 1}} \\
\filling{4}{x'_{\beta - 2}} \\
\filling{4}{\vdots} \\
\filling{4}{x'_{i'}} \\
*(green) \filling{4}{y'_{i'}} \\
*(green) \filling{4}{\vdots} \\
*(green) \filling{4}{y'_2} \\
*(green) \filling{4}{y'_1}
\end{ytableau}}
\end{equation}

Since $i' = i$, $\psi_0$ moves back exactly the entries in $T'$ that were originally moved by $\varphi_0$ in $T$, so $\psi_0$ is the inverse of $\varphi_0$.
\end{proof}

As a corollary, the injection describes a class of cover relations in the immersion poset. As a specific example, it can partially address the two 
column case, which was completely addressed by Lemma~\ref{two col} for the standard immersion poset.

\begin{corollary}
\label{corollary.cover two col immersion0}
The partitions $\lambda$ and $\mu$ as in~\eqref{equation.mu lambda} with $\beta \geqslant \alpha + 2$ form a cover in the immersion poset. 
In particular, $\lambda = (2^\alpha, 1^\beta)$ and $\mu = (2^{\alpha + 1}, 1^{\beta - 2})$ form a cover.
\end{corollary}

\begin{proof}
The partition $\mu$ covers $\lambda$ in dominance order, and the injection shows that $\mu$ is greater than $\lambda$ in the immersion 
poset, so $\mu$ must also cover $\lambda$ in the immersion poset.
\end{proof}

The injection also gives a few conditions on which partitions cannot be maximal.
\begin{corollary}
\label{cor.firstparts}
If $\lambda = (a^{\beta}, b, \dots)$, where $a > b $, and $\beta \geqslant 2$, then $\lambda$ is not maximal.
\end{corollary}

\begin{proof}
We have $\beta \geqslant 2$ with $\alpha = 0$, so we can apply the injection.
\end{proof}

\begin{corollary}
\label{cor.notmaxhooks}
If $\lambda = (a, b^{\beta}, c, \dots)$, where $a > b > c$, and $\beta \geqslant 3$, then $\lambda$ is not maximal.
In particular, $\lambda = (a, 1^\beta)$ is not maximal for $a \geqslant 2$, $\beta \geqslant 3$.
\end{corollary}

\begin{proof}
We have $\beta \geqslant 3$ with $\alpha = 1$, so we can apply the injection.
\end{proof}
Note that Corollary~\ref{cor.firstparts} and Corollary~\ref{cor.notmaxhooks} repeat the results from Proposition~\ref{prop.notmaximal} and the forward direction of Proposition~\ref{maxhook} concerning nonmaximal elements in the standard immersion poset.

\begin{corollary}
If $\lambda = (a, b, c, d)$ is maximal in the immersion poset, then it has no more than two identical non-zero parts.
\end{corollary}

\begin{proof}
If $\lambda$ has three or more identical parts, then $\lambda$ is one of $(a^4)$, $(a^3, d)$, or $(a, b^3)$, so we can apply the injection.
\end{proof}

As stated in the proof of Proposition~\ref{two column injection 0}, the bound $\beta \geqslant \alpha + 2$ is necessary for $\varphi_0(T)$ to be 
semistandard for $T$ semistandard. When $\beta < \alpha + 2$, $\varphi_0$ can cause an ``overlapping'' row, where for certain 
$1 \leqslant j \leqslant \alpha - \beta + 2$, $y_{\beta - 2 + j}$ is to the left of $x_j$, yet $y_{\beta - 2 + j} > x_j$.

\begin{example}
\label{example.failure}
For $\lambda = (2^2, 1^3)$, so $\alpha = 2$ and $\beta = 3 = \alpha + 1$, $\varphi_0$ can give:
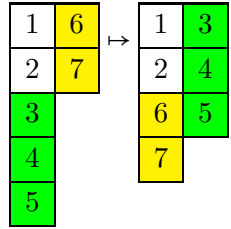
\begin{equation*}
\ytableausetup{boxsize=normal}
\begin{ytableau}
1 & *(yellow) 6 \\
2 & *(yellow) 7 \\
*(green) 3 \\
*(green) 4 \\
*(green) 5
\end{ytableau}
\mapsto
\begin{ytableau}
1 & *(green) 3 \\
2 & *(green) 4 \\
*(yellow) 6 & *(green) 5 \\
*(yellow) 7 
\end{ytableau}
\end{equation*}

One natural modification to restore weakly increasing rows is to swap $y_{\beta - 2 + j}$ and $x_j$ whenever the problem occurs. Unfortunately, doing 
so on its own would not maintain injectivity. If we try to swap the $5$ and $6$ in the previous example, our final tableau is the same as the 
following tableau obtained from $\varphi_0$ with no switches:
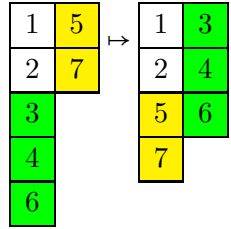
\begin{equation*}
\begin{ytableau}
1 & *(yellow) 5 \\
2 & *(yellow) 7 \\
*(green) 3 \\
*(green) 4 \\
*(green) 6
\end{ytableau}
\mapsto
\begin{ytableau}
1 & *(green) 3 \\
2 & *(green) 4 \\
*(yellow) 5 & *(green) 6 \\
*(yellow)7 
\end{ytableau}
\end{equation*}
\end{example}

However, if we are able to implement subsequent modifications in a way such that the resulting tableau is semistandard, yet cannot be obtained 
from $\varphi_0$ alone, then we can restore injectivity.

We now define the modification of our original $\varphi_0$ injection for the case when $\beta = \alpha + 1$, and $\alpha \geqslant 2$, which we call 
\[
\varphi_1 \colon \mathsf{SSYT}(\lambda) \to \mathsf{SSYT}(\mu).
\]
From now on, we drop the content $\nu$ as all maps in this subsection preserve the content.

Let $T \in \mathsf{SSYT}(\lambda)$. As before, suppose that the entries in the $c$-th column of $T$ in increasing order are 
$x_{\beta + \alpha}, x_{\beta + \alpha - 1}, \dots, x_1$, and the entries in the $(c + 1)$-th column of $T$ in increasing order are 
$y_\alpha, y_{\alpha - 1}, \dots, y_1$.
We define $\varphi_1(T)$ to be the same as $\varphi_0(T)$ if $\varphi_0(T) \in \mathsf{SSYT}(\mu)$.

If $\varphi_0(T) \not \in \mathsf{SSYT}(\mu)$, then necessarily $i = \alpha + 1$ as defined for $\varphi_0$ and $x_1$ is to the right of 
$y_{\beta - 1} = y_\alpha$, with $y_\alpha > x_1$. Then $\varphi_1(T)$ is the same as $\varphi_0(T)$, except we swap $y_\alpha$ with $x_1$, 
as well as $x_{\beta + 1}$ with $x_{\beta}$.

Concretely, when $\varphi_1(T) \neq \varphi_0(T)$, the $c$-th and $(c+1)$-th columns of $T$, $\varphi_0(T)$, and $\varphi_1(T)$ look as follows:
\begin{equation}
\label{phi.one tableaux}
T:
\ytableausetup{boxsize=5em}%
\scalebox{0.8}[.25]{
\begin{ytableau}
\filling{4}{x_{\beta + \alpha}} & *(yellow) \filling{4}{y_\alpha} \\
\filling{4}{x_{\beta + \alpha - 1}} & *(yellow) \filling{4}{y_{\alpha - 1}} \\
\filling{4}{\vdots} & *(yellow) \filling{4}{\vdots} \\
\filling{4}{x_{\beta + 2}} & *(yellow) \filling{4}{y_2} \\
\filling{4}{x_{\beta + 1}} & *(yellow) \filling{4}{y_1} \\
*(green) \filling{4}{x_\beta} \\
*(green) \filling{4}{x_{\beta - 1}} \\
*(green) \filling{4}{\vdots} \\
*(green) \filling{4}{x_2} \\
*(green) \filling{4}{x_1}
\end{ytableau}} \quad
\varphi_0(T):
\scalebox{0.8}[.25]{
\begin{ytableau}
\filling{4}{x_{\beta + \alpha}} & *(green) \filling{4}{x_\beta} \\
\filling{4}{x_{\beta + \alpha - 1}} & *(green) \filling{4}{x_{\beta - 1}} \\
\filling{4}{\vdots} & *(green) \filling{4}{\vdots} \\
\filling{4}{x_{\beta + 2}} & *(green) \filling{4}{x_3} \\
\filling{4}{x_{\beta + 1}} & *(green) \filling{4}{x_2} \\
*(yellow) \filling{4}{y_\alpha} & *(green) \filling{4}{x_1} \\
*(yellow) \filling{4}{y_{\alpha - 1}} \\
*(yellow) \filling{4}{\vdots} \\
*(yellow) \filling{4}{y_1}
\end{ytableau}} \quad
\varphi_1(T):
\scalebox{0.8}[.25]{
\begin{ytableau}
\filling{4}{x_{\beta + \alpha}} & \filling{4}{\boldsymbol{x_{\beta + 1}}} \\
\filling{4}{x_{\beta + \alpha - 1}} & *(green) \filling{4}{x_{\beta - 1}} \\
\filling{4}{\vdots} & *(green) \filling{4}{\vdots} \\
\filling{4}{x_{\beta + 2}} & *(green) \filling{4}{x_3} \\
*(green) \filling{4}{\boldsymbol{x_\beta}} & *(green) \filling{4}{x_2} \\
*(green) \filling{4}{x_1} & *(yellow) \filling{4}{y_\alpha} \\
*(yellow) \filling{4}{y_{\alpha - 1}} \\
*(yellow) \filling{4}{\vdots} \\
*(yellow) \filling{4}{y_1}
\end{ytableau}}
\end{equation}
We indicate the additional swaps $\varphi_1$ adds to $\varphi_0$ using boldface on the relevant entries, $x_{\beta + 1}$ and $x_\beta$. We will continue 
to use this convention for any subsequent modifications to $\varphi_0$.

Observe that for the $x_{\beta + 1}$ with $x_\beta$ swap to be between cells in different rows, we must have $\alpha \geqslant 2$. This property is 
necessary for the tableau to remain semistandard after the swap.

The intuition behind the swaps is that the swap of $y_\alpha$ with $x_1$ makes the tableau semistandard, and the swap of $x_{\beta + 1}$ with 
$x_{\beta}$ prevents the new tableau from being in the image of $\varphi_0$.

\begin{example}
For $T$ in Example~\ref{example.failure}, $\varphi_1$ maps:
\ytableausetup{boxsize=normal}
\begin{equation}
\begin{ytableau}
1 & *(yellow) 6 \\
2 & *(yellow) 7 \\
*(green) 3 \\
*(green) 4 \\
*(green) 5
\end{ytableau}
\mapsto
\begin{ytableau}
1 & \bf{2} \\
*(green) \bf{3} & *(green) 4 \\
*(green) 5 & *(yellow) 6 \\
*(yellow) 7
\end{ytableau}
\end{equation}
\end{example}

\begin{proposition} 
\label{two column injection 1}
Let $\lambda$ and $\mu$ be as in \eqref{equation.mu lambda} with $\beta = \alpha + 1 \geqslant 3$.
Then $\varphi_1$ as defined above is an injection 
\[
\varphi_1 \colon \mathsf{SSYT}(\lambda) \to \mathsf{SSYT}(\mu).
\]
\end{proposition}
\begin{proof}
Let $T \in \mathsf{SSYT}(\lambda)$. We need to check that $\varphi_1(T)$ is semistandard.
It suffices to do so for the case when $\varphi_1(T) \neq \varphi_0(T)$, where there is an overlap in $\varphi_0(T)$ consisting of a single decreasing pair 
of cells in a row, $y_\alpha > x_1$.
In particular, $\varphi_0(T)$ would be semistandard if it were not for this single pair by the proof  of Proposition~\ref{two column injection 0}, so it suffices to 
check that swapping the $x_1$ with $y_\alpha$ makes the tableau semistandard, and swapping the $x_{\beta + 1}$ with $x_{\beta}$ keeps it semistandard, 
by examining the changed entries.

Swapping the entries $x_1$ with $y_\alpha$ makes the $(\alpha + 1)$-th row weakly increasing since $x_1 < y_\alpha$ by assumption.
The $c$-th column remains strictly increasing since
$x_{\beta + 1} < x_1 < y_\alpha < y_{\alpha - 1}$,
and the $(c + 1)$-th column remains strictly increasing since
$x_2 < x_1 < y_\alpha$.

Swapping the entries $x_\beta$ with $x_{\beta + 1}$ keeps the relevant rows, namely the $1$st row and $\alpha$-th row, weakly increasing and their 
columns strictly increasing since $x_j > x_k$ for all $j < k$.

Specifically, the $c$-th column remains strictly increasing since $x_{\beta + 2} < x_\beta < x_1$, and $(c + 1)$-th column remains strictly increasing 
since $x_{\beta + 1} < x_{\beta - 1}$.
The $1$st row remains weakly increasing because $x_{\beta + 1} < x_\beta$, so $x_{\beta + 1}$ is also less than all entries to its right, which are originally 
right of $x_\beta$.
The $\alpha$-th row remains weakly increasing because $x_{\beta + 1} < x_\beta$, so $x_\beta$ is greater than all entries to its left, which are 
originally left of $x_{\beta + 1}$.

To show injectivity, it suffices to check that the modified tableau cannot be in the image of $\varphi_0$, allowing us to define an explicit inverse $\psi_1$ 
by likewise modifying $\psi_0$.
Namely, we must verify that $\varphi_1(T)$ is not equal to $\varphi_0(S)$ for any $S \in \mathsf{SSYT}(\lambda)$.

Indeed, consider $\varphi_0(S)$ for any $S \in \mathsf{SSYT}(\lambda)$. Let $i$ be as in the definition of $\varphi_0$.
Then the $i$-th entry from the bottom of the $(c + 1)$-th column in $\varphi_0(S)$ must have originally been below and hence greater than 
the $i$-th entry from the bottom of the $c$-th column in $\varphi_0(S)$, which stays in the same place in $\varphi_0(S)$.
That is, if the $j$-th entry from the bottom of the $(c + 1)$-th column in $\varphi_1(T)$ is less than or equal to the $j$-th entry from the bottom of 
the $c$-th column in $\varphi_1(T)$ for all $1 \leqslant j \leqslant \alpha + 1$, then $\varphi_1(T) \neq \varphi_0(S)$ for all $S \in \mathsf{SSYT}(\lambda)$.
If we check these corresponding pairs of entries in $\varphi_1(T)$, we have $y_\alpha < y_1$, $x_j < y_j$ for $2 \leqslant j \leqslant \alpha - 1$, 
$x_{\beta - 1} < x_1$, and $x_{\beta + 1} < x_\beta$. Thus, we do not have a requisite pair of entries, and no $S$ satisfies $\varphi_1(T) = \varphi_0(S)$.

We can now define our explicit inverse $\psi_1$. Let $T' \in \varphi_1(\mathsf{SSYT}(\lambda))$. As before, the entries in the $c$-th column of $T'$ in increasing order are
$x'_{\beta + \alpha - 1}, x'_{\beta + \alpha - 2}, \dots, x'_1$,
and the entries in the $(c + 1)$-th column of $T'$ in increasing order are $y'_{\alpha + 1}, y'_\alpha, \dots, y'_1$.

Let $\psi_1(T') = \psi_0(T')$ when $T' \in \varphi_0(\mathsf{SSYT}(\lambda))$, the domain of $\psi_0$. This occurs when there exists an $i'$ such that $y'_{i'} > x'_{i'}$, so we can take the smallest such $i'$ as in the definition of $\psi_0$.

If such an $i'$ does not exist, then $\psi_1$ first swaps $x'_\beta$ with $y'_{\alpha + 1}$, and $x'_{\beta - 1}$ with $y'_1$, undoing the modifications. 
Relabelling the new tableau obtained after these swaps $T''$, we now let $\psi_1(T') = \psi_0(T'')$.

Concretely, when $\psi_1$ differs from $\psi_0$, the $c$-th and $(c+1)$-th columns look as follows:
\begin{equation}
T':
\ytableausetup{boxsize=6.5em}%
\scalebox{0.8}[.25]{
\begin{ytableau}
\filling{4}{x'_{\beta + \alpha - 1}} & \filling{4}{\boldsymbol{y'_{\alpha + 1}}} \\
\filling{4}{x'_{\beta + \alpha - 2}} & *(green) \filling{4}{y'_\alpha} \\
\filling{4}{\vdots} & *(green) \filling{4}{\vdots} \\
*(green) \filling{4}{\boldsymbol{x'_\beta}} & *(green) \filling{4}{y'_2} \\
*(green) \filling{4}{x'_{\beta - 1}} & *(yellow) \filling{4}{y'_1} \\
*(yellow) \filling{4}{x'_{\beta - 2}} \\
*(yellow) \filling{4}{\vdots} \\
*(yellow) \filling{4}{x'_1}
\end{ytableau}}
\qquad
T'':
\scalebox{0.8}[.25]{
\begin{ytableau}
\filling{4}{x'_{\beta + \alpha - 1}} & *(green) \filling{4}{\boldsymbol{x'_\beta}} \\
\filling{4}{x'_{\beta + \alpha - 2}} & *(green) \filling{4}{y'_\alpha} \\
\filling{4}{\vdots} & *(green) \filling{4}{\vdots} \\
\filling{4}{\boldsymbol{y'_{\alpha + 1}}} & *(green) \filling{4}{y'_2} \\
*(yellow) \filling{4}{y'_1} & *(green) \filling{4}{x'_{\beta - 1}} \\
*(yellow) \filling{4}{x'_{\beta - 2}} \\
*(yellow) \filling{4}{\vdots} \\
*(yellow) \filling{4}{x'_1}
\end{ytableau}}
\qquad
\psi_1(T'):
\scalebox{0.8}[.25]{
\begin{ytableau}
\filling{4}{x'_{\beta + \alpha - 1}} & *(yellow) \filling{4}{y'_1} \\
\filling{4}{x'_{\beta + \alpha - 2}} & *(yellow) \filling{4}{x'_{\beta - 2}} \\
\filling{4}{\vdots} & *(yellow) \filling{4}{\vdots} \\
\filling{4}{y'_{\alpha + 1}} & *(yellow) \filling{4}{x'_1} \\
*(green) \filling{4}{x'_\beta} \\
*(green) \filling{4}{y'_\alpha} \\
*(green) \filling{4}{\vdots} \\
*(green) \filling{4}{y'_2} \\
*(green) \filling{4}{x'_{\beta - 1}}
\end{ytableau}}
\end{equation}

It is straightforward to check that for $T' = \varphi_1(T)$, $\psi_1$ exactly reverses all the swaps done by $\varphi_1$.
\end{proof}

We now obtain stronger versions of the corollaries obtained from the previous injection, in particular Corollary~\ref{corollary.cover two col immersion0}.

\begin{corollary}
\label{corollary.cover two col immersion1}
The partitions $\lambda$ and $\mu$ as in~\eqref{equation.mu lambda} with $\beta \geqslant \alpha + 1 \geqslant 3$ form a cover in the immersion poset.
\end{corollary}

\begin{corollary}
If $\lambda = (a^2, b^{\beta}, c, \dots)$, where $a > b > c$, and $\beta \geqslant 3$, then $\lambda$ is not maximal.
\end{corollary}

\begin{corollary}
If $\lambda = (a, b, c, d, e)$ is maximal in the immersion poset, then it has no more than two identical non-zero parts.
\end{corollary}

In order to further improve the bound for the injection, we must continue to apply modifications to resolve decreasing pairs in ``overlapping'' rows, 
and then apply further modifications to establish injectivity. However, there are now multiple cases to consider.

Firstly, any combination of the overlapping rows containing both $y_{\beta - 2 + j}$ and $x_j$ can be decreasing.

\begin{example}
\label{example.tworowoverlap}
For $\lambda = (2^4, 1^4)$, so $\alpha = \beta = 4$, following $\varphi_0$ can give two overlapping rows. We have each possible combination of rows with decreasing pairs as follows:
\begin{equation}
\ytableausetup{boxsize=normal}
\begin{ytableau}
1 & *(yellow) 7 \\
2 & *(yellow) 10 \\
3 & *(yellow) 11 \\
*(green) 4 & *(yellow) 12 \\
*(green) 5 \\
*(green) 6 \\
*(green) 8 \\
*(green) 9
\end{ytableau}
\mapsto
\begin{ytableau}
1 & *(green) 4 \\
2 & *(green) 5 \\
3 & *(green) 6 \\
*(yellow) 7 & *(green) 8 \\
*(yellow) 10 & *(green) 9 \\
*(yellow) 11 \\
*(yellow) 12
\end{ytableau}
\qquad
\begin{ytableau}
1 & *(yellow) 8 \\
2 & *(yellow) 9 \\
3 & *(yellow) 11 \\
*(green) 4 & *(yellow) 12 \\
*(green) 5 \\
*(green) 6 \\
*(green) 7 \\
*(green) 10
\end{ytableau}
\mapsto
\begin{ytableau}
1 & *(green) 4 \\
2 & *(green) 5 \\
3 & *(green) 6 \\
*(yellow) 8 & *(green) 7 \\
*(yellow) 9 & *(green) 10 \\
*(yellow) 11 \\
*(yellow) 12
\end{ytableau}
\qquad
\begin{ytableau}
1 & *(yellow) 9 \\
2 & *(yellow) 10 \\
3 & *(yellow) 11 \\
*(green) 4 & *(yellow) 12 \\
*(green) 5 \\
*(green) 6 \\
*(green) 7 \\
*(green) 8
\end{ytableau}
\mapsto
\begin{ytableau}
1 & *(green) 4 \\
2 & *(green) 5 \\
3 & *(green) 6 \\
*(yellow) 9 & *(green) 7 \\
*(yellow) 10 & *(green) 8 \\
*(yellow) 11 \\
*(yellow) 12
\end{ytableau}
\end{equation}
\end{example}

While all these previous tableaux give $2$ rows of overlap, it is also possible for a tableau of the same shape to give $0$ or $1$ rows of overlap instead. More generally, $\varphi_0(T)$ for $T$ of shape $\lambda$ as in \eqref{equation.mu lambda} can have anywhere between $0$ and $\max\{0, \alpha - \beta + 2\}$ rows of overlap.

\begin{example}
\label{example.onerowoverlap}
For the same $\lambda = (2^4, 1^4)$ as in Example~\ref{example.tworowoverlap}, $\varphi_0$ can give a single overlapping row, which contains a decreasing pair:
\begin{equation}
\begin{ytableau}
1 & 5 \\
2 & *(yellow) 10 \\
3 & *(yellow) 11 \\
4 & *(yellow) 12 \\
*(green) 6 \\
*(green) 7 \\
*(green) 8 \\
*(green) 9
\end{ytableau}
\mapsto
\begin{ytableau}
1 & 5 \\
2 & *(green) 6 \\
3 & *(green) 7 \\
4 & *(green) 8 \\
*(yellow) 10 & *(green) 9 \\
*(yellow) 11 \\
*(yellow) 12
\end{ytableau}
\end{equation}
\end{example}

Hence, in our next modifications of $\varphi_0$, we must encode the information of every possible case in a way that both distinguishes the cases 
from $\varphi_0$ with no modifications, and distinguishes the cases from each other.
To achieve this, our modifications will involve cyclically rotating certain entries in the $c$-th and $(c + 1)$-th columns. These rotations will be 
analogous to the $x_\beta$ with $x_{\beta + 1}$ swap in $\varphi_1$, which can be thought of as a rotation of $2$ elements.

We now define a second set of modifications of our original $\varphi_0$ injection for the case when $\beta = \alpha$, and $\alpha \geqslant 4$, 
which we call 
\[
	\varphi_2 \colon \mathsf{SSYT}(\lambda) \to \mathsf{SSYT}(\mu).
\]

Let $T \in \mathsf{SSYT}(\lambda)$. As before, suppose that the entries in the $c$-th column of $T$ in increasing order are 
$x_{\beta + \alpha}, x_{\beta + \alpha - 1}, \ldots, x_1$, and the entries in the $(c + 1)$-th column of $T$ in increasing order are 
$y_\alpha, y_{\alpha - 1}, \dots, y_1$.
We define $\varphi_2(T)$ to be the same as $\varphi_0(T)$ if $\varphi_0(T) \in \mathsf{SSYT}(\mu)$.

If $\varphi_0(T) \not \in \mathsf{SSYT}(\mu)$, then we have several cases. 
If there are two rows of overlap, then $i$ as defined for $\varphi_0$ is $\alpha + 1$, $x_1$ is to the right of $y_{\alpha - 1}$, and $x_2$ is to the right of 
$y_\alpha = y_\beta$:
\begin{equation}
T:
\ytableausetup{boxsize=5em}%
\scalebox{1}[.25]{
\begin{ytableau}
\filling{4}{x_{\beta + \alpha}} & *(yellow) \filling{4}{y_\alpha} \\
\filling{4}{x_{\beta + \alpha - 1}} & *(yellow) \filling{4}{y_{\alpha - 1}} \\
\filling{4}{x_{\beta + \alpha - 2}} & *(yellow) \filling{4}{y_{\alpha - 2}} \\
\filling{4}{\vdots} & *(yellow) \filling{4}{\vdots} \\
\filling{4}{x_{\beta + 3}} & *(yellow) \filling{4}{y_3} \\
\filling{4}{x_{\beta + 2}} & *(yellow) \filling{4}{y_2} \\
*(green) \filling{4}{x_{\beta + 1}} & *(yellow) \filling{4}{y_1} \\
*(green) \filling{4}{x_\beta} \\
*(green) \filling{4}{x_{\beta - 1}} \\
*(green) \filling{4}{\vdots} \\
*(green)\filling{4}{x_2} \\
*(green) \filling{4}{x_1}
\end{ytableau}} \quad
\varphi_0(T):
\scalebox{1}[.25]{
\begin{ytableau}
\filling{4}{x_{\beta + \alpha}} & *(green) \filling{4}{x_{\beta + 1}} \\
\filling{4}{x_{\beta + \alpha - 1}} & *(green) \filling{4}{x_{\beta}} \\
\filling{4}{x_{\beta + \alpha - 2}} & *(green) \filling{4}{x_{\beta - 1}} \\
\filling{4}{\vdots} & *(green) \filling{4}{\vdots} \\
\filling{4}{x_{\beta + 3}} & *(green) \filling{4}{x_4} \\
\filling{4}{x_{\beta + 2}} & *(green) \filling{4}{x_3} \\
*(yellow) \filling{4}{y_\alpha} & *(green) \filling{4}{x_2} \\
*(yellow) \filling{4}{y_{\alpha - 1}} & *(green) \filling{4}{x_1} \\
*(yellow) \filling{4}{y_{\alpha - 2}} \\
*(yellow) \filling{4}{\vdots} \\
*(yellow) \filling{4}{y_1}
\end{ytableau}}
\end{equation}

If $y_{\alpha - 1} > x_1$ and $y_\alpha \leqslant x_2$, then we swap $y_{\alpha - 1}$ with $x_1$. We also ``clockwise rotate" the entries $x_{\beta + 2}$ and $x_{\beta + 3}$ in the $c$-th column, and $x_{\beta + 1}$ in the $(c + 1)$-th column, as shown in our next diagram.

In our definition of $\varphi_2$, a clockwise rotation of a set of entries in the $c$-th and $(c + 1)$-th columns moves all entries in the $c$-th column up one cell except the topmost entry, which moves to the topmost cell in the $(c + 1)$-th column containing an entry being rotated. The rotation moves all entries in the $(c + 1)$-th column down one cell except the bottommost entry, which moves to the bottommost cell in the $c$-th column containing an entry being rotated. As another example, a rotation of a single entry in the $c$-th column and a single entry in the $(c + 1)$-th column is a swap of those entries. We will continue to describe all cases of $\varphi_2$ with rotations of different sets of entries.

If $y_{\alpha - 1} \leqslant x_1$ and $y_\alpha > x_2$, then we swap $y_\alpha$ with $x_2$. We also clockwise rotate $x_{\beta + 2}$ in the $c$-th column, and $x_{\beta + 1}$ and $x_\beta$ in the $(c + 1)$-th column.

If $y_{\alpha - 1} > x_1$ and $y_\alpha > x_2$, then we swap both $y_{\alpha - 1}$ with $x_1$ and $y_\alpha$ with $x_2$. We also clockwise rotate $x_{\beta + 2}$ and $x_{\beta + 3}$ in the $c$-th column, and $x_{\beta + 1}$ and $x_\beta$ in the $(c + 1)$-th column.

Concretely, the two row overlap cases are as follows:
\begin{align}
\varphi_2(T):
& \scalebox{1}[.25]{
\begin{ytableau}
\filling{4}{x_{\beta + \alpha}} & \filling{4}{\boldsymbol{x_{\beta + 3}}} \\
\filling{4}{x_{\beta + \alpha - 1}} & *(green) \filling{4}{x_\beta} \\
\filling{4}{x_{\beta + \alpha - 2}} & *(green) \filling{4}{x_{\beta - 1}} \\
\filling{4}{\vdots} & *(green) \filling{4}{\vdots} \\
\filling{4}{\boldsymbol{x_{\beta + 2}}} & *(green) \filling{4}{x_4} \\
*(green) \filling{4}{\boldsymbol{x_{\beta + 1}}} & *(green) \filling{4}{x_3} \\
*(yellow) \filling{4}{y_\alpha} & *(green) \filling{4}{x_2} \\
*(green) \filling{4}{x_1} & *(yellow) \filling{4}{y_{\alpha - 1}} \\
*(yellow) \filling{4}{y_{\alpha - 2}} \\
*(yellow) \filling{4}{\vdots} \\
*(yellow) \filling{4}{y_1}
\end{ytableau}}
& \scalebox{1}[.25]{
\begin{ytableau}
\filling{4}{x_{\beta + \alpha}} & \filling{4}{\boldsymbol{x_{\beta + 2}}} \\
\filling{4}{x_{\beta + \alpha - 1}} & *(green) \filling{4}{\boldsymbol{x_{\beta + 1}}} \\
\filling{4}{x_{\beta + \alpha - 2}} & *(green) \filling{4}{x_{\beta - 1}} \\
\filling{4}{\vdots} & *(green) \filling{4}{\vdots} \\
\filling{4}{x_{\beta + 3}} & *(green) \filling{4}{x_4} \\
*(green) \filling{4}{\boldsymbol{x_\beta}} & *(green) \filling{4}{x_3} \\
*(green) \filling{4}{x_2} & *(yellow) \filling{4}{y_\alpha} \\
*(yellow) \filling{4}{y_{\alpha - 1}} & *(green) \filling{4}{x_1} \\
*(yellow) \filling{4}{y_{\alpha - 2}} \\
*(yellow) \filling{4}{\vdots} \\
*(yellow) \filling{4}{y_1}
\end{ytableau}}
\hspace{2.5em}
& \scalebox{1}[.25]{
\begin{ytableau}
\filling{4}{x_{\beta + \alpha}} & \filling{4}{\boldsymbol{x_{\beta + 3}}} \\
\filling{4}{x_{\beta + \alpha - 1}} & *(green) \filling{4}{\boldsymbol{x_{\beta + 1}}} \\
\filling{4}{x_{\beta + \alpha - 2}} & *(green) \filling{4}{x_{\beta - 1}} \\
\filling{4}{\vdots} & *(green) \filling{4}{\vdots} \\
\filling{4}{\boldsymbol{x_{\beta + 2}}} & *(green) \filling{4}{x_4} \\
*(green) \filling{4}{\boldsymbol{x_{\beta}}} & *(green) \filling{4}{x_3} \\
*(green) \filling{4}{x_2} & *(yellow) \filling{4}{y_\alpha} \\
*(green) \filling{4}{x_1} & *(yellow) \filling{4}{y_{\alpha - 1}} \\
*(yellow) \filling{4}{y_{\alpha - 2}} \\
*(yellow) \filling{4}{\vdots} \\
*(yellow) \filling{4}{y_1}
\end{ytableau}}
\\
& & & \nonumber
\\
& y_{\alpha - 1} > x_1, \quad y_\alpha \leqslant x_2
& y_{\alpha - 1} \leqslant x_1, \quad y_\alpha > x_2 \hspace{5em}
& y_{\alpha - 1} > x_1, \quad y_\alpha > x_2 \nonumber
\end{align}

Consider the entries involved in the clockwise rotation in each case.
For the topmost entry in the $c$-th column to move strictly up to the $(c + 1)$-th column, and the bottommost entry in the $(c + 1)$-th column move strictly down to the $c$-th column, we must have $\alpha \geqslant 3$.
This property is necessary for the tableau to remain semistandard after the rotation, which partially necessitates the $\alpha \geqslant 4$ assumption, which is analogous to the $\alpha \geqslant 2$ assumption for $\varphi_1$.
 
If there is one row of overlap, then $i$ as defined for $\varphi_0$ is $\alpha$, and $x_1$ is to the right of $y_{\alpha - 1}$ with $x_1 < y_{\alpha - 1}$:
\begin{equation}
T:
\ytableausetup{boxsize=6em}%
\scalebox{0.8}[.25]{
\begin{ytableau}
\filling{4}{x_{\beta + \alpha}} & \filling{4}{y_\alpha} \\
\filling{4}{x_{\beta + \alpha - 1}} & *(yellow) \filling{4}{y_{\alpha - 1}} \\
\filling{4}{x_{\beta + \alpha - 2}} & *(yellow) \filling{4}{y_{\alpha - 2}} \\
\filling{4}{\vdots} & *(yellow) \filling{4}{\vdots} \\
\filling{4}{x_{\beta + 3}} & *(yellow) \filling{4}{y_3} \\
\filling{4}{x_{\beta + 2}} & *(yellow) \filling{4}{y_2} \\
\filling{4}{x_{\beta + 1}} & *(yellow) \filling{4}{y_1} \\
*(green) \filling{4}{x_\beta} \\
*(green) \filling{4}{x_{\beta - 1}} \\
*(green) \filling{4}{\vdots} \\
*(green) \filling{4}{x_2} \\
*(green) \filling{4}{x_1}
\end{ytableau}}
\quad
\varphi_0(T):
\scalebox{0.8}[.25]{
\begin{ytableau}
\filling{4}{x_{\beta + \alpha}} & \filling{4}{y_\alpha} \\
\filling{4}{x_{\beta + \alpha - 1}} & *(green) \filling{4}{x_\beta} \\
\filling{4}{x_{\beta + \alpha - 2}} & *(green) \filling{4}{x_{\beta - 1}} \\
\filling{4}{\vdots} & *(green) \filling{4}{\vdots} \\
\filling{4}{x_{\beta + 3}} & *(green) \filling{4}{x_4} \\
\filling{4}{x_{\beta + 2}} & *(green) \filling{4}{x_3} \\
\filling{4}{x_{\beta + 1}} & *(green) \filling{4}{x_2} \\
*(yellow) \filling{4}{y_{\alpha - 1}} & *(green) \filling{4}{x_1} \\
*(yellow) \filling{4}{y_{\alpha - 2}} \\
*(yellow) \filling{4}{\vdots} \\
*(yellow) \filling{4}{y_1}
\end{ytableau}}
\end{equation}

In this case, we only have the single pair of decreasing entries, $y_{\alpha - 1} < x_1$, so we swap $y_{\alpha - 1}$ with $x_1$. However, for the 
additional modifications after and in addition to this swap, we have different subcases.

If $y_\alpha < x_{\beta + 2}$, we clockwise rotate $x_{\beta + 1}$ and $x_{\beta + 2}$ in the $c$-th column, and $x_\beta$ in the $(c + 1)$-th column.

If $x_{\beta + 2} \leqslant y_\alpha < x_{\beta + 1}$, we swap $x_{\beta + 2}$ with $y_\alpha$, and $x_{\beta + 1}$ with $x_\beta$. Observe in particular 
that this subcase is two separate swaps, and not a rotation.

If $x_{\beta + 1} \leqslant y_\alpha$, we clockwise rotate $x_{\beta + 1}$ and $x_{\beta + 2}$ in the $c$-th column, and $y_\alpha$, $x_\beta$, and 
$x_{\beta - 1}$ in the $(c + 1)$-th column.
Concretely, the one row overlap cases are as follows:
\begin{align}
\varphi_2(T):
& \scalebox{0.8}[.25]{
\begin{ytableau}
\filling{4}{x_{\beta + \alpha}} & \filling{4}{y_\alpha} \\
\filling{4}{x_{\beta + \alpha - 1}} & \filling{4}{\boldsymbol{x_{\beta + 2}}} \\
\filling{4}{x_{\beta + \alpha - 2}} & *(green) \filling{4}{x_{\beta - 1}} \\
\filling{4}{\vdots} & *(green) \filling{4}{\vdots} \\
\filling{4}{x_{\beta + 3}} & *(green) \filling{4}{x_4} \\
\filling{4}{\boldsymbol{x_{\beta + 1}}} & *(green) \filling{4}{x_3} \\
*(green) \filling{4}{\boldsymbol{x_\beta}} & *(green) \filling{4}{x_2} \\
*(green) \filling{4}{x_1} & *(yellow) \filling{4}{y_{\alpha - 1}} \\
*(yellow) \filling{4}{y_{\alpha - 2}} \\
*(yellow) \filling{4}{\vdots} \\
*(yellow) \filling{4}{y_1}
\end{ytableau}} \quad
& \scalebox{0.8}[.25]{
\begin{ytableau}
\filling{4}{x_{\beta + \alpha}} & \filling{4}{\boldsymbol{x_{\beta + 2}}} \\
\filling{4}{x_{\beta + \alpha - 1}} & \filling{4}{\boldsymbol{x_{\beta + 1}}} \\
\filling{4}{x_{\beta + \alpha - 2}} & *(green) \filling{4}{x_{\beta - 1}} \\
\filling{4}{\vdots} & *(green) \filling{4}{\vdots} \\
\filling{4}{x_{\beta + 3}} &  *(green)\filling{4}{x_4} \\
\filling{4}{\boldsymbol{y_\alpha}} & *(green) \filling{4}{x_3} \\
*(green) \filling{4}{\boldsymbol{x_\beta}} & *(green) \filling{4}{x_2} \\
*(green) \filling{4}{x_1} & *(yellow) \filling{4}{y_{\alpha - 1}} \\
*(yellow) \filling{4}{y_{\alpha - 2}} \\
*(yellow) \filling{4}{\vdots} \\
*(yellow) \filling{4}{y_1}
\end{ytableau}} \quad
& \scalebox{0.8}[.25]{
\begin{ytableau}
\filling{4}{x_{\beta + \alpha}} & \filling{4}{\boldsymbol{x_{\beta + 2}}} \\
\filling{4}{x_{\beta + \alpha - 1}} & \filling{4}{\boldsymbol{y_\alpha}} \\
\filling{4}{x_{\beta + \alpha - 2}} & *(green) \filling{4}{\boldsymbol{x_\beta}} \\
\filling{4}{\vdots} & *(green) \filling{4}{\vdots} \\
\filling{4}{x_{\beta + 3}} & *(green) \filling{4}{x_4} \\
\filling{4}{\boldsymbol{x_{\beta + 1}}} & *(green) \filling{4}{x_3} \\
*(green) \filling{4}{\boldsymbol{x_{\beta - 1}}} & *(green) \filling{4}{x_2} \\
*(green) \filling{4}{x_1} & *(yellow) \filling{4}{y_{\alpha - 1}} \\
*(yellow) \filling{4}{y_{\alpha - 2}} \\
*(yellow) \filling{4}{\vdots} \\
*(yellow) \filling{4}{y_1}
\end{ytableau}}
\\
& & & \nonumber
\\
& y_\alpha < x_{\beta + 2}
& x_{\beta + 2} \leqslant y_\alpha < x_{\beta + 1} \hspace{6em}
& x_{\beta + 1} \leqslant y_\alpha \nonumber
\end{align}

Again, consider the entries involved in the modifications in each case, either the rotations when $y_\alpha < x_{\beta + 2}$ or 
$x_{\beta + 1} \leqslant y_\alpha$, or the swaps when $x_{\beta +2} \leqslant y_\alpha < x_{\beta + 1}$.
For the topmost entry in the $c$-th column to move strictly up to the $(c + 1)$-th column, and the bottommost entry in the $(c + 1)$-th column to move 
strictly down to the $c$-th column, we must have $\alpha \geqslant 4$.
This property is necessary for the tableau to remain semistandard after the rotation, as we will see in Lemma~\ref{cyclic rotation}, which fully necessitates 
the $\alpha \geqslant 4$ assumption.

\begin{example}
For $\lambda = (2^4, 1^4)$ and $T$ from Example~\ref{example.tworowoverlap}, we get all the two row overlap cases of $\varphi_2$:
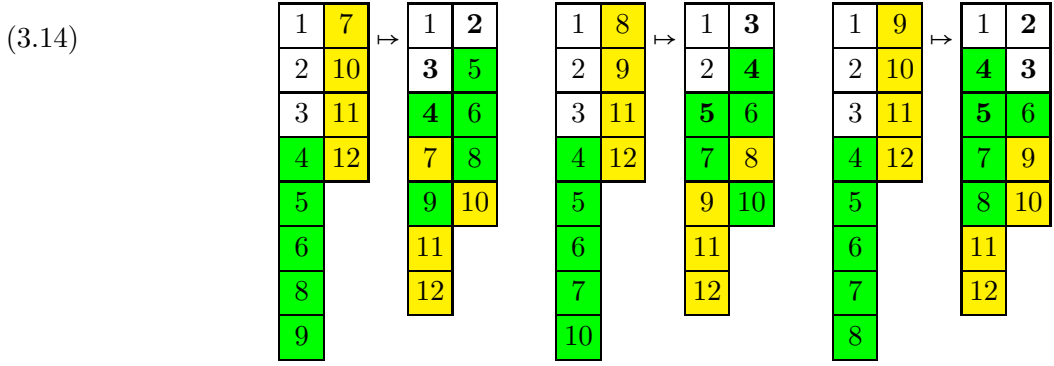
\begin{equation}
\ytableausetup{boxsize=normal}
\begin{ytableau}
1 & *(yellow) 7 \\
2 & *(yellow) 10 \\
3 & *(yellow) 11 \\
*(green) 4 & *(yellow) 12 \\
*(green) 5 \\
*(green) 6 \\
*(green) 8 \\
*(green) 9
\end{ytableau}
\mapsto
\begin{ytableau}
1 & \bf{2} \\
\bf{3} & *(green) 5 \\
*(green) \bf{4} & *(green) 6 \\
*(yellow) 7 & *(green) 8 \\
*(green) 9 & *(yellow) 10 \\
*(yellow) 11 \\
*(yellow) 12
\end{ytableau}
\qquad
\begin{ytableau}
1 & *(yellow) 8 \\
2 & *(yellow) 9 \\
3 & *(yellow) 11 \\
*(green) 4 & *(yellow) 12 \\
*(green) 5 \\
*(green) 6 \\
*(green) 7 \\
*(green) 10
\end{ytableau}
\mapsto
\begin{ytableau}
1 & \bf{3} \\
2 & *(green) \bf{4} \\
*(green) \bf{5} & *(green) 6 \\
*(green) 7 & *(yellow) 8 \\
*(yellow) 9 & *(green) 10 \\
*(yellow) 11 \\
*(yellow) 12
\end{ytableau}
\qquad
\begin{ytableau}
1 & *(yellow) 9 \\
2 & *(yellow) 10 \\
3 & *(yellow) 11 \\
*(green) 4 & *(yellow) 12 \\
*(green) 5 \\
*(green) 6 \\
*(green) 7 \\
*(green) 8
\end{ytableau}
\mapsto
\begin{ytableau}
1 & \bf{2} \\
*(green) \bf{4} & \bf{3} \\
*(green) \bf{5} & *(green) 6 \\
*(green) 7 & *(yellow) 9 \\
*(green) 8 & *(yellow) 10 \\
*(yellow) 11 \\
*(yellow) 12
\end{ytableau}
\end{equation}
For the same $\lambda$, we have all the one row overlap cases of $\varphi_2$ as follows, including the $T$ from Example~\ref{example.onerowoverlap}:
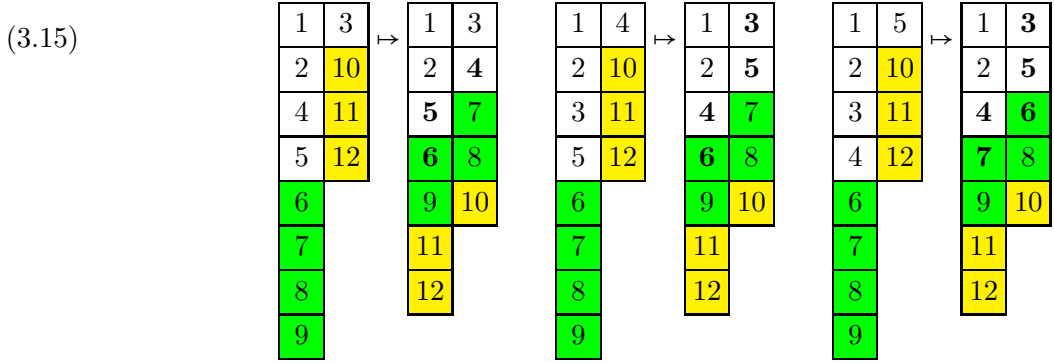
\begin{equation}
\begin{ytableau}
1 & 3 \\
2 & *(yellow) 10 \\
4 & *(yellow) 11 \\
5 & *(yellow) 12 \\
*(green) 6 \\
*(green) 7 \\
*(green) 8 \\
*(green) 9
\end{ytableau}
\mapsto
\begin{ytableau}
1 & 3 \\
2 & \bf{4} \\
\bf{5} & *(green) 7 \\
*(green) \bf{6} & *(green) 8 \\
*(green) 9 & *(yellow) 10 \\
*(yellow) 11 \\
*(yellow) 12
\end{ytableau}
\qquad
\begin{ytableau}
1 & 4 \\
2 & *(yellow) 10 \\
3 & *(yellow) 11 \\
5 & *(yellow) 12 \\
*(green) 6 \\
*(green) 7 \\
*(green) 8 \\
*(green) 9
\end{ytableau}
\mapsto
\begin{ytableau}
1 & \bf{3} \\
2 & \bf{5} \\
\bf{4} & *(green) 7 \\
*(green) \bf{6} & *(green) 8 \\
*(green) 9 & *(yellow) 10 \\
*(yellow) 11 \\
*(yellow) 12
\end{ytableau}
\qquad
\begin{ytableau}
1 & 5 \\
2 & *(yellow) 10 \\
3 & *(yellow) 11 \\
4 & *(yellow) 12 \\
*(green) 6 \\
*(green) 7 \\
*(green) 8 \\
*(green) 9
\end{ytableau}
\mapsto
\begin{ytableau}
1 & \bf{3} \\
2 & \bf{5} \\
\bf{4} & *(green) \bf{6} \\
*(green) \bf{7} & *(green) 8 \\
*(green) 9 & *(yellow) 10 \\
*(yellow) 11 \\
*(yellow) 12
\end{ytableau}
\end{equation}
\end{example}

The proof that $\varphi_2$ is an injection of semistandard tableaux relies on the following lemma regarding the clockwise rotation.
\begin{lemma}
\label{cyclic rotation}
Suppose $S$ is a semistandard tableau. Suppose that we perform a clockwise rotation on elements in the $c$-th and $(c + 1)$-th columns of $S$ to obtain $S'$, such that the following are true:
\begin{enumerate}
    \item The bottommost rotated entry in the $c$-th column in $S$ is less than or equal to the topmost rotated entry in the $(c + 1)$-th column in $S$.
    \item The entry moving from the $c$-th column in $S$ to the $(c + 1)$-th column in $S'$ moves strictly upwards.
    \item The entry moving from the $(c + 1)$-th column in $S$ to the $c$-th column in $S'$ moves strictly downward.
    \item The entry moving from the $c$-th column in $S$ to the $(c + 1)$-th column in $S'$ is greater than the entry above it in the $(c + 1)$-th column in $S'$, if such an entry exists.
    \item The entry moving from the $(c + 1)$-th column to the $c$-th column is less than the entry below it in the $c$-th column in $S'$.
\end{enumerate}
Then $S'$ is semistandard.
\end{lemma}

\begin{proposition} 
\label{two column injection 2}
Let $\lambda$ and $\mu$ be as in \eqref{equation.mu lambda} with $\beta = \alpha \geqslant 4$.
Then $\varphi_2$ as defined above is an injection 
\[
\varphi_2 \colon \mathsf{SSYT}(\lambda) \to \mathsf{SSYT}(\mu).
\]
\end{proposition}

The proof of Proposition~\ref{two column injection 2} is technical and omitted here. It follows similar ideas to the proof of Proposition~\ref{two column injection 1}.
We can now further improve upon Corollary~\ref{corollary.cover two col immersion0} and Corollary~\ref{corollary.cover two col immersion1}.

\begin{corollary}
\label{corollary.cover two col immersion2}
The partitions $\lambda$ and $\mu$ as in~\eqref{equation.mu lambda} with $\beta \geqslant \alpha \geqslant 4$ form a cover in the immersion poset.
\end{corollary}

We summarize the bounds on $\alpha$ and $\beta$ needed for each map to be an injection:
\begin{center}
\begin{tabular}{ |c|c|c| }
    \hline
    Map & $\alpha$ & $\beta$ \\
    \hline
    $\varphi_0$ & $\alpha \geqslant 0$ & $\beta \geqslant \alpha + 2$ \\
    \hline
    $\varphi_1$ & $\alpha \geqslant 2$ & $\beta = \alpha + 1$ \\
    \hline
    $\varphi_2$ & $\alpha \geqslant 4$ & $\beta = \alpha$ \\
    \hline
\end{tabular}
\end{center}

\subsection{Immersion poset on hook partitions}
\label{section.hooks}

For this section, set $\lambda^{i}=(i, 1^{n-i})\vdash n$ and let $S=\{\lambda^{i} \mid 1\leqslant i \leqslant n\}$ be the set of all hook partitions.
We study the immersion poset restricted to $S$.

\begin{proposition}
        \label{proposition.hook}
        Let $1 \leqslant i \leqslant n$ and $\alpha = (\alpha_1,\dots,\alpha_k) \vdash n$ such that 
        $\alpha \leqslant_D \lambda^i$. Then
        \[
        K_{\lambda^i,\alpha} = \binom{k-1}{n-i}.
        \]
\end{proposition}

\begin{proof}
        Since $\lambda^i$ dominates $\alpha$, we know that $K_{\lambda^i,\alpha}\geqslant 1$. 
        To form a semistandard Young tableau of shape $\lambda^i$ and content $\alpha$, the $\alpha_1$ entries 1 must be placed leftmost in the first row of 
        $\lambda^i$. The remaining $n-i$ positions in the first column of $\lambda^i$ can be filled with distinct values from the set $\{2,3,\ldots,k\}$.
        This gives $\binom{k-1}{n-i}$ choices. Once these are placed, there is only one way to fill the remainder of the first row so that the resulting tableau is semistandard. 
\end{proof}

Recall from Lemma~\ref{lemma.immersion K} that $\mu \leqslant_I \lambda$ if and only if $K_{\mu,\alpha} \leqslant K_{\lambda,\alpha}$ for all 
$\alpha\vdash n$. Hence with Proposition~\ref{proposition.hook}, we are now ready to describe all the relations between hook partitions $\lambda^i\in S$ 
in the immersion poset. To illustrate what the proposition implies, we form a matrix of values in the following way:
\begin{itemize}
        \item The $j$-th column is indexed by the content $\alpha^{j}$, where $\alpha^{j}$ is any content that has $j$ parts.
        \item The $i$-th row is indexed by the shape $\lambda^{i}$.
        \item The $(i,j)$ entry of this matrix is the value $T_{i,j}:=K_{\lambda^{i},\alpha^j} = \binom{j-1}{n-i}$ for $1\leqslant i,j\leqslant n$.
\end{itemize}

\begin{example}
        \label{example.n=7}
        We give the matrix for $n=7$:
        \begin{center}
                {\renewcommand{\arraystretch}{1.3}
                \begin{tabular}{ |l|c|c|c|c|c|c|c| }
                        \hline
                        \diagbox{Partition}{$\#$ of parts} & 1 & 2 & 3 & 4 & 5 & 6 & 7 \\
                        \hline
                        $(1^7)$ & 0 & 0 & 0 & 0 & 0 & 0 & 1 \\
                        \hline
                        $(2, 1^5)$ & 0 & 0 & 0 & 0 & 0 & 1 & 6 \\
                        \hline
                        $(3, 1^4)$ & 0 & 0 & 0 & 0 & 1 & 5 & 15 \\
                        \hline
                        $(4, 1^3)$ & 0 & 0 & 0 & 1 & 4 & 10 & 20 \\
                        \hline
                        $(5, 1^2)$ & 0 & 0 & 1 & 3 & 6 & 10 & 15 \\
                        \hline
                        $(6,1)$ & 0 & 1 & 2 & 3 & 4 & 5 & 6 \\
                        \hline
                        $(7)$ & 1 & 1 & 1 & 1 & 1 & 1 & 1 \\
                        \hline
                \end{tabular}}
        \end{center}
\end{example}

\begin{remark}
        \label{remark.hook}
        In this context, $\lambda^i \geqslant_I \lambda^j$ if and only if $T_{i,m} \geqslant T_{j,m}$ for all $m$. 
        Equivalently, since $T_{j,m}= 0$ when $n-j \geqslant m$ and $\lambda^i$ dominates $\lambda^j$ when $i>j$, we need only show 
        $T_{i,m} \geqslant T_{j,m}$ for all $m> n-j$ when $i>j$.
\end{remark}
The following lemma is used to prove the structure of the immersion poset restricted to hook partitions. 

\begin{lemma}
        \label{lemma.binom}
        Suppose $\binom{n-1}{n-i}\geqslant \binom{n-1}{n-j}$ and $i > j$ (note that this implies $j\leqslant \frac{n}{2}$). Then for all $0\leqslant p\leqslant j-1$, we have
        \[ 
        \binom{n-1-p}{n-i}\geqslant \binom{n-1-p}{n-j}.
        \]
\end{lemma}

The proof follows from basic properties of binomial coefficients, and is omitted here.

\begin{corollary}
\label{corollary.T to I}
If $T_{i,n} \geqslant T_{j,n}$ for $i>j$, then $\lambda^{i} \geqslant_I \lambda^{j}$.
\end{corollary}

\begin{proof}
By Proposition~\ref{proposition.hook}, $T_{i, n-p} = K_{\lambda^i, \alpha^{n-p}} = \binom{(n-p) - 1}{n-i}$. Hence if $T_{i,n} \geqslant T_{j,n}$, by Lemma~\ref{lemma.binom},
we also have $T_{i,n-p} \geqslant T_{j,n-p}$ for $0\leqslant p\leqslant j-1$. By Remark~\ref{remark.hook}, this implies $\lambda^{i} \geqslant_I \lambda^{j}$.
\end{proof}

\begin{example}
        Take the rows corresponding to the partitions $(5, 1^2)$ and $(3,1^4)$ in Example~\ref{example.n=7}. Since the last column entries 
        give $T_{5,7} = 15\geqslant 15 = T_{3,7}$, then by Corollary~\ref{corollary.T to I} we also have $T_{5,7-p} \geqslant T_{3,7-p}$ for $1 \leqslant p \leqslant 2$: 
        $T_{5,6} = 10\geqslant 5 = T_{3,6}$, $T_{5,5} = 6\geqslant 1 = T_{3,5}$.
\end{example}

We now describe the relations in the immersion poset on $S$ depending upon whether $n$ is even or odd.

\begin{proposition}
        \label{proposition.hook odd}
        Let $n=2k+1$ be odd, then:
        \begin{enumerate}
                \item $\lambda^{\ell+1} \geqslant_I \lambda^{\ell}$ for all $1\leqslant \ell\leqslant k$.
                \item $(\lambda^{k+1-\ell})^t = \lambda^{k+1+\ell} \geqslant_I \lambda^{k+1-\ell}$ for all $1\leqslant \ell\leqslant k$.
                \item For any $1<i\leqslant k+1$, $\lambda^i$ is incomparable to $\lambda^j$ for all $j>n-i+1$.
                \item For any $k+2\leqslant i < n$, $\lambda^i$ is incomparable to $\lambda^j$ for all $j>i$.
        \end{enumerate}
        These describe all relations in the immersion poset restricted to hook partitions $S$.
\end{proposition}

\begin{proof}
        Let us first prove (1). Fix an $\ell$ with $1\leqslant \ell \leqslant k$. Then by Corollary~\ref{corollary.T to I},
        $\lambda^{\ell+1} \geqslant_I \lambda^{\ell}$ if and only if $T_{\ell+1,n} \geqslant T_{\ell,n}$.
        Note that $T_{\ell+1,n} = \binom{n-1}{n-\ell-1} = \binom{n-1}{\ell}$ and $T_{\ell,n} = \binom{n-1}{n-\ell}=\binom{n-1}{\ell-1}$.
        Since $1\leqslant \ell \leqslant k$, we have $\binom{n-1}{\ell}\geqslant\binom{n-1}{\ell-1}$ and the result follows.

        To prove (2), note that $\lambda^{k+1+\ell} \geqslant_I \lambda^{k+1-\ell}$ if and only if $T_{k+1+\ell,n} \geqslant T_{k+1-\ell,n}$.
        Since $T_{k+1+\ell,n} = \binom{n-1}{k-\ell} = \binom{2k}{k-\ell} = \binom{2k}{k+\ell} = \binom{n-1}{k+\ell} = T_{k+1-\ell,n}$, the result follows.

        To prove (3) we show for any $1<i\leqslant k+1$ that $\lambda^i$ is incomparable to $\lambda^j$ for all $j>n-i+1$. Since $\lambda^j$ dominates $\lambda^i$, 
        we need only show there exists some $\alpha$ such that $K_{\lambda^i,\alpha} > K_{\lambda^j,\alpha}$. Choose $\alpha=(1^n)$. Then 
        $K_{\lambda^i,\alpha} = \binom{n-1}{n-i} = \binom{n-1}{i-1} > \binom{n-1}{n-j} = K_{\lambda^j,\alpha}$ because $i-1>n-j$ and $1<i\leqslant k+1$.

        Lastly, to prove (4) we follow the same strategy as (3). Since $\lambda^j$ dominates $\lambda^i$, we can let $\alpha=(1^n)$, and since $k+1<i<j$ we get $K_{\lambda^i,\alpha} = \binom{n-1}{n-i} > \binom{n-1}{n-j} = K_{\lambda^j,\alpha}$, and the result follows.
\end{proof}

\begin{proposition}
        \label{proposition.hook even}
        Let $n=2k$ be even, then:
        \begin{enumerate}
                \item $\lambda^{\ell+1} \geqslant_I \lambda^{\ell}$ for all $1\leqslant \ell< k$.
                \item $(\lambda^{k-\ell})^t = \lambda^{k+1+\ell} \geqslant_I \lambda^{k-\ell}$ for all $0\leqslant \ell\leqslant k-1$.
                \item For any $1<i\leqslant k$, $\lambda^i$ is incomparable to $\lambda^j$ for all $j>n-i+1$.
                \item For any $k+1\leqslant i < n$, $\lambda^i$ is incomparable to $\lambda^j$ for all $j>i$.
        \end{enumerate}
        These describe all relations in the immersion poset restricted to hook partitions $S$.
\end{proposition}

The proof of the even case is similar to the odd case.

\begin{figure}[t]
        \begin{tikzpicture}
                \matrix (a) [matrix of math nodes, column sep=0.6cm, row sep=0.5cm]{
                        (k+1,1^k) & (k+2, 1^{k-1}) \\
                        (k,1^{k+1}) & \vdots \\
                        \vdots & (2k, 1) \\
                        (2,1^{2k-1}) & (2k+1)\\
                        (1^{2k+1}) &  \\};

                \foreach \i/\j in {1-1/2-1, 2-1/3-1, 3-1/4-1, 4-1/5-1, %
                        1-2/2-1, 3-2/4-1, 4-2/5-1}
                \draw (a-\i) -- (a-\j);
        \end{tikzpicture}
        \hspace{2cm}
        \begin{tikzpicture}
                \matrix (a) [matrix of math nodes, column sep=0.6cm, row sep=0.5cm]{
                        & (k+1, 1^{k-1}) \\
                        (k,1^{k}) & \vdots \\
                        \vdots & (2k-1, 1) \\
                        (2,1^{2k-2}) & (2k)\\
                        (1^{2k}) &  \\};

                \foreach \i/\j in {2-1/3-1, 3-1/4-1, 4-1/5-1, %
                        1-2/2-1, 3-2/4-1, 4-2/5-1}
                \draw (a-\i) -- (a-\j);
        \end{tikzpicture}
\caption{Immersion poset restricted to hook partitions for $n=2k+1$ (left) and $n=2k$ (right).
\label{figure.hook}}
\end{figure}
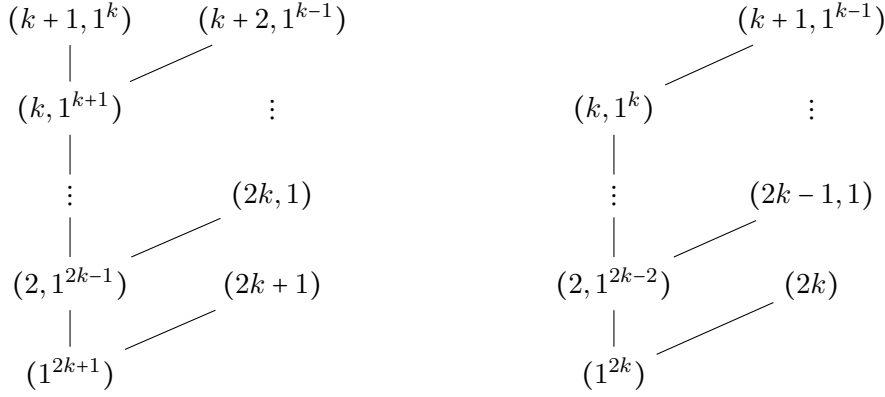

The Hasse diagram of the immersion poset restricted to hook partitions is given in Figure~\ref{figure.hook}.
Notice that item (1) in Propositions~\ref{proposition.hook odd} and~\ref{proposition.hook even} proves the string of covers on the left going up each 
Hasse diagram, while (2) proves the covers going up the right side which are the transposes.

\begin{corollary}
        \label{corollary.hook transpose}
        Let $\lambda \in S$ be a hook partition such that $\lambda \leqslant_D \lambda^t$. Then $\lambda \leqslant_I \lambda^t$.
\end{corollary}

The \defn{rank} of a poset is the length of the longest chain of elements of the poset. 
\begin{corollary}
\label{corr.rank}
	The rank of the immersion poset $(\mathcal{P}(n), \leqslant_I)$ is at least $\lfloor n/2 \rfloor$.
\end{corollary}

\subsection{Immersion poset on two column partitions}
\label{section.two column}

Now let $S$ be the set partitions with at most two columns, that is, $S=\{\lambda \mid \lambda_1\leqslant 2\}$. If $n=2k$, then for this section we define 
$\lambda^j=(2^{k-j}, 1^{2j})$ for $0\leqslant j\leqslant k$. Similarly, if $n=2k+1$, then $\lambda^j=(2^{k-j}, 1^{2j+1})$ for $0\leqslant j\leqslant k$. 
In this section, we study the immersion poset restricted to $S$.

\begin{remark}
        \label{remark.two column cover}
        Note that $K_{\lambda,\mu}=0$ if $\lambda\in S$ and $\mu\notin S$. Hence there does not exist an immersion pair $\mu \leqslant_I \lambda$  
        with $\lambda \in S $ and $\mu \not \in S$. This implies that if $\lambda^i$ is a cover for $\lambda^j$ in the subposet restricted to $S$, 
        then $\lambda^i$ is a cover for $\lambda^j$ in the immersion poset.
\end{remark}

This remark implies that we only need to consider $K_{\lambda, \mu}$ for $\lambda,\mu\in S$ when determining the immersion relations for 
this subset. Recall that $f^\lambda$ is the number of standard Young tableaux of shape $\lambda$.

\begin{proposition}
        \label{proposition.two column Kostkas}
        \mbox{}
        \begin{enumerate}
                \item
                Let $\lambda^j=(2^{k-j}, 1^{2j}) \vdash 2k$ for $0\leqslant j\leqslant k$. 
                Then $K_{\lambda^i,\lambda^j} = f^{(j+i,j-i)}$ when $i\leqslant j$ and $K_{\lambda^i,\lambda^j} = 0$ when $i>j$.
                \item
                Let $\lambda^j=(2^{k-j}, 1^{2j+1}) \vdash 2k+1$ for $0\leqslant j\leqslant k$. Then 
                $K_{\lambda^i,\lambda^j} = f^{(j+i+1,j-i)}$ when $i\leqslant j$ and $K_{\lambda^i,\lambda^j} = 0$ when $i>j$.
        \end{enumerate}
\end{proposition}

\begin{proof}
        For (1), if $i>j$, $\lambda^j$ dominates $\lambda^i$ and hence $K_{\lambda^i,\lambda^j} = 0$. If $i=j$, clearly 
        $K_{\lambda^i,\lambda^i} = f^{(2i)} = 1$. Suppose $j>i$. Then the first $k-j$ of the $k-i$ two length rows of any tableau 
        $T \in \mathsf{SSYT}(\lambda^i,\lambda^j)$ are fixed by the content. Hence, there is a bijection $\mathsf{SSYT}(\lambda^i,\lambda^j)
        \to \mathsf{SSYT}((2^{j-i},1^{2i}), (1^{2j}))$ by removing the first $k-j$ rows. Note that $K_{(2^{j-i},1^{2i}), (1^{2j})} = 
        f^{(2^{j-i},1^{2i})}$, which is also equal to the number of standard tableaux of the transpose of $(2^{j-i},1^{2i})$. The result follows.

        The proof of part (2) is similar.
\end{proof}

Using the hook length formula with Proposition~\ref{proposition.two column Kostkas}, we can describe $K_{\lambda^i,\lambda^j}$. We present this in the 
form of a matrix. More explicitly, suppose $n=2k$ or $n=2k+1$. Then for all $0\leqslant i,j\leqslant k$, the $i$-th row and $j$-th column 
entry of the matrix $T=(T_{i,j})$ is $T_{i,j} = K_{\lambda^i,\lambda^j}$. Note that the indexing starts with $0$.

\begin{example}
        \label{example.2colmatrix}
        Below are the matrices in tabular form for cases $n=14,15$.
        \begin{center}
                The case when $n=14$:
                \vspace{.1in}
                {\renewcommand{\arraystretch}{1.3}
                \begin{tabular}{ |l|c|c|c|c|c|c|c|c| }
                        \hline
                        \diagbox{Shape}{Content} & $(2^7)$ & $(2^6, 1^2)$ & $(2^5,1^4)$ & $(2^4,1^6)$ & $(2^3,1^8)$ & $(2^2,1^{10})$ & $(2,1^{12})$ & $(1^{14})$ \\
                        \hline
                        $(2^7)$ & 1 & 1 & 2 & 5 & 14 & 42 & 132 & 429\\
                        \hline
                        $(2^6,1^2)$ & 0 & 1 & 3 & 9 & 28 & 90 & 297 & 1001\\
                        \hline
                        $(2^5,1^4)$ & 0 & 0 & 1 & 5 & 20 & 75 & 275 & 1001\\
                        \hline
                        $(2^4,1^6)$ & 0 & 0 & 0 & 1 & 7 & 35 & 154 & 637\\
                        \hline
                        $(2^3,1^8)$ & 0 & 0 & 0 & 0 & 1 & 9 & 54 & 273\\
                        \hline
                        $(2^2,1^{10})$ & 0 & 0 & 0 & 0 & 0 & 1 & 11 & 77\\
                        \hline
                        $(2,1^{12})$ & 0 & 0 & 0 & 0 & 0 & 0 & 1 & 13\\
                        \hline
                        $(1^{14})$ & 0 & 0 & 0 & 0 & 0 & 0 & 0 & 1 \\
                        \hline
                \end{tabular}}
                \vspace{.2in}

                The case when $n=15$:
                \vspace{.1in}
                {\renewcommand{\arraystretch}{1.3}
                \begin{tabular}{ |l|c|c|c|c|c|c|c|c| }
                        \hline
                        \diagbox{Shape}{Content} & $(2^7,1)$ & $(2^6, 1^3)$ & $(2^5,1^5)$ & $(2^4,1^7)$ & $(2^3,1^9)$ & $(2^2,1^{11})$ & $(2,1^{13})$ 
                        & $(1^{15})$ \\
                        \hline
                        $(2^7,1)$ & 1 & 2 & 5 & 14 & 42 & 132 & 429 & 1430\\
                        \hline
                        $(2^6,1^3)$ & 0 & 1 & 4 & 14 & 48 & 165 & 572 & 2002\\
                        \hline
                        $(2^5,1^5)$ & 0 & 0 & 1 & 6 & 27 & 110 & 429 & 1638\\
                        \hline
                        $(2^4,1^7)$ & 0 & 0 & 0 & 1 & 8 & 44 & 208 & 910\\
                        \hline
                        $(2^3,1^9)$ & 0 & 0 & 0 & 0 & 1 & 10 & 65 & 350\\
                        \hline
                        $(2^2,1^{11})$ & 0 & 0 & 0 & 0 & 0 & 1 & 12 & 90\\
                        \hline
                        $(2,1^{13})$ & 0 & 0 & 0 & 0 & 0 & 0 & 1 & 14\\
                        \hline
                        $(1^{15})$ & 0 & 0 & 0 & 0 & 0 & 0 & 0 & 1 \\
                        \hline
                \end{tabular}}
        \end{center}
        \vspace{.2in}
\end{example}

Since the columns and rows are decreasing in dominance order, for any $i<j$ we have $\lambda^i \geqslant_I \lambda^j$ if $T_{i,m}\geqslant T_{j,m}$ for 
all $0\leqslant m \leqslant k$. In the following lemma, we prove some properties of the matrix $T$ that will show that this statement is 
equivalent to only comparing values in the last column of the matrix. That is, if $i<j$ and $T_{i,k} \geqslant T_{j,k}$, then 
$\lambda^i \geqslant_I \lambda^j$. The reader can verify this in Example~\ref{example.2colmatrix}.

\begin{lemma}
        \label{lemma.two column matrix}
        The matrix $(T_{i,j}) = (K_{\lambda^i,\lambda^j})$ defined above with $0\leqslant i,j\leqslant k$ has the following properties:
        \begin{enumerate}
                \item[(1)]The entries weakly increase within each row.
                \item[(2)] The entries within each column are unimodal.
                \item[(3)] The rate of change of entries within a row increases as the row number increases. In particular, for any fixed $i$ and $j$ 
                with $0\leqslant i < j \leqslant k$, we have for all $j\leqslant r < k$:
                \[
                \frac{T_{i,r+1}}{T_{i,r}} < \frac{T_{j,r+1}}{T_{j,r}}.
                \]

                \item[(4)] For any fixed $i$ and $j$ with $i<j$, if $T_{i,k}\geqslant T_{j,k}$, then $T_{i,m}\geqslant T_{j,m}$ for all $0\leqslant m\leqslant k$.
        \end{enumerate}
\end{lemma}

\begin{proof}
        We begin by proving (1). Let $n=2k$ be even. Then for a fixed row $i$, given any $i\leqslant j <k$, we need to show that 
        $T_{i,j+1} \geqslant T_{i,j}$. Using Proposition~\ref{proposition.two column Kostkas}, we have:
        \[
        \frac{T_{i,j+1}}{T_{i,j}} = \frac{f^{(j+1+i,j+1-i)}}{f^{(j+i,j-i)}} = \frac{(2j+2)(2j+1)}{(j+i+2)(j+1-i)} \geqslant 1
        \]
        because $j\geqslant i$ implies
        \[
        2j+2\geqslant j+i+2 \quad \text{and} \quad 2j+1 \geqslant j+1-i.
        \]
        Now let $n=2k+1$ be odd. Using the same strategy as the even case we have:
        \[
        \frac{T_{i,j+1}}{T_{i,j}} = \frac{f^{(j+2+i,j+1-i)}}{f^{(j+i+1,j-i)}} = \frac{(2j+3)(2j+2)}{(j+i+3)(j+1-i)} \geqslant 1
        \]
        because $j\geqslant i$ implies
        \[
        2j+3\geqslant j+i+3 \quad \text{and} \quad 2j+2 \geqslant j+1-i.
        \]

        Next we prove statement (2). Let $n=2k$ be even. Since statement (2) holds trivially if there is only one non-zero entry in the column, we focus 
        on columns with more than one non-zero entry. Fix a $2\leqslant j\leqslant k$. To determine when the column is increasing and decreasing we 
        consider the fraction:
        \[
        \frac{T_{i+1,j}}{T_{i,j}} = \frac{f^{(j+i+1,j-i-1)}}{f^{(j+i,j-i)}} = \frac{\frac{(2j)!(2i+3)}{(j+i+2)!(j-i-1)!}}{\frac{(2j)!(2i+1)}{(j+i+1)!(j-i)!}} 
        = \frac{(2i+3)(j-i)}{(2i+1)(j+i+2)}.
        \]
        Analyzing the following inequalities gives:
        \begin{equation}
        \label{equation.unimodaleven}
        \begin{split}
            \frac{T_{i+1,j}}{T_{i,j}}>1 \quad \iff \quad 2i^2 +4i + 1 < j, \\
            \frac{T_{i+1,j}}{T_{i,j}}=1 \quad \iff \quad 2i^2 +4i + 1 = j, \\
            \frac{T_{i+1,j}}{T_{i,j}}<1 \quad \iff \quad 2i^2 +4i + 1 > j.
        \end{split}
        \end{equation}
        Thus, for values of $i$ such that $ 2i^2 +4i + 1 < j $ the column entries are increasing, and when the values of $i$ satisfy $2i^2 +4i + 1 > j$ 
        the column entries are decreasing. This proves (2) for the even case.

        Now let $n=2k+1$ be odd. Fix a $2\leqslant j\leqslant k$. Similar to the even case we have:
        \[
        \frac{T_{i+1,j}}{T_{i,j}} = \frac{f^{(j+i+2,j-i-1)}}{f^{(j+i+1,j-i)}} = \frac{\frac{(2j+1)!(2i+4)}{(j+i+3)!(j-i-1)!}}{\frac{(2j+1)!(2i+2)}{(j+i+2)!(j-i)!}} 
        = \frac{(2i+4)(j-i)}{(2i+2)(j+i+3)}.
        \]
         Analyzing the following inequalities gives:
        \begin{equation}
	\label{equation.unimodalodd}
        \begin{split}
            \frac{T_{i+1,j}}{T_{i,j}}>1 \quad \iff \quad 2i^2 +6i + 3 < j, \\
            \frac{T_{i+1,j}}{T_{i,j}}=1 \quad \iff \quad 2i^2 +6i + 3 = j, \\
            \frac{T_{i+1,j}}{T_{i,j}}<1 \quad \iff \quad 2i^2 +6i + 3 > j.
        \end{split}
        \end{equation}
        Again, we notice that for values of $i$ such that $ 2i^2 +6i + 3 < j $ the column entries are increasing, and when the values of $i$ satisfy 
        $2i^2 +6i + 3 > j$ the column entries are decreasing. This concludes the proof of (2).

        To prove statement (3), we use Proposition~\ref{proposition.two column Kostkas} and the hook length formula to get the following equivalences:
        \begin{equation*}
                \frac{T_{i,r+1}}{T_{i,r}} < \frac{T_{j,r+1}}{T_{j,r}} \quad
                \iff \quad \frac{K_{\lambda^{i},\lambda^{r+1}}}{K_{\lambda^{i},\lambda^{r}}} < \frac{K_{\lambda^{j},\lambda^{r+1}}}{K_{\lambda^{j},\lambda^{r}}} 
                \quad
                \iff \quad \frac{(r+j+2)(r+1-j)}{(r+i+2)(r+1-i)} < 1 \quad \iff \quad i^2+i < j^2 + j.
        \end{equation*}
        The last inequality is always true since $0\leqslant i < j$, thus proving (3).

        To prove (4), fix $i$ and $j$ with $i<j$ where $T_{i,k}\geqslant T_{j,k}$. Then by statement (3) it directly follows that 
        $T_{i,m}\geqslant T_{j,m}$ for all $j\leqslant m\leqslant k$. Because $T_{j,m}=0$ for all $0\leqslant m < j$, it trivially follows that $T_{i,m}\geqslant T_{j,m}$ for these values of $m$, this finishes the proof of (4).
\end{proof}

The beauty of Lemma~\ref{lemma.two column matrix}, in particular statement (4), is that we can now reduce much of the work in determining 
the immersion relations between partitions in $S$ to just comparing the numbers of standard Young tableaux, as is done in the next proposition.

\begin{proposition}
        \label{prop.k col}
        For $n=2k$ even or $n=2k+1$ odd, the last ($k$-th) column of $T$ can be used to completely determine relations in the immersion 
        poset restricted to the subset $S$. In particular:
        \begin{enumerate}
                \item $\lambda^i \geqslant_I \lambda^j$ if and only if $i<j$ and $T_{i,k}\geqslant T_{j,k}$,
                \item For $i<j$, $\lambda^i$ and $\lambda^j$ are incomparable if and only if $T_{j,k}> T_{i,k}$.
        \end{enumerate}
\end{proposition}

\begin{proof}
        To prove (1), by definition $\lambda^i \geqslant_I \lambda^j$ if and only if $\lambda^i >_D \lambda^j$ and $T_{i,\ell}\geqslant T_{j,\ell}$ 
        for all $0\leqslant \ell\leqslant k$. But $\lambda^i$ dominates $\lambda^j$ if and only if $i<j$, and by (4) of Lemma~\ref{lemma.two column matrix}, 
        $T_{i,\ell}\geqslant T_{j,\ell}$ for all $0\leqslant \ell\leqslant k$ if and only if $T_{i,k}\geqslant T_{j,k}$ (when $i<j$).

        To prove (2), let $i<j$. If $\lambda^i$ and $\lambda^j$ are incomparable, then there exists some $\ell$ such that $T_{j,\ell}> T_{i,\ell}$.
        By (3) of Lemma~\ref{lemma.two column matrix}, we have:
        $$\frac{T_{i,r+1}}{T_{i,r}}\leqslant \frac{T_{j,r+1}}{T_{j,r}}$$
        for all $\ell\leqslant r<k$, which guarantees that $T_{j,k}> T_{i,k}$.
\end{proof}

As a consequence we obtain the following immediate corollary.

\begin{corollary}
        The cover relations for the immersion poset of the set $S$ are the exact same as those in the standard immersion poset.
\end{corollary}

We can now explain the cover relations of the immersion poset restricted to the set $S$.

\begin{proposition}
        \label{prop.two col relations}
        Let $n=2k$ be even or $n=2k+1$ be odd, then:
        \begin{enumerate}
                \item $\lambda^{i} \gtrdot_I \lambda^{i+1}$ when $2i^2 + 4i + 2 > k$ for $n$ even and $2i^2 + 6i + 4>k$ for $n$ odd. This also 
                coincides with Lemma~\ref{two col}, taking $a=k-i-1$ and $b=2i+2$ ($n$ even) or $b=2i+3$ ($n$ odd).
                \item $\lambda^i$ and $\lambda^{j}$ are incomparable in the immersion poset for all $0\leqslant i,j\leqslant i_{max}$ with $i\neq j$ and $i_{max}$ 
                being the largest $i$ value not satisfying (1).
                \item Fix $i$ with $0\leqslant i \leqslant i_{max}$ and let $m>i_{max}-i$ be smallest such that $T_{i,k}\geqslant T_{i+m,k}$.
                Then $\lambda^i \gtrdot_I \lambda^{i+m}$.
        \end{enumerate}
\end{proposition}

\begin{proof}
        If $T_{i,k}\geqslant T_{i+1,k}$, then by Proposition~\ref{prop.k col} and Remark~\ref{remark.two column cover}, we have that 
        $\lambda^{i} \geqslant_I \lambda^{i+1}$ is a cover. We determine the values for $i$ such that $T_{i,k}\geqslant T_{i+1,k}$ by using the 
        middle equation and bottom inequality of (\ref{equation.unimodaleven}) (for $n$ even) and (\ref{equation.unimodalodd}) (for $n$ odd), 
        where we replace $j$ with $k$. Specifically, for $n$ even:
        $$2i^2+4i+1 \geqslant k \quad \implies \quad 2i^2 + 4i + 2> k,$$
        and for $n$ odd:
        $$2i^2+6i + 3 \geqslant k \quad \implies \quad 2i^2 + 6i + 4>k.$$

        To prove (2), notice that since $i_{max}$ is the number of the row containing the first maximum, by the increasing nature of the column up 
        to the maximum value given by (2) of Lemma~\ref{lemma.two column matrix}, then for any $0\leqslant i<j\leqslant i_{max}$ we have 
        $T_{i,k}<T_{j,k}$. Hence by Proposition~\ref{prop.k col} (2), $\lambda^i$ and $\lambda^j$ are incomparable.

        To prove (3) notice that by Proposition~\ref{prop.k col} statement (1), since $m$ is the smallest value it must be a cover.
\end{proof}

\begin{example}
        Suppose $n=14$, so that $k=7$. By (1) of Proposition~\ref{prop.two col relations},  the inequality holds for $1\leqslant i \leqslant k=7$ 
        so we obtain:
        $$\lambda^{7} \lessdot_I \lambda^{6}\lessdot_I \lambda^{5}\lessdot_I \lambda^{4} \lessdot_I \lambda^{3} \lessdot_I \lambda^{2} 
        \lessdot_I \lambda^{1}.$$
	Applying (3) of Proposition~\ref{prop.two col relations} to $\lambda^0$, with $i=0$ we find that $m=4$:
        $$T_{0, 7} = 429 \geqslant 273 = T_{4,7}.$$
        Notice that $m=3$ does not satisfy the inequality:
        $$T_{0, 7} = 429 \ngeqslant 637 = T_{3,7}.$$
        So our final cover relation for the poset is $\lambda^{4} \lessdot_I \lambda^{0}$.
\end{example}

\subsection{Lower intervals and Schur-positivity of interval power sums}
\label{section.lowerintervals}
In this section, we make conjectures about certain lower intervals $A_\mu := \{\lambda \,|\, (1^n) \leqslant_I \lambda \leqslant_I \mu\}$ in 
the immersion poset. Determining intervals will
\begin{enumerate}
        \item enhance our understanding of the immersion of polynomial representations for $GL_N(\mathbb{C})$ and 
        \item allow us to investigate when $p_{A_\mu}$ of Equation~\eqref{char} is Schur-positive, as asked in Question~\ref{question.Sundaram}.
        We call $p_{A_\mu}$ an interval power sum.
        It also helps towards constructing a natural corresponding representation of the symmetric group.
\end{enumerate}
In this section, we prove that $p_{A_\mu}$ is Schur-positive for the conjectured intervals.

\begin{conjecture}
        \label{conjecture.interval1}
        For $n=5$ and $n \geqslant 9$, the interval $A_{(n-2, 2)} = \{ \lambda \, |\, (1^n) \leqslant_I \lambda \leqslant_I (n-2, 2)\}$ is exactly
        \[
        (1^n) \lessdot_I (2, 1^{n-2}) \lessdot_I (2, 2, 1^{n-4}) \lessdot_I (n-2, 2).
        \]
\end{conjecture}

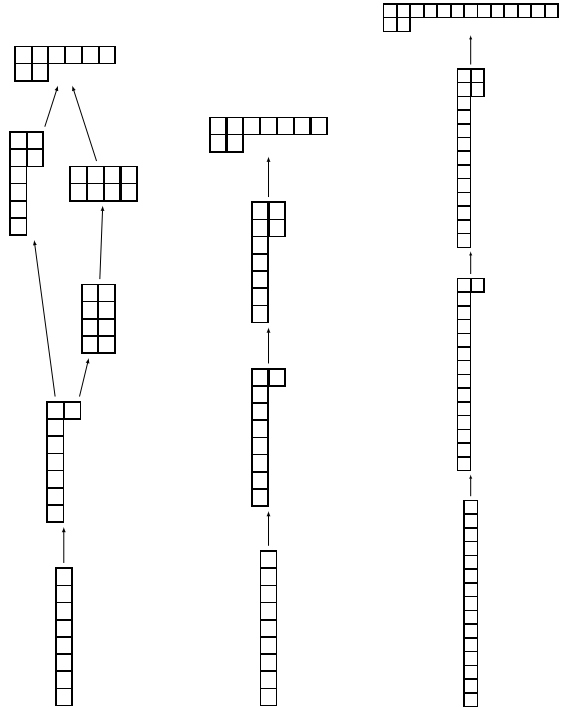
\begin{figure}[t]
        \scalebox{0.5}{
                \begin{tikzpicture}[>=latex,line join=bevel,]
                        \node (node_0) at (42.5bp,51.0bp) [draw,draw=none] {${\def\lr#1{\multicolumn{1}{|@{\hspace{.6ex}}c@{\hspace{.6ex}}|}{\raisebox{-.3ex}{$#1$}}}\raisebox{-.6ex}{$\begin{array}[b]{*{1}c}\cline{1-1}\lr{\phantom{x}}\\\cline{1-1}\lr{\phantom{x}}\\\cline{1-1}\lr{\phantom{x}}\\\cline{1-1}\lr{\phantom{x}}\\\cline{1-1}\lr{\phantom{x}}\\\cline{1-1}\lr{\phantom{x}}\\\cline{1-1}\lr{\phantom{x}}\\\cline{1-1}\lr{\phantom{x}}\\\cline{1-1}\end{array}$}}$};
                        \node (node_1) at (42.5bp,183.0bp) [draw,draw=none] {${\def\lr#1{\multicolumn{1}{|@{\hspace{.6ex}}c@{\hspace{.6ex}}|}{\raisebox{-.3ex}{$#1$}}}\raisebox{-.6ex}{$\begin{array}[b]{*{2}c}\cline{1-2}\lr{\phantom{x}}&\lr{\phantom{x}}\\\cline{1-2}\lr{\phantom{x}}\\\cline{1-1}\lr{\phantom{x}}\\\cline{1-1}\lr{\phantom{x}}\\\cline{1-1}\lr{\phantom{x}}\\\cline{1-1}\lr{\phantom{x}}\\\cline{1-1}\lr{\phantom{x}}\\\cline{1-1}\end{array}$}}$};
                        \node (node_2) at (14.5bp,393.0bp) [draw,draw=none] {${\def\lr#1{\multicolumn{1}{|@{\hspace{.6ex}}c@{\hspace{.6ex}}|}{\raisebox{-.3ex}{$#1$}}}\raisebox{-.6ex}{$\begin{array}[b]{*{2}c}\cline{1-2}\lr{\phantom{x}}&\lr{\phantom{x}}\\\cline{1-2}\lr{\phantom{x}}&\lr{\phantom{x}}\\\cline{1-2}\lr{\phantom{x}}\\\cline{1-1}\lr{\phantom{x}}\\\cline{1-1}\lr{\phantom{x}}\\\cline{1-1}\lr{\phantom{x}}\\\cline{1-1}\end{array}$}}$};
                        \node (node_3) at (68.5bp,291.0bp) [draw,draw=none] {${\def\lr#1{\multicolumn{1}{|@{\hspace{.6ex}}c@{\hspace{.6ex}}|}{\raisebox{-.3ex}{$#1$}}}\raisebox{-.6ex}{$\begin{array}[b]{*{2}c}\cline{1-2}\lr{\phantom{x}}&\lr{\phantom{x}}\\\cline{1-2}\lr{\phantom{x}}&\lr{\phantom{x}}\\\cline{1-2}\lr{\phantom{x}}&\lr{\phantom{x}}\\\cline{1-2}\lr{\phantom{x}}&\lr{\phantom{x}}\\\cline{1-2}\end{array}$}}$};
                        \node (node_5) at (43.5bp,483.5bp) [draw,draw=none] {${\def\lr#1{\multicolumn{1}{|@{\hspace{.6ex}}c@{\hspace{.6ex}}|}{\raisebox{-.3ex}{$#1$}}}\raisebox{-.6ex}{$\begin{array}[b]{*{6}c}\cline{1-6}\lr{\phantom{x}}&\lr{\phantom{x}}&\lr{\phantom{x}}&\lr{\phantom{x}}&\lr{\phantom{x}}&\lr{\phantom{x}}\\\cline{1-6}\lr{\phantom{x}}&\lr{\phantom{x}}\\\cline{1-2}\end{array}$}}$};
                        \node (node_4) at (72.5bp,393.0bp) [draw,draw=none] {${\def\lr#1{\multicolumn{1}{|@{\hspace{.6ex}}c@{\hspace{.6ex}}|}{\raisebox{-.3ex}{$#1$}}}\raisebox{-.6ex}{$\begin{array}[b]{*{4}c}\cline{1-4}\lr{\phantom{x}}&\lr{\phantom{x}}&\lr{\phantom{x}}&\lr{\phantom{x}}\\\cline{1-4}\lr{\phantom{x}}&\lr{\phantom{x}}&\lr{\phantom{x}}&\lr{\phantom{x}}\\\cline{1-4}\end{array}$}}$};
                        \draw [black,->] (node_0) ..controls (42.5bp,110.59bp) and (42.5bp,119.26bp)  .. (node_1);
                        \draw [black,->] (node_1) ..controls (31.989bp,262.08bp) and (25.675bp,308.99bp)  .. (node_2);
                        \draw [black,->] (node_1) ..controls (55.423bp,236.69bp) and (57.626bp,245.67bp)  .. (node_3);
                        \draw [black,->] (node_2) ..controls (29.86bp,440.87bp) and (32.863bp,450.04bp)  .. (node_5);
                        \draw [black,->] (node_3) ..controls (70.148bp,333.19bp) and (70.91bp,352.24bp)  .. (node_4);
                        \draw [black,->] (node_4) ..controls (63.304bp,422.07bp) and (56.68bp,442.28bp)  .. (node_5);
                \end{tikzpicture}
                \qquad
                \qquad
                \begin{tikzpicture}[>=latex,line join=bevel,]
                        \node (node_0) at (42.0bp,57.0bp) [draw,draw=none] {${\def\lr#1{\multicolumn{1}{|@{\hspace{.6ex}}c@{\hspace{.6ex}}|}{\raisebox{-.3ex}{$#1$}}}\raisebox{-.6ex}{$\begin{array}[b]{*{1}c}\cline{1-1}\lr{\phantom{x}}\\\cline{1-1}\lr{\phantom{x}}\\\cline{1-1}\lr{\phantom{x}}\\\cline{1-1}\lr{\phantom{x}}\\\cline{1-1}\lr{\phantom{x}}\\\cline{1-1}\lr{\phantom{x}}\\\cline{1-1}\lr{\phantom{x}}\\\cline{1-1}\lr{\phantom{x}}\\\cline{1-1}\lr{\phantom{x}}\\\cline{1-1}\end{array}$}}$};
                        \node (node_1) at (42.0bp,201.0bp) [draw,draw=none] {${\def\lr#1{\multicolumn{1}{|@{\hspace{.6ex}}c@{\hspace{.6ex}}|}{\raisebox{-.3ex}{$#1$}}}\raisebox{-.6ex}{$\begin{array}[b]{*{2}c}\cline{1-2}\lr{\phantom{x}}&\lr{\phantom{x}}\\\cline{1-2}\lr{\phantom{x}}\\\cline{1-1}\lr{\phantom{x}}\\\cline{1-1}\lr{\phantom{x}}\\\cline{1-1}\lr{\phantom{x}}\\\cline{1-1}\lr{\phantom{x}}\\\cline{1-1}\lr{\phantom{x}}\\\cline{1-1}\lr{\phantom{x}}\\\cline{1-1}\end{array}$}}$};
                        \node (node_2) at (42.0bp,333.0bp) [draw,draw=none] {${\def\lr#1{\multicolumn{1}{|@{\hspace{.6ex}}c@{\hspace{.6ex}}|}{\raisebox{-.3ex}{$#1$}}}\raisebox{-.6ex}{$\begin{array}[b]{*{2}c}\cline{1-2}\lr{\phantom{x}}&\lr{\phantom{x}}\\\cline{1-2}\lr{\phantom{x}}&\lr{\phantom{x}}\\\cline{1-2}\lr{\phantom{x}}\\\cline{1-1}\lr{\phantom{x}}\\\cline{1-1}\lr{\phantom{x}}\\\cline{1-1}\lr{\phantom{x}}\\\cline{1-1}\lr{\phantom{x}}\\\cline{1-1}\end{array}$}}$};
                        \node (node_3) at (42.0bp,429.5bp) [draw,draw=none] {${\def\lr#1{\multicolumn{1}{|@{\hspace{.6ex}}c@{\hspace{.6ex}}|}{\raisebox{-.3ex}{$#1$}}}\raisebox{-.6ex}{$\begin{array}[b]{*{7}c}\cline{1-7}\lr{\phantom{x}}&\lr{\phantom{x}}&\lr{\phantom{x}}&\lr{\phantom{x}}&\lr{\phantom{x}}&\lr{\phantom{x}}&\lr{\phantom{x}}\\\cline{1-7}\lr{\phantom{x}}&\lr{\phantom{x}}\\\cline{1-2}\end{array}$}}$};
                        \draw [black,->] (node_0) ..controls (42.0bp,122.5bp) and (42.0bp,131.17bp)  .. (node_1);
                        \draw [black,->] (node_1) ..controls (42.0bp,260.59bp) and (42.0bp,269.26bp)  .. (node_2);
                        \draw [black,->] (node_2) ..controls (42.0bp,386.96bp) and (42.0bp,395.85bp)  .. (node_3);
                \end{tikzpicture}
        }
        \quad
        \scalebox{0.4}{
                \begin{tikzpicture}[>=latex,line join=bevel,]
                        \node (node_0) at (75.0bp,93.0bp) [draw,draw=none] {${\def\lr#1{\multicolumn{1}{|@{\hspace{.6ex}}c@{\hspace{.6ex}}|}{\raisebox{-.3ex}{$#1$}}}\raisebox{-.6ex}{$\begin{array}[b]{*{1}c}\cline{1-1}\lr{\phantom{x}}\\\cline{1-1}\lr{\phantom{x}}\\\cline{1-1}\lr{\phantom{x}}\\\cline{1-1}\lr{\phantom{x}}\\\cline{1-1}\lr{\phantom{x}}\\\cline{1-1}\lr{\phantom{x}}\\\cline{1-1}\lr{\phantom{x}}\\\cline{1-1}\lr{\phantom{x}}\\\cline{1-1}\lr{\phantom{x}}\\\cline{1-1}\lr{\phantom{x}}\\\cline{1-1}\lr{\phantom{x}}\\\cline{1-1}\lr{\phantom{x}}\\\cline{1-1}\lr{\phantom{x}}\\\cline{1-1}\lr{\phantom{x}}\\\cline{1-1}\lr{\phantom{x}}\\\cline{1-1}\end{array}$}}$};
                        \node (node_1) at (75.0bp,309.0bp) [draw,draw=none] {${\def\lr#1{\multicolumn{1}{|@{\hspace{.6ex}}c@{\hspace{.6ex}}|}{\raisebox{-.3ex}{$#1$}}}\raisebox{-.6ex}{$\begin{array}[b]{*{2}c}\cline{1-2}\lr{\phantom{x}}&\lr{\phantom{x}}\\\cline{1-2}\lr{\phantom{x}}\\\cline{1-1}\lr{\phantom{x}}\\\cline{1-1}\lr{\phantom{x}}\\\cline{1-1}\lr{\phantom{x}}\\\cline{1-1}\lr{\phantom{x}}\\\cline{1-1}\lr{\phantom{x}}\\\cline{1-1}\lr{\phantom{x}}\\\cline{1-1}\lr{\phantom{x}}\\\cline{1-1}\lr{\phantom{x}}\\\cline{1-1}\lr{\phantom{x}}\\\cline{1-1}\lr{\phantom{x}}\\\cline{1-1}\lr{\phantom{x}}\\\cline{1-1}\lr{\phantom{x}}\\\cline{1-1}\end{array}$}}$};
                        \node (node_2) at (75.0bp,513.0bp) [draw,draw=none] {${\def\lr#1{\multicolumn{1}{|@{\hspace{.6ex}}c@{\hspace{.6ex}}|}{\raisebox{-.3ex}{$#1$}}}\raisebox{-.6ex}{$\begin{array}[b]{*{2}c}\cline{1-2}\lr{\phantom{x}}&\lr{\phantom{x}}\\\cline{1-2}\lr{\phantom{x}}&\lr{\phantom{x}}\\\cline{1-2}\lr{\phantom{x}}\\\cline{1-1}\lr{\phantom{x}}\\\cline{1-1}\lr{\phantom{x}}\\\cline{1-1}\lr{\phantom{x}}\\\cline{1-1}\lr{\phantom{x}}\\\cline{1-1}\lr{\phantom{x}}\\\cline{1-1}\lr{\phantom{x}}\\\cline{1-1}\lr{\phantom{x}}\\\cline{1-1}\lr{\phantom{x}}\\\cline{1-1}\lr{\phantom{x}}\\\cline{1-1}\lr{\phantom{x}}\\\cline{1-1}\end{array}$}}$};
                        \node (node_3) at (75.0bp,645.5bp) [draw,draw=none] {${\def\lr#1{\multicolumn{1}{|@{\hspace{.6ex}}c@{\hspace{.6ex}}|}{\raisebox{-.3ex}{$#1$}}}\raisebox{-.6ex}{$\begin{array}[b]{*{13}c}\cline{1-13}\lr{\phantom{x}}&\lr{\phantom{x}}&\lr{\phantom{x}}&\lr{\phantom{x}}&\lr{\phantom{x}}&\lr{\phantom{x}}&\lr{\phantom{x}}&\lr{\phantom{x}}&\lr{\phantom{x}}&\lr{\phantom{x}}&\lr{\phantom{x}}&\lr{\phantom{x}}&\lr{\phantom{x}}\\\cline{1-13}\lr{\phantom{x}}&\lr{\phantom{x}}\\\cline{1-2}\end{array}$}}$};
                        \draw [black,->] (node_0) ..controls (75.0bp,194.67bp) and (75.0bp,203.16bp)  .. (node_1);
                        \draw [black,->] (node_1) ..controls (75.0bp,404.64bp) and (75.0bp,413.16bp)  .. (node_2);
                        \draw [black,->] (node_2) ..controls (75.0bp,603.26bp) and (75.0bp,612.06bp)  .. (node_3);
                \end{tikzpicture}
        }
        \caption{Subposet of the immersion poset only containing partitions in $A_{(n-2,2)}$ for $n=8$ (left), $n=9$ (middle), and $n=15$ (right).\label{figure.interval1ex}} 
\end{figure}

\begin{remark} 
        The first two covers are consequences of Proposition~\ref{prop.two col relations} (1). The map
        \[
        \mathsf{SSYT}((2,2,1^{n-4}), \nu) \longrightarrow \mathsf{SSYT}((n-2,2), \nu),
        \]
        which is the transpose if $\nu_1=1$ and which moves the boxes in positions $(3,1), \ldots, (n-2,1)$ to positions $(1,3), \ldots, (1,n-2)$
        if $\nu_1=2$ is an injection. This shows that $(2,2,1^{n-4}) <_I (n-2,2)$.  Therefore, we have
        \[
        \{(1^n), (2, 1^{n-2}), (2, 2, 1^{n-4}), (n-2, 2) \} \subseteq A_{(n-2, 2)}.
        \]
        However, we have not proven the cover relation $(2,2,1^{n-4}) \lessdot_I (n-2,2)$. One strategy to show the reverse containment is to argue that for all partitions $\lambda$ such that $\lambda$ and $\lambda^t$ are not included in the above list, we have $f^\lambda > \frac{n(n-3)}{2} = f^{(n-2,2)}$. This would prove that $\lambda \not<_I (n-2,2)$, hence $\lambda \not\in A_{(n-2,2)}$. 
        We have confirmed the conjecture up to $n=18$. 
        See Figure \ref{figure.interval1ex}. 
\end{remark}

\begin{proposition}
\label{prop.sposinterval1}
\mbox{}
\begin{enumerate}
\item For $n<7$ and $n=8$, $p_{A_{(n-2,2)}}$ is Schur-positive.
\item For $n \geqslant 9$, $p_{(1^n)} + p_{(2,1^{n-2})} +  p_{(2,2,1^{n-4})} +  p_{(n-2,2)}$ is Schur-positive. 
\end{enumerate}
\end{proposition}

\begin{proof}
Part (1) can be checked explicitly by {\sc SageMath}. For part (2), let 
\begin{equation}
\label{eq.interval 1}
	p_{(1^n)} + p_{(2,1^{n-2})} +  p_{(2,2,1^{n-4})} +  p_{(n-2,2)}  = \sum_{\lambda \vdash n} c_\lambda s_\lambda.
\end{equation}
We prove that $c_\lambda \geqslant 0$ for all $\lambda \vdash n$ by proving that all partial sums 
$p_{(1^n)}$, $p_{(1^n)} + p_{(2,1^{n-2})}$, $p_{(1^n)} + p_{(2,1^{n-2})} +  p_{(2,2,1^{n-4})}$, 
$p_{(1^n)} + p_{(2,1^{n-2})} +  p_{(2,2,1^{n-4})} +  p_{(n-2,2)}$ are Schur-positive.
We employ the combinatorial Murnaghan--Nakayama rule involving ribbon tableaux (see for example~\cite[Chapter 7.17]{EC2})
\[
	p_\mu = \sum_{\lambda \vdash n}\chi^{\lambda}(\mu)s_\lambda 
	\qquad \text{where} \qquad 
	\chi^\lambda(\mu) = \sum_{T \in \mathsf{R}(\lambda, \mu)}(-1)^{\text{ht}(T)}
\]
and $\mathsf{R}(\lambda,\mu)$ is the set of all ribbon tableaux of shape $\lambda$ and type $\mu$ and $\text{ht}(T)$ is equal to the sum of the 
heights of all ribbons in $T$. We show that each subset of ribbon tableaux that contributes a negative term to $c_\lambda$ is in bijection with a 
distinct subset of ribbon tableaux that contributes a positive number to $c_\lambda$, ensuring that $c_\lambda \geqslant 0$. We examine each 
partial sum of power sum symmetric functions, and demonstrate Schur-positivity at each step through these bijections.

\medskip

\noindent
(1) It is well-known that $p_{(1^n)} = \sum_{\lambda\vdash n} f^\lambda s_\lambda$ since $\mathsf{R}(\lambda, (1^n))$ is the set of all standard
Young tableaux of shape $\lambda$.

\medskip

\noindent
(2) For $T \in \mathsf{R}(\lambda, (2,1^{n-2}))$, $T$ has either a horizontal or a vertical $2$-ribbon and the remaining are single box ribbons.
If $T$ has a horizontal $2$-ribbon, then $\text{ht}(T) = 0$ and $T$ contributes $+1$ to $\chi^\lambda((2,1^{n-2}))$. There are 
$f^{\lambda/(2)}$ such ribbon tableaux in $\mathsf{R}(\lambda, (2,1^{n-2}))$, where $f^{\lambda/\mu}$ is the cardinality of $\mathsf{SYT}(\lambda/\mu)$, the set of standard Young tableaux of skew shape $\lambda/\mu$.
 If $T$ has a vertical $2$-ribbon, then $\text{ht}(T) = 1$ and $T$ contributes $-1$ to 
$\chi^\lambda((2,1^{n-2}))$. There are $f^{\lambda/(1,1)}$ such ribbon tableaux in $R(\lambda, (2,1^{n-2}))$. Therefore, the coefficient of 
$s_\lambda$ in $p_{(2, 1^{n-2})}$ is $f^{\lambda/(2)} - f^{\lambda/(1,1)}$.
If $(1,1) \subseteq \lambda$, then $c_\lambda$ includes $-f^{\lambda / (1,1)}$. The natural bijection
\[
	\mathsf{SYT}(\lambda / (1,1)) \rightarrow \{ T \in \mathsf{SYT}(\lambda) \mid T_{1,1} = 1 \text{ and } T_{2,1} = 2 \}
\]
demonstrates that $f^\lambda-f^{\lambda/(1,1)} \geqslant 0$. Hence $p_{(1^n)} + p_{(2,1^{n-2})}$ is Schur-positive.

\medskip
\noindent
(3) For any $T \in \mathsf{R}(\lambda, (2,2,1^{n-4}))$, there are six possible ways to arrange two $2$-ribbons. 
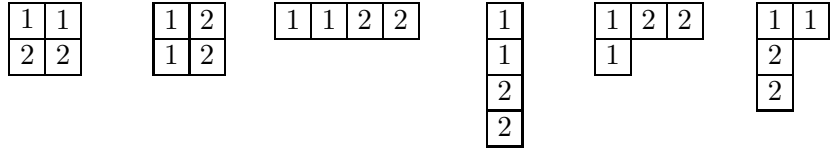
\[
\begin{array}{cccccc}
\ytableausetup{boxsize=1.2em} 
\begin{ytableau}
1 & 1 \\
2 & 2
\end{ytableau} & \ytableausetup{boxsize=1.2em} 
\begin{ytableau}
1 & 2 \\
1 & 2
\end{ytableau} & \ytableausetup{boxsize=1.2em} 
\begin{ytableau}
1 & 1 &2 & 2
\end{ytableau} &\ytableausetup{boxsize=1.2em} 
\begin{ytableau}
1 \\
1 \\
2\\
2
\end{ytableau}& \ytableausetup{boxsize=1.2em} 
\begin{ytableau}
1 & 2 & 2\\
1
\end{ytableau} & \ytableausetup{boxsize=1.2em} 
\begin{ytableau}
1 & 1 \\
2 \\
2
\end{ytableau} \\
& & & & & \\
\text{ht}(T) = 0 & \text{ht}(T) = 2 & \text{ht}(T) = 0 & \text{ht}(T) = 2 & \text{ht}(T) = 1 & \text{ht}(T) = 1
\end{array}
\]
The remaining $n-4$ ribbons in $T$ are single boxes. Therefore, the coefficient of $s_\lambda$ in $p_{(2, 2, 1^{n-4})}$ is
\[
2 f^{\lambda / (2,2)} + f^{\lambda / (4)} + f^{\lambda / (1^4)} - f^{\lambda / (3,1)} - f^{\lambda / (2,1,1)}.
\]
If $(3,1) \subseteq \lambda$, then $c_\lambda$ includes $-f^{\lambda / (3,1)}$. Consider the bijection
\[
\mathsf{SYT}(\lambda / (3,1)) \rightarrow \{ T \in \mathsf{SYT}(\lambda/(2)) \mid T_{1,3} = 1, \text{ and } T_{2,1} = 2 \}.
\]
If $(2,1,1) \subseteq \lambda$, then $c_\lambda$ includes $-f^{\lambda / (2,1,1)}$. Consider the bijection
\[
\mathsf{SYT}(\lambda / (2,1,1)) \rightarrow \{ T \in \mathsf{SYT}(\lambda / (2)) \mid T_{2,1} = 1 \text{ and } T_{3,1} = 2 \}.
\]
Hence $f^{\lambda/(2)}$ from (2) and the terms $-f^{\lambda/(3,1)} - f^{\lambda/(2,1,1)}$ from (3) satisfy 
$f^{\lambda/(2)}-f^{\lambda/(3,1)} - f^{\lambda/(2,1,1)}\geqslant 0$.
So far, we have shown that $p_{(1^n)} + p_{(2,1^{n-2})} + p_{(2,2,1^{n-4})}$ is Schur-positive.

\medskip
\noindent
(4) For $T \in \mathsf{R}(\lambda, (n-2,2))$, the possible ways of arranging one $(n-2)$-ribbon and one $2$-ribbon in $T$ are the following. Note that $1^0$ appearing in $\lambda$ means that there are no parts of size $1$ in $\lambda$.
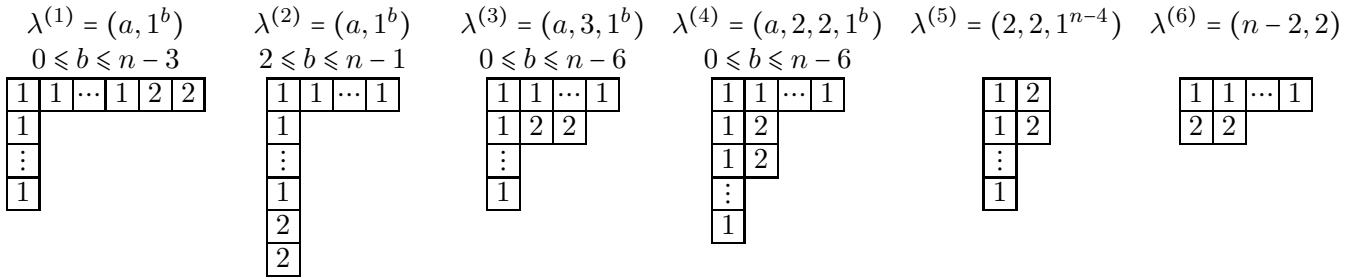
\[
\begin{array}{cccccc}
\lambda^{(1)} = (a, 1^b) &\lambda^{(2)} = (a, 1^b)  &\lambda^{(3)} = (a, 3, 1^b)  & \lambda^{(4)} = (a, 2, 2, 1^b)  & \lambda^{(5)} = (2,2,1^{n-4})  
& \lambda^{(6)} = (n-2,2) \\
0\leqslant b \leqslant n-3 & 2\leqslant b \leqslant n-1 & 0\leqslant b \leqslant n-6 & 0\leqslant b \leqslant n-6 \\
 \ytableausetup{boxsize=1.1em} 
\begin{ytableau}
1 & 1 & \cdots & 1 & 2 & 2\\
1 \\
\vdots \\
1
\end{ytableau} &  \ytableausetup{boxsize=1.1em} 
\begin{ytableau}
1 & 1 & \cdots & 1 \\
1 \\
\vdots \\
1 \\
2\\
2
\end{ytableau} &  \ytableausetup{boxsize=1.1em} 
\begin{ytableau}
1 & 1 & \cdots & 1 \\
1 & 2 & 2\\
\vdots \\
1
\end{ytableau} &  \ytableausetup{boxsize=1.1em} 
\begin{ytableau}
1 & 1 & \cdots & 1 \\
1 &2 \\
1 &2 \\
\vdots \\
1
\end{ytableau}   &  \ytableausetup{boxsize=1.1em} 
\begin{ytableau}
1 & 2\\
1 &2 \\
\vdots \\
1
\end{ytableau} & \ytableausetup{boxsize=1.1em} 
\begin{ytableau}
1 & 1 & \cdots &1  \\
2 & 2 
\end{ytableau}  \\
& & &&& \\
\text{ht}(T) = b & \text{ht}(T) = b-2+1 & \text{ht}(T) = b + 1 & \text{ht}(T) = b +2 + 1 & \text{ht}(T) = n-3 + 1 & \text{ht}(T) = 0 
\end{array}
\]
If $\lambda = (a,1^b)$ with $2\leqslant b \leqslant n-3$, then cases $\lambda^{(1)}$ and $\lambda^{(2)}$ both apply. Since $\text{ht}(\lambda^{(1)})$ 
and $\text{ht}(\lambda^{(2)})$ have opposite parity, $\chi^{(a,1^b)}((n-2,2)) = 0$. 
For $\lambda = (2, 1^{n-2})$, only $\lambda^{(2)}$ applies and $\chi^{(2,1^{n-2})}((n-2,2)) = (-1)^{n-3}$. For $\lambda = (1^n)$, only $\lambda^{(2)}$ 
applies and $\chi^{(1^{n})}((n-2,2)) = (-1)^{n-1}$. Since $(1^4)$ is contained in both $\lambda = (2,1^{n-2}), (1^n)$,  $c_\lambda$ also includes 
$f^{\lambda /(1^4)} \geqslant 1$. For $\lambda = (n-1,1)$, only $\lambda^{(1)}$ applies and $\chi^{(n-1,1)}((n-2,2)) = -1$. In this case, 
$(4) \subseteq \lambda$, and thus $c_\lambda$ also includes $f^{\lambda /(4)} \geqslant 1$. For $\lambda = (n)$, the height of any ribbon 
tableaux is $0$, so there are no negatives to worry about. 

If $\lambda$ is of the form $\lambda^{(3)}$, $\lambda^{(4)}$, or $\lambda^{(5)}$, then it is possible that $c_{\lambda^{(i)}}$ includes $-1$ from the 
unique $T \in \mathsf{R}(\lambda^{(i)}, (n-2,2))$ for $i = 3,4,5$. In any of these disjoint cases, $(2,2) \subseteq \lambda^{(i)}$, which means 
$c_{\lambda^{(i)}}$ also includes $f^{\lambda^{(i)} / (2,2)} \geqslant 1$. 
Hence $f^{\lambda / (1^4)}$, $f^{\lambda / (4)}$, and $f^{\lambda / (2,2)}$ from (3) and $\chi^\lambda((n-2,2))$ from (4) satisfy $f^{\lambda / (1^4)} + f^{\lambda / (4)} + f^{\lambda / (2,2)} + \chi^\lambda((n-2,2)) \geqslant 0$. 
 We have shown that $p_{(1^n)} + p_{(2,1^{n-2})} + p_{(2,2,1^{n-4})} + p_{(n-2,2)}$ is Schur-positive. 
\end{proof}

\begin{conjecture}
        \label{conjecture.interval2}
        For $n \geqslant 9$, the interval $A_{(n-2, 1,1)} = \{ \lambda \, |\, (1^n) \leqslant_I \lambda \leqslant_I (n-2, 1,1)\}$ is exactly
        \[
        (1^n) \lessdot_I (2, 1^{n-2}) \lessdot_I (2, 2, 1^{n-4}) \lessdot_I (3, 1^{n-3}) \lessdot_I (n-2, 1,1).
        \]
\end{conjecture}

\begin{remark} 
        The first two covers are consequences of Corollary~\ref{corollary.cover two col immersion0}. By Proposition~\ref{two column injection 0},
        \[
        		\varphi_0 \colon \mathsf{SSYT}((2,2,1^{n-4}), \nu) \rightarrow \mathsf{SSYT}((3,1^{n-3}), \nu)
	\]
	is an injection (with $\alpha = 0, \beta = 2$). Since $(2,2,1^{n-4}) <_D (3,1^{n-3})$, this implies $(2,2,1^{n-4}) \lessdot_I (3,1^{n-3})$.   
       By Corollary~\ref{corollary.hook transpose}, we know $(3, 1^{n-3}) <_I (n-2, 1,1)$ because $(3, 1^{n-3}) <_D (3, 1^{n-3})^t = (n-2, 1,1)$. This implies
        \[
        \{(1^n), (2, 1^{n-2}), (2, 2, 1^{n-4}), (3, 1^{n-3}), (n-2, 1,1) \} \subseteq A_{(n-2, 1, 1)}.
        \]
        However, we have not proven the cover relation $ (3, 1^{n-3}) \lessdot_I (n-2, 1,1)$. To show that $A_{(n-2, 1, 1)}$ is 
        contained in the above set, one could show that for all partitions $\lambda$ such that $\lambda$ and $\lambda^t$ are not included in the above list, 
        we have $f^\lambda > \frac{(n-1)(n-2)}{2} = f^{(n-2,1,1)}$. This would prove that $\lambda \not\in A_{(n-2,1,1)}$. We have confirmed the conjecture 
        up to $n=18$. 
\end{remark}

\begin{proposition}
\label{prop.sposinterval2} 
\mbox{}
\begin{enumerate}
\item For $n<9$, $p_{A_{(n-2,1,1)}}$ is Schur-positive.
\item For $n\geqslant 9$, $p_{(1^n)} + p_{(2,1^{n-2})} +  p_{(2,2,1^{n-4})} +  p_{(3,1^{n-3})} + p_{(n-2,1,1)}$ is Schur-positive.
\end{enumerate}
\end{proposition}

\begin{proof}
Part (1) can be checked explicitly using {\sc SageMath}. For part (2), 
as shown in the proof of Proposition~\ref{prop.sposinterval1}, $p_{(1^n)} + p_{(2,1^{n-2})} +  p_{(2,2,1^{n-4})} $ is Schur-positive. We next show 
$p_{(1^n)} + p_{(2,1^{n-2})} +  p_{(2,2,1^{n-4})} +  p_{(3,1^{n-3})}$ is Schur-positive. For $T \in \mathsf{R}(\lambda, (3, 1^{n-3}))$, there are three 
possible ways of arranging one $3$-ribbon in $T$. 
\[
\begin{array}{ccc}
\ytableausetup{boxsize=1.1em} 
\begin{ytableau}
1 & 1 &1  \\ 
\end{ytableau} & \ytableausetup{boxsize=1.1em} 
\begin{ytableau}
1 & 1  \\ 
1
\end{ytableau} & \ytableausetup{boxsize=1.1em} 
\begin{ytableau}
1   \\ 
1\\
1
\end{ytableau} \\
& &  \\
\text{ht}(T) = 0  & \text{ht}(T) =  1  & \text{ht}(T) = 2
\end{array} 
\]
Therefore, the coefficient of $s_\lambda$ in $p_{(3,1^{n-3})}$ is $f^{\lambda / (3)} - f^{\lambda / (2,1)} + f^{\lambda / (1,1,1)}$. 

If $(2,1) \subseteq \lambda$, then $c_\lambda$ includes $ - f^{\lambda / (2,1)}$.
Consider the bijection
\[
\mathsf{SYT}(\lambda / (2,1)) \rightarrow \{ T \in \mathsf{SYT}(\lambda) \mid T_{1,1} = 1, T_{1,2} = 2,  \text{ and } T_{2,1} = 3 \}.
\]
Note, the above subset of standard Young tableaux of shape $\lambda$ in the term $f^{\lambda}$ in $p_{(1^n)}$ was not used in the bijections in the 
proof of Proposition~\ref{prop.sposinterval1}. This shows that $p_{(1^n)} + p_{(2,1^{n-2})} +  p_{(2,2,1^{n-4})} +  p_{(3,1^{n-3})}$ is Schur-positive.

\smallskip

We now examine the Schur expansion of $p_{(n-2,1,1)}$. There are a few specific shapes $\lambda$ where $\mathsf{R}(\lambda,  (n-2,1,1))$ is nonempty.
Note that for a ribbon tableau $T$ of type $(n-2,1,1)$, 
$\text{ht}(T) = \text{ht}(R_1)$, where $R_1$ is the $(n-2)$-ribbon of $1$'s in $T$. In Case 1, we examine all hook shapes $\lambda = (a, 1^b)$. 
In Case 2, we examine all shapes $\lambda = (a, 2, 1^b)$. In Case 3, we examine the remaining two shapes $(a,3, 1^b)$ and $(a,2,2, 1^b)$.

\smallskip

\noindent \textbf{Case 1a:} $\lambda = (a,1^b)$ with $4 \leqslant a \leqslant n-3$ and $3 \leqslant b \leqslant n-4$. These conditions on $a,b$ require 
that the $(n-2)$-ribbon forms a hook with nontrivial arm and nontrivial leg.
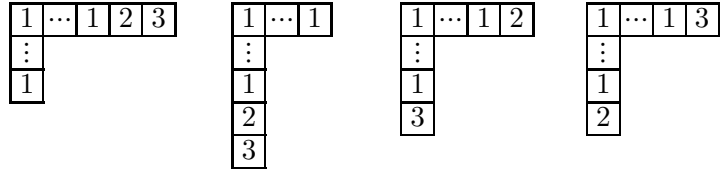
\[
\begin{array}{cccc}
 \ytableausetup{boxsize=1.1em} 
\begin{ytableau}
1 & \cdots & 1 & 2 & 3\\
\vdots \\
1
\end{ytableau} &  \ytableausetup{boxsize=1.1em} 
\begin{ytableau}
1 & \cdots & 1 \\
\vdots \\
1 \\
2\\
3
\end{ytableau} & \ytableausetup{boxsize=1.1em} 
\begin{ytableau}
1 & \cdots & 1 & 2 \\
\vdots \\
1 \\
3
\end{ytableau} & \ytableausetup{boxsize=1.1em} 
\begin{ytableau}
1 & \cdots & 1 & 3\\
\vdots \\
1\\
2
\end{ytableau}  \\
& & & \\
\text{ht}(T) = b & \text{ht}(T) = b-2 & \text{ht}(T) = b -1& \text{ht}(T) = b-1
\end{array}
\]
Since $b, b-2$ have opposite parity to $b-1$, $\chi^\lambda((n-2,1,1)) = 0$ for $\lambda = (a,1^b)$ with $4 \leqslant a \leqslant n-3$ and 
$3 \leqslant b \leqslant n-4$.

\smallskip

\noindent \textbf{Case 1b:} $\lambda = (3,1^{n-3})$. 
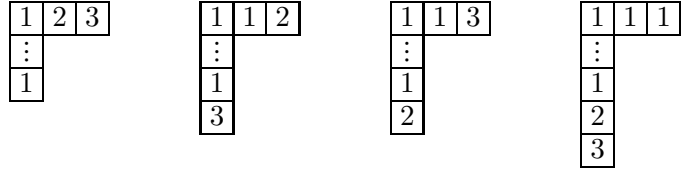
\[
\begin{array}{cccc}
 \ytableausetup{boxsize=1.1em} 
\begin{ytableau}
1 & 2 & 3\\
\vdots \\
1
\end{ytableau} &  \ytableausetup{boxsize=1.1em} 
\begin{ytableau}
1 & 1& 2\\
\vdots \\
1 \\
3
\end{ytableau} & \ytableausetup{boxsize=1.1em} 
\begin{ytableau}
1  & 1 & 3\\
\vdots \\
1 \\
2
\end{ytableau}  & \ytableausetup{boxsize=1.1em} 
\begin{ytableau}
1  & 1 & 1\\
\vdots \\
1 \\
2 \\
3
\end{ytableau}\\
& & & \\
\text{ht}(T) = n-3 & \text{ht}(T) = n-4 & \text{ht}(T) = n-4 & \text{ht}(T) = n-5 
\end{array}
\]
Since $n-3, n-5$ have opposite parity to $n-4$, $\chi^\lambda((n-2,1,1)) = 0$ for $\lambda = (3,1^{n-3})$.

\smallskip

\noindent \textbf{Case 1c:} $\lambda = (2,1^{n-2})$. 
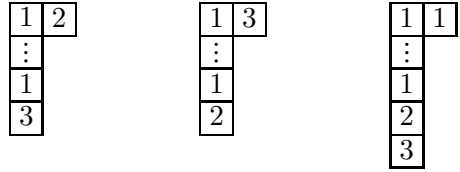
\[
\begin{array}{ccc}
 \ytableausetup{boxsize=1.1em} 
\begin{ytableau}
1 & 2\\
\vdots \\
1\\
3
\end{ytableau} &  \ytableausetup{boxsize=1.1em} 
\begin{ytableau}
1 & 3\\
\vdots \\
1 \\
2
\end{ytableau} & \ytableausetup{boxsize=1.1em} 
\begin{ytableau}
1  & 1 \\
\vdots \\
1 \\
2\\
3
\end{ytableau}  \\
& & \\
\text{ht}(T) = n-3 & \text{ht}(T) = n-3 & \text{ht}(T) = n-4 
\end{array}
\]
Since $n-3$ and $n-4$ have opposite parity, $\chi^\lambda((n-2,1,1)) = (-1)^{n-3}$ for $\lambda = (2,1^{n-2})$. Since $n\geqslant 5$, 
$(1^4) \subseteq \lambda$ and $f^{\lambda/(1^4)} \geqslant 1$, the coefficient of $s_\lambda$ in $p_{(2,2,1^{n-4})}$ will cancel this potential negative.

\smallskip

\noindent \textbf{Case 1d:} $\lambda = (1^{n})$. The unique ribbon tableau $T \in \mathsf{R}((1^n), (n-2,1,1))$ has height $n-3$, hence 
$\chi^\lambda((n-2,1,1)) = (-1)^{n-3}$ for $\lambda = (1^{n})$. Since $n\geqslant 5$, $(1^4) \subseteq \lambda$ and $f^{\lambda/(1^4)} \geqslant 1$, 
the coefficient of $s_\lambda$ in $p_{(2,2,1^{n-4})}$ will cancel this potential negative.

\smallskip

\noindent \textbf{Case 1e:} $\lambda = (n-2 ,1,1)$. 
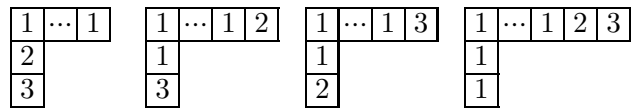
\[
\begin{array}{cccc}
 \ytableausetup{boxsize=1.1em} 
\begin{ytableau}
1 & \cdots & 1\\
2 \\
3
\end{ytableau} &  \ytableausetup{boxsize=1.1em} 
\begin{ytableau}
1 & \cdots & 1 & 2\\
1 \\
3
\end{ytableau} & \ytableausetup{boxsize=1.1em} 
\begin{ytableau}
1 & \cdots & 1 & 3\\
1 \\
2
\end{ytableau}  & \ytableausetup{boxsize=1.1em} 
\begin{ytableau}
1  & \cdots & 1 & 2 & 3\\
 1 \\
1
\end{ytableau}\\
& & & \\
\text{ht}(T) = 0 & \text{ht}(T) = 1 & \text{ht}(T) = 1 & \text{ht}(T) = 2
\end{array}
\]
Thus, $\chi^\lambda((n-2,1,1)) = 0 $ for $\lambda = (n-2, 1,1)$. \\

\noindent \textbf{Case 1f:} $\lambda = (n-1 ,1)$. 
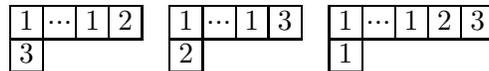
\[
\begin{array}{ccc}
 \ytableausetup{boxsize=1.1em} 
\begin{ytableau}
1 & \cdots & 1 & 2\\
3 \\
\end{ytableau} &  \ytableausetup{boxsize=1.1em} 
\begin{ytableau}
1 & \cdots & 1 & 3\\
2
\end{ytableau} & \ytableausetup{boxsize=1.1em} 
\begin{ytableau}
1 & \cdots & 1 & 2 & 3\\
1 
\end{ytableau}  \\
& & \\
\text{ht}(T) = 0 & \text{ht}(T) = 0 & \text{ht}(T) = 1 
\end{array}
\]
Thus, $\chi^\lambda((n-2,1,1)) = 1$ for $\lambda = (n-1, 1)$. 

\smallskip

\noindent \textbf{Case 1g:} $\lambda = (n)$. The unique ribbon tableau $T \in \mathsf{R}((n), (n-2,1,1))$ has height $0$, hence $\chi^\lambda((n-2,1,1)) = 1$ 
for $\lambda = (n)$.

\smallskip

\noindent \textbf{Case 2a:} $\lambda = (a, 2, 1^b)$ with $3 \leqslant a \leqslant n-3$ and $1 \leqslant b \leqslant n-5$. 
\[
\begin{array}{cccc}
 \ytableausetup{boxsize=1.1em} 
\begin{ytableau}
1 &  1& \cdots & 1 & 3\\
1 & 2 \\
\vdots \\
1
\end{ytableau} &  \ytableausetup{boxsize=1.1em} 
\begin{ytableau}
1 &  1& \cdots & 1 & 2\\
1 & 3 \\
\vdots \\
1
\end{ytableau} & \ytableausetup{boxsize=1.1em} 
\begin{ytableau}
1 &  1& \cdots & 1 \\
1 & 2 \\
\vdots \\
1 \\
3 
\end{ytableau}  & \ytableausetup{boxsize=1.1em} 
\begin{ytableau}
1 &  1& \cdots & 1 \\
1 & 3 \\
\vdots \\
1 \\
2 
\end{ytableau}  \\
& & & \\
\text{ht}(T) = b+1 & \text{ht}(T) = b+1 & \text{ht}(T) = b & \text{ht}(T) = b 
\end{array}
\]
Since $b$ and $b+1$ have opposite parity, $\chi^\lambda((n-2,1,1)) = 0$ for $\lambda = (a, 2, 1^b)$ with $3 \leqslant a \leqslant n-3$ and 
$1 \leqslant b \leqslant n-5$.

\smallskip

\noindent \textbf{Case 2b:} $\lambda = (n-2, 2)$. 
\[
\begin{array}{ccc}
 \ytableausetup{boxsize=1.1em} 
\begin{ytableau}
1 &  1& \cdots & 1\\
2 & 3 
\end{ytableau} &  \ytableausetup{boxsize=1.1em} 
\begin{ytableau}
1 &  1& \cdots & 1 & 3\\
1 & 2
\end{ytableau} & \ytableausetup{boxsize=1.1em} 
\begin{ytableau}
1 &  1& \cdots & 1 & 2\\
1 & 3
\end{ytableau}   \\
& & \\
\text{ht}(T) = 0 & \text{ht}(T) = 1 & \text{ht}(T) = 1 
\end{array}
\]
Thus, $\chi^\lambda((n-2,1,1)) = -1$ for $\lambda = (n-2, 2)$. Since $(2,2) \subseteq \lambda$ and $f^{\lambda/(2,2)} \geqslant 1$, the coefficient of 
$s_\lambda$ in $p_{(2,2,1^{n-4})}$ will cancel this negative.

\smallskip

\noindent \textbf{Case 2c:} $\lambda = (2, 2, 1^{n-4})$. 
\[
\begin{array}{ccc}
 \ytableausetup{boxsize=1.1em} 
\begin{ytableau}
1 &  2\\
1 & 3\\
\vdots \\
1 \\
1
\end{ytableau} &  \ytableausetup{boxsize=1.1em} 
\begin{ytableau}
1 &  1\\
1 & 3 \\
\vdots \\
1 \\
2
\end{ytableau} & \ytableausetup{boxsize=1.1em} 
\begin{ytableau}
1 &  1\\
1 & 2 \\
\vdots \\
1 \\
3
\end{ytableau}   \\
& & \\
\text{ht}(T) = n-3 & \text{ht}(T) = n-4 & \text{ht}(T) = n-4 
\end{array}
\]
Since $n-3$ and $n-4$ have opposite parity, $\chi^\lambda((n-2,1,1)) = (-1)^{n-4}$ for $\lambda = (2, 2, 1^{n-4})$. Since $(2,2) \subseteq \lambda$ and 
$f^{\lambda/(2,2)} \geqslant 1$, the coefficient of $s_\lambda$ in $p_{(2,2,1^{n-4})}$ will cancel this potential negative.

\smallskip

\noindent \textbf{Case 3a:} $\lambda = (n-(3+b), 3, 1^b)$ for $0\leqslant b \leqslant n-6$. The unique ribbon tableau 
$T \in \mathsf{R}((n-(3+b),3,1^b), (n-2,1,1))$ has height $b+1$, hence $\chi^\lambda((n-2,1,1)) = (-1)^{b+1}$ for $\lambda =  (n-(3+b), 3, 1^b)$. 
Since $(2,2) \subseteq \lambda$ and $f^{\lambda/(2,2)} \geqslant 1$, the coefficient of $s_\lambda$ in $p_{(2,2,1^{n-4})}$ will cancel this potential negative.

\smallskip

\noindent \textbf{Case 3b:} $\lambda = (n-(4+b), 2,2, 1^b)$ for $0\leqslant b \leqslant n-6$. The unique ribbon tableau 
$T \in \mathsf{R}((n-(4+b),2,2,1^b), (n-2,1,1))$ has height $b+2$, hence $\chi^\lambda((n-2,1,1)) = (-1)^{b+2}$ for $\lambda  = (n-(4+b), 2,2, 1^b)$. 
Since $(2,2) \subseteq \lambda$ and $f^{\lambda/(2,2)} \geqslant 1$, the coefficient of $s_\lambda$ in $p_{(2,2,1^{n-4})}$ will cancel this 
potential negative. 
\end{proof}

\section{Discussion}
\label{section.discussion}

In this paper, we studied various properties of the immersion and standard immersion poset, which are tightly linked to finite-dimensional 
irreducible polynomial representations of $GL_N(\CC)$ through their immersion pairs.

There are still many open questions to pursue in this line of research. It would be interesting to characterize all maximal elements in the immersion
and standard immersion poset. In particular, a proof of Conjecture~\ref{conjecture.maximal} seems in reach with the methods developed in this paper. 
In Corollary~\ref{corollary.hook transpose}, we showed that for hook shapes $\lambda$ and $\lambda^t$ form an immersion pair. The same
seems true for two column partitions. It would be interesting to classify when $\lambda$ and its transpose form an immersion pair.
In Corollary~\ref{corr.rank}, we showed that the rank of the immersion poset is at least $\lfloor n/2 \rfloor$. It would be desirable
to find better bounds for the rank.

Furthermore, it would be interesting to classify all intervals and chains in the immersion poset, in particular to obtain proofs of
Conjectures~\ref{conjecture.interval1} and~\ref{conjecture.interval2}. In view of the results of Section~\ref{section.lowerintervals}, the following
question is natural.
\begin{question}
    Which intervals $A_\mu := \{\lambda \,|\, (1^n) \leqslant_I \lambda \leqslant_I \mu\}$ in the immersion poset give rise to Schur-positivity of 
    $p_{A_\mu}$?
\end{question}

Sundaram conjectured that all intervals $[(1^n), \mu]$ in reverse lexicographic order make $\eqref{char}$ 
Schur-positive~\cite[Conjecture 1]{MR3855421}, and has proven the conjecture for certain intervals ~\cite{conjecture}. 
When $n \geqslant 5$, it appears that the immersion poset always contains some interval(s) which do not give rise to Schur-positivity. 
For example, $p_{A_{(n-1,1)}} = p_{(1^n)} + p_{(2,1^{n-2})} + p_{(n-1,1)}$ contains $-s_{(1^n)}$ when $n$ is odd. This observation shows that 
the analog of~\cite[Conjecture 1]{MR3855421} is false for the immersion poset order. However, it does seem true that a large percentage 
of intervals $A_\mu$ in the immersion poset yield Schur-positivity. Using  {\sc SageMath}~\cite{sagemath}, we observe that when 
$6 \leqslant n \leqslant 9$ at least 91\% of the intervals in the immersion poset make 
$\eqref{char}$ Schur-positive. When $n = 10,11$ the percentage of Schur-positive intervals drops to at least 81\%, and when $n=18$, the percentage 
is approximately 73.5\%. 

We conclude with some probabilistic and asymptotic questions.
\begin{question}
For randomly chosen partition $\lambda <_D \mu$, what is the probability that  $\lambda \leqslant_I \mu$? 
\end{question} 
Based on computer evidence, we conjecture that the probability is near $0.5$.
For a partition $\lambda$ of any size, consider the padded partition $\lambda[N]:=(N-|\lambda|,\lambda_1,\lambda_2,\ldots)$ of size $N$, 
where $N\geqslant |\lambda|$. For any two partitions $\lambda\leqslant_D\mu$ (of any size), what can we say about $\lambda[N] \leqslant_I \mu[N]$ 
for $N\gg 1$? Furthermore, it would be interesting to study the asymptotical behaviors of the (standard) immersion poset.

\bibliography{ref}
\bibliographystyle{alpha}

\end{document}